\documentclass[11pt]{amsart}  
\usepackage[hypertex]{hyperref}
\usepackage{amssymb,amscd} 
\usepackage[margin=2.5cm]{geometry}
\usepackage[all]{xy} 
\usepackage{psfrag} 
\usepackage{url}
\usepackage{latexsym} 
\usepackage{mathrsfs}
\usepackage{times}
\newcommand{\zz}{\mathbb{Z}}
\newcommand{\qq}{\mathbb{Q}}
\newcommand{\rr}{\mathbb{R}} 
\newcommand{\cc}{\mathbb{C}} 
\newcommand{\pp}{\mathbb{P}} 
\newcommand{\nn}{\mathbb{N}} 

\newcommand{\bun}{\mathbf{1}}
\newcommand{\bc}{\mathbf{c}}
\newcommand{\bt}{\mathbf{t}}
\newcommand{\be}{\mathbf{e}}
\newcommand{\bfs}{\mathbf{s}} 

\newcommand{\bI}{\mathbf{I}} 
\newcommand{\bv}{\mathbf{v}} 
\newcommand{\bF}{\mathbf{F}}
\newcommand{\bG}{\mathbf{G}} 

\newcommand{\cH}{\mathcal{H}}
\newcommand{\cL}{\mathcal{L}} 
\newcommand{\cF}{\mathcal{F}} 
\newcommand{\cO}{\mathcal{O}}
\newcommand{\cU}{\mathcal{U}}
\newcommand{\cV}{\mathcal{V}} 

\newcommand{\frE}{\mathfrak{E}}
\newcommand{\frM}{\mathfrak{M}}

\newcommand{\ttau}{\tilde{\tau}} 
\newcommand{\tT}{\widetilde{T}}
\newcommand{\hGamma}{\widehat{\Gamma}}
\newcommand{\chK}{\check{K}}
\newcommand{\chnabla}{\check{\nabla}}

\DeclareMathOperator{\ev}{ev}
\DeclareMathOperator{\Lap}{Lap}
\DeclareMathOperator{\QS}{QS}
\DeclareMathOperator{\grav}{grav}
\DeclareMathOperator{\eff}{Eff}
\DeclareMathOperator{\lrl}{lrl}
\DeclareMathOperator{\e}{e}%
\DeclareMathOperator{\eqeuler}{\mathbf{e}_{\lambda}}%
\newcommand{\eqeulerinv}{\mathbf{e}^{-1}_{\lambda}}%
\newcommand{\eqeulerstar}{\mathbf{e}^{*}_{\lambda}}%
\DeclareMathOperator{\ctop}{\mathbf{e}}
\DeclareMathOperator{\pt}{pt}%
\DeclareMathOperator{\Td}{Td}%
\DeclareMathOperator{\Ch}{Ch}%
\DeclareMathOperator{\ch}{ch}%
\DeclareMathOperator{\Mir}{\mathfrak{m}}%
\DeclareMathOperator{\Gr}{gr}%
\DeclareMathOperator{\amb}{amb}%
\DeclareMathOperator{\QDM}{QDM}%
\DeclareMathOperator{\SQDM}{SQDM}%
\DeclareMathOperator{\Tor}{Tor}
\DeclareMathOperator{\rank}{rk}%
\DeclareMathOperator{\Res}{Res}%

\DeclareMathOperator{\Ext}{Ext}%
\DeclareMathOperator{\im}{Im}
\DeclareMathOperator{\Id}{id}%
\DeclareMathOperator{\Spec}{Spec}
\DeclareMathOperator{\vir}{vir}%
\newcommand{\End}{\operatorname{End}} 
\newcommand{\Spf}{\operatorname{Spf}} 
\newcommand{\Ker}{\operatorname{Ker}}

\newcommand{\cEd}{(\mathbf{c}^{*},E^{\vee})}
\newcommand{\cE}{(\mathbf{c},E)}

\newcommand{\eqeEi}{(\eqeulerinv,E^{\vee})}
\newcommand{\eqeE}{(\eqeuler,E)}

\newcommand{\eE}{(\e,E)}
\newcommand{\eK}{(\e,K_{X}^{-1})}

\newcommand{\eqeEprod}[1]{\bullet_{#1}^{\eqeE}}
\newcommand{\eEprod}[1]{\bullet_{#1}^{\eE}}

\newcommand{\eu}{{\text{\rm eu}}}
\newcommand{\loc}{{\text{\rm loc}}}

\newcommand{\llangle}{\left\langle}
\newcommand{\rrangle}{\right\rangle}
\newcommand{\<}{\left\langle\!\left\langle}
\renewcommand{\>}{\right\rangle\!\right\rangle}
\def\bigcorr#1{\left\langle\!\!\left\langle #1\right\rangle\!\!\right\rangle}
\def\Bigcorr#1{\left\langle\!\!\!\left\langle #1 \right\rangle\!\!\!\right\rangle}

\numberwithin{equation}{section} 

\newtheorem{thm}{Theorem}[section]
\newtheorem{cor}[thm]{Corollary}
\newtheorem{lem}[thm]{Lemma}
\newtheorem{prop}[thm]{Proposition}
 
\theoremstyle{definition} 
\newtheorem{defi}[thm]{Definition}
\newtheorem{rem}[thm]{Remark}
\newtheorem{recall}[thm]{Recall}%
\newtheorem{notn}[thm]{Notation}%

\begin{document}

\title{Quantum Serre theorem as a duality between quantum $D$-modules}

\author{Hiroshi Iritani}
\address{Hiroshi Iritani, Department of Mathematics, Kyoto University, 
Kitashirakawa-Oiwake-cho, Sakyo-ku, Kyoto, 606-8502, Japan.}
\email{iritani@math.kyoto-u.ac.jp}
\author{Etienne Mann}
\address{Universit\'e Montpellier 2, Institut de Math\'ematique et
    de Mod\'elisation de Montpellier, UMR 5149, Case courier 051
Place Eug\`ene Bataillon F-34 095 Montpellier CEDEX 5, France.}
\email{etienne.mann@math.univ-montp2.fr}
\urladdr{http://www.math.univ-montp2.fr/~mann/index.html}

 \author{Thierry Mignon}
 \address{Universit\'e Montpellier 2, Institut de Math\'ematique et
     de Mod\'elisation de Montpellier, UMR 5149, Case courier 051
 Place Eug\`ene Bataillon
 F-34 095 Montpellier CEDEX 5, France.}
 \email{thierry.mignon@math.univ-montp2.fr}
 \urladdr{http://www.math.univ-montp2.fr/~mignon/}

\subjclass[2010]{14N35, 53D45, 14F10} 
\keywords{quantum cohomology, Gromov-Witten invariants, 
quantum differential equation, Fourier-Laplace transformation, 
quantum Serre, second structure connection, Frobenius manifold}

\begin{abstract} 
We give an interpretation of quantum Serre theorem of Coates and Givental 
as a duality of twisted quantum $D$-modules. 
This interpretation admits a non-equivariant limit, and 
we obtain a precise relationship among 
(1) the quantum $D$-module of $X$ twisted by a convex vector bundle $E$ 
and the Euler class, 
(2) the quantum $D$-module of the 
total space of the dual bundle $E^\vee\to X$, 
and (3) the quantum $D$-module of a submanifold $Z\subset X$ 
cut out by a regular section of $E$. 

When $E$ is the anticanonical line bundle $K_X^{-1}$, 
we identify these twisted quantum $D$-modules 
with second structure connections with different parameters, 
which arise as Fourier-Laplace transforms of the quantum $D$-module 
of $X$. In this case, we show that the duality pairing is identified with 
Dubrovin's second metric (intersection form). 
\end{abstract}


\maketitle
\setcounter{tocdepth}{1}     
\tableofcontents

\section{Introduction} 

Genus-zero Gromov-Witten invariants of a smooth projective variety 
$X$ can be encoded in different mathematical objects: 
a generating function that satisfies some system of PDE 
(WDVV equations), an associative and commutative product 
called quantum product, the Lagrangian cone $\mathcal{L}_{X}$ 
of Givental \cite{Givental:symplectic} or in a meromorphic flat connection 
called quantum connection. 
These objects are all equivalent to each other; in this paper 
we focus on the realization of Gromov-Witten invariants 
as a meromorphic flat connection. 

Encoding Gromov-Witten invariants in a meromorphic flat connection 
defines the notion of quantum $D$-module \cite{Givental:ICM}, denoted by 
$\QDM(X)$, that is a tuple $(F,\nabla,S)$ consisting of 
a trivial holomorphic vector bundle 
$F$ over $H^{\ev}(X)\times \cc_z$ with fiber $H^{\ev}(X)$, 
a meromorphic flat connection $\nabla$ on $F$ given by the 
quantum connection: 
\[
\nabla = d + \sum_{\alpha=0}^s (T_\alpha \bullet_\tau) dt^\alpha 
+ \left( - \frac{1}{z} (\frE\bullet) + \frac{\deg}{2} \right) \frac{dz}{z} 
\]
and a flat non-degenerate pairing $S$ on $F$ given by the Poicar\'{e} 
pairing (see Definition \ref{def:quantumD-mod} 
and Remark \ref{rem:connection_z}). 
These data may be viewed as a generalization of a variation of Hodge structure 
(see \cite{Katzarkov-Pantev-Kontsevich-ncVHS}). 

\emph{Quantum Serre theorem} of Coates and Givental 
\cite[\S 10]{Givental-Coates-2007-QRR} describes a certain relationship 
between \emph{twisted} Gromov-Witten invariants. 
The data of a twist is given by a pair $\cE$ of an invertible multiplicative 
characteristic class $\bc$ and a vector bundle $E$ over $X$. 
Since twisted Gromov-Witten invariants satisfy properties similar 
to usual Gromov-Witten invariants, we can define 
twisted quantum product, twisted quantum $D$-module 
$\QDM_{\cE}(X)$ and twisted Lagrangian cone $\cL^{\cE}$ 
associated to the twist $\cE$. 
Let $\bc^*$ denote the characteristic class  
satisfying $\bc(V)\bc^*(V^\vee)  =1$ for any vector bundle $V$. 
Quantum Serre theorem (at genus zero) gives the equality of the 
twisted Lagrangian cones: 
\begin{equation}
\label{eq:QS_Lagcone} 
\mathcal{L}^{\cEd} = \bc(E) \mathcal{L}^{\cE}. 
\end{equation}
Quantum Serre theorem of Coates and Givental was not 
stated as a duality. An observation in this paper is that this result 
can be restated as a duality between twisted quantum $D$-modules: 
\begin{thm}[see Theorem \ref{thm:quantum,Serre,QDM} 
for more precise statements] 
\label{thm,intro}
There exists a (typically non-linear) map 
$f\colon H^{\ev}(X) \to H^{\ev}(X)$ (see \eqref{eq:f,map}) 
such that the following holds: 
\begin{enumerate}
\item The twisted quantum $D$-modules 
$\QDM_{\cE}(X)$ and $f^{*}\QDM_{\cEd}(X)$ are dual to 
each other; the duality pairing $S^{\QS}$ is given by 
the Poincar\'{e} pairing. 
\item The map  $\QDM_{\cE}(X) \to f^{*}\QDM_{\cEd}(X)$ 
sending $\alpha$ to $\bc (E)\cup\alpha$ 
is a morphism of quantum $D$-modules. 
  \end{enumerate}
\end{thm}

Genus-zero twisted Gromov-Witten invariants were originally designed 
to compute Gromov-Witten invariants for Calabi-Yau hypersurfaces 
or non-compact local Calabi-Yau manifolds \cite{Kontsevich:enumeration, 
Givental-Equivariant-GW, 
LocalMirrorSymmetry-Chiang-Klemm-Yau-Zaslow-99}. 
Suppose that $E$ is a convex vector bundle and $\bc$ is the 
equivariant Euler class $\eqeuler$. 
In this case, non-equivariant limits of $\cE$-twisted Gromov-Witten invariants 
yield Gromov-Witten invariants of a regular section $Z\subset X$ of $E$ 
and non-equivariant limits of $\cEd$-twisted 
Gromov-Witten invariants yield Gromov-Witten invariants 
for the total space $E^\vee$. 
The original statement \eqref{eq:QS_Lagcone} of quantum Serre theorem 
does not admit a non-equivariant limit since the non-equivariant 
Euler class is not invertible. We see however that our restatement 
above passes to the non-equivariant limit as follows: 

\begin{cor}[Theorem \ref{thm:euler,quantum,Serre,QDM}, 
Corollary \ref{cor:Z}] 
Let $E$ be a convex vector bundle and let $\e$ denote the 
(non-equivariant) Euler class. 
Let $h\colon H^{\ev}(X) \to H^{\ev}(X)$ be the map given 
by $h(\tau) = \tau + \pi\sqrt{-1} c_1(E)$ 
and let $\overline{f} \colon H^{\ev}(X) \to H^{\ev}(X)$ 
denote the non-equivariant limit of the map $f$ 
of Theorem \ref{thm,intro} in the case where $\bc = \eqeuler$. 
We have the following: 
\begin{enumerate}
\item The quantum $D$-modules 
$\QDM_{\eE}(X)$ and 
$(h\circ \overline{f})^{*}\QDM(E^{\vee})$ 
are dual to each other. 

\item Let $Z$ be the zero-locus of a regular section of $E$ 
and suppose that $Z$ satisfies one of the conditions in 
Lemma \ref{lem:cond_Z}. 
Denote by $\iota:Z\hookrightarrow X$ the inclusion.  
Then the morphism $e(E) \colon \QDM_{\eE}(X) 
\to (h\circ \overline{f})^* \QDM(E^\vee)$ factors through 
the ambient part quantum 
$D$-module $\QDM_{\amb}(Z)$ of $Z$ as:   
\begin{equation}
\label{eq:factor_diagram}
\begin{aligned} 
    \xymatrix{\QDM_{\eE}(X) \ar@{->>}[d]_{\iota^*} 
\ar[r]^-{\e(E)\cup} 
   & (h\circ \overline{f})^{*}\QDM(E^{\vee})\\ 
(\iota^*)^*\QDM_{\amb}(Z) \ar@{^{(}->}^-{\iota_*}[ru]}
\end{aligned} 
\end{equation}  
\end{enumerate} 
\end{cor}
\noindent 
What is non-trivial here is the existence of an embedding of 
$\QDM_{\amb}(Z)$ into $\QDM(E^\vee)$. This is reminiscent 
of the Kn\"{o}rrer periodicity \cite{Knorrer:CM, 
Orlov:equivalence_between_LG}: we expect 
that this would be a special case of a more general phenomenon 
which relates quantum cohomology of a non-compact space equipped 
with a holomorphic function $W$ to quantum cohomology 
of the critical locus of $W$. 

In \S\ref{sec:quant-serre-integr}, we introduce certain integral structures 
for the quantum $D$-modules $\QDM_{\eE}(X)$, $\QDM(E^\vee)$ 
and $\QDM_{\amb}(Z)$, generalizing the construction in 
\cite{Iritani-2009-Integral-structure-QH,Katzarkov-Pantev-Kontsevich-ncVHS}. 
These integral structures are lattices in the space of flat sections 
which are isomorphic to the $K$-group $K(X)$ of vector bundles. 
We show in Propositions \ref{prop:compa,integral}, 
\ref{prop:integral} that the duality pairing $S^{\QS}$ 
is identified with the Euler pairing on the $K$-groups, 
and that the maps appearing in the diagram \eqref{eq:factor_diagram} 
are induced by natural functorial maps between $K$-groups. 

In \S\ref{sec:QS,KX,abstract,Fourier}, we consider the case 
where $E=K_{X}^{-1}$ and study quantum 
cohomology of a Calabi-Yau hypersurface in $X$ and 
the total space of $K_X$. 
We show that the small\footnote{Small quantum $D$-modules 
are the restriction of quantum $D$-modules 
to the $H^{2}(X)$ parameter space.} quantum $D$-modules 
$\SQDM_{\eK}(X)$ and $\SQDM(K_{X})$ are isomorphic to the 
second structure connections of Dubrovin \cite{Dtft}. 
The second structure connections are meromorphic flat connections 
$\chnabla^{(\sigma)}$ 
on the trivial vector bundle $\check{F}$ over $H^{\ev}(X) \times \cc_x$ 
with fiber $H^{\ev}(X)$, which is obtained from the quantum connection of $X$ 
via the \emph{Fourier-Laplace transformation} with respect to $z^{-1}$ 
(see \eqref{eq:second_str_conn}): 
\[
\partial_{z^{-1}} \leadsto x, \qquad 
z^{-1} \leadsto - \partial_x.  
\]
The second structure connection has a complex parameter $\sigma$; 
we will see that the two small quantum $D$-modules correspond 
to different values of $\sigma$.  

\begin{thm}[see Theorems \ref{thm:QS,with,I,Coates} 
and \ref{thm:pairing,vs,quantum,serre} for more precise statements] 
Suppose that the anticanonical class $-K_X$ of $X$ is nef. 
Let $n$ be the dimension of $X$. 
\begin{enumerate}
\item There exist maps 
$\pi_{\eu}, \pi_\loc \colon H^2(X) \times \cc_x \to H^2(X)$ 
and isomorphisms of vector bundles with connections: 
  \begin{align*}
    \psi_{\eu}: \left(\check{F}, \chnabla^{(\frac{n+1}{2})}\right) 
\Bigr|_{H^2(X)\times \cc_x}
&\longrightarrow \pi_{\eu}^{*}\SQDM_{\eK}(X) \Bigr|_{z=1}\\
\psi_{\loc}:\left(\check{F}, \chnabla^{\left(-\frac{n+1}{2}\right)}\right) 
\Bigr|_{H^2(X)\times \cc_x} 
 &\longrightarrow 
\pi_{\loc}^{*}\SQDM(K_{X}) \Bigr|_{z=1}
\end{align*} 
which are defined in a neighbourhood of the large radius limit point 
and for sufficiently large $|x|$. 

\item The duality pairing $S^{\QS}$ between 
$\pi_\eu^* \SQDM_{\eK}(X)$ and 
$\pi_\loc^* \SQDM(K_X)$ is identified with the 
second metric $\check{g} \colon 
\big(\cO(\check{F}),\chnabla^{(\frac{n+1}{2})}\big) 
\times \big(\cO(\check{F}),\chnabla^{(-\frac{n+1}{2})}\big) 
\to \cO$ given by 
\[
\check{g}(\gamma_1,\gamma_2) = 
\int_X \gamma_1 \cup (c_1(X) \bullet_\tau - x)^{-1} \gamma_2 
\] 
over $H^2(X)\times \cc_x$. 
\end{enumerate}
\end{thm}
Combined with the commutative diagram \eqref{eq:factor_diagram}, 
this theorem gives an entirely algebraic description of the ambient 
part quantum $D$-module of a Calabi-Yau hypersurface $Z$ in 
a Fano manifold $X$ (Corollary \ref{cor:anticanonical_hypersurface}). 
We will also describe the A-model Hodge filtration for these small 
quantum $D$-modules in terms of the second structure connection 
in \S \ref{subsec:Hodge_filt}. These results are illustrated for a 
quintic threefold in $\pp^4$ in \S \ref{sec:quintic-pp4}. 
Note that the second structure connection is closely related 
to the \emph{almost dual} Frobenius manifold of Dubrovin 
(see \cite[Proposition 3.3]{Dubrovin-Almost-Frob-2004}) 
and our result may be viewed as a generalization of 
the example in \cite[\S 5.4]{Dubrovin-Almost-Frob-2004}.

This paper arose out of our previous works 
\cite{iritani_quantum_2011,Mann-Mignon-2011-QDM-lci-nef} 
on quantum $D$-modules of (toric) complete intersections. 
The embedding of $\QDM_{\amb}(Z)$ into $\QDM(E^\vee)$ 
appeared in \cite[Remark 6.14]{iritani_quantum_2011} 
in the case where $X$ is a weak Fano toric orbifold and $E^\vee=K_X$; 
in \cite[Theorem 1.1]{Mann-Mignon-2011-QDM-lci-nef}, 
$\QDM_{\amb}(Z)$ was presented as the quotient 
$\QDM_{\eE}(X)$ by $\Ker(\e(E)\cup)$ when $E$ is a direct 
sum of ample line bundles. 
We would also like to draw attention to a recent work of 
Borisov-Horja \cite{Borisov-Horja:applications} on the duality of 
better behaved GKZ systems. The conjectural duality in their work 
should correspond to a certain form of quantum Serre duality 
generalized to toric Deligne-Mumford stacks. 

We assume that the reader is familiar with Givental's formalism 
and quantum cohomology Frobenius manifold. 
As preliminary reading for the reader, 
we list \cite{Givental:symplectic}, 
\cite{Givental-Coates-2007-QRR} 
and \cite[Chapter I,II]{Mfm}.

\begin{notn}
We use the following notation throughout the paper. 
\begin{center} 
\begin{tabular}{ll}
 $X$ & a smooth projective variety of complex dimension $n$. \\ 
 $E$ & a vector bundle over $X$ of rank $r$ with 
$E^{\vee}$ the dual vector bundle. \\ 
$(T_{0}, \ldots ,T_{s})$ & an homogeneous basis of 
$H^{\ev}(X) = \bigoplus_{p=0}^n H^{2p}(X,\cc)$ 
such that  \\ 
& $T_{0}=\bf{1}$ and $\{T_1,\dots,T_r\}$ 
form a nef integral basis of $H^2(X)$. \\
$(t^0,\dots,t^s)$ 
& the linear coordinates dual to the basis $(T_0,\dots,T_s)$; \\
& we write $\tau:=\sum_{\alpha=0}^{s}t^{\alpha}T_{\alpha}$ 
and $\partial_{\alpha}:=(\partial/\partial t^\alpha)$. \\ 
$(T^0,\dots,T^s)$ & the Poincar\'e dual basis 
such that $\int_X T_\alpha \cup T^\beta  =\delta_\alpha^\beta$. \\  
$\eff(X)$ & the set of classes in $H_{2}(X,\zz)$ represented by effective curves. \\
$\gamma(d)$ &  the pairing $\int_d \gamma$ between 
$\gamma\in H^{2}(X)$ and $d\in H_{2}(X)$. \\ 
$\eqeuler$  & the equivariant Euler class. \\ 
 $\e$ & the non-equivariant Euler class. 
\end{tabular} 
\end{center} 
\end{notn} 

\medskip
\noindent 
{\bf Acknowledgments.} 
H.I.~thanks John Alexander Cruz Morales and Anton Mellit 
for helpful discussions. 
We thank referees for suggesting several improvements 
on the statement and the proof of Lemma \ref{lem:properness}. 
This work is supported by the grants ANR-13-IS01-0001-01 and 
ANR-09-JCJC-0104-01 of the Agence nationale de la recherche 
and JSPS KAKENHI Grant Number 25400069, 23224002, 
26610008.

\section{Quantum Serre theorem as a duality}
\label{sec:Quantum-Serre}

In this section, we reformulate Quantum Serre theorem of 
Coates-Givental \cite{Givental-Coates-2007-QRR} 
as a duality of quantum $D$-modules. 
After reviewing twisted Gromov-Witten invariants and 
twisted quantum $D$-modules, we give our reformulation   
in Theorem \ref{thm:quantum,Serre,QDM}. 

\subsection{Notation} 
We introduce the notation we use throughout the paper. 
Let $\{T_0,\dots,T_s\}$ be a homogeneous basis of 
the cohomology group 
$H^{\ev}(X) = \bigoplus_{p=0}^n H^{2p}(X;\cc)$ 
of even degree. We assume $T_0 = \bun$ and 
$T_1,\dots,T_r$ form a nef integral basis of $H^2(X)$ 
for $r = \dim H^2(X) \le s$.  
Let $\{t^0,\dots,t^s\}$ denote the linear co-ordinates 
on $H^{\ev}(X)$ dual to the basis 
$\{T_0,\dots,T_s\}$ and write 
$\tau = \sum_{\alpha=0}^s t^\alpha T_\alpha$ for a 
general point of $H^{\ev}(X)$.  
We write $\partial_\alpha = \partial/\partial t^\alpha$ for 
the partial derivative. 

Let $\eff(X)$ denote the set of classes of effective curves 
in $H_2(X;\zz)$. 
Let $K$ be a commutative ring. For $d\in \eff(X)$, 
we write $Q^d$ for the corresponding element 
in the group ring $K[Q] := K[\eff(X)]$. 
The variable $Q$ is called the Novikov variable. 
We write $K[\![Q]\!]$ for the natural completion of $K[Q]$. 
For an infinite set $\bfs = \{s_0,s_1,s_2,\dots\}$ of variables, 
we define the formal power series ring 
\[
K[\![ \bfs]\!] = K[\![s_0,s_1,s_2,\dots]\!] 
\]
to be the (maximal) completion of $K[s_0,s_1,s_2,\dots]$ with 
respect to the additive valuation $v$ defined by 
$v(s_k) = k+1$. 
We write 
\[
\cc[\![Q,\tau]\!], \qquad 
\cc[\![Q,\bfs,\tau]\!] \quad \text{and} \quad
\cc[z][\![Q,\bfs,\tau]\!] 
\] 
for the completions of 
$\cc[\![Q]\!][t^0,\dots,t^s]$, 
$\cc[\![Q]\!][t^0,\dots,t^s,s_0,s_1,s_2,\dots]$ 
and $\cc[z][\![Q]\!][t^0,\dots,t^s,s_0,s_1,s_2,\dots]$ 
respectively. 
We write $\tau_2 = \sum_{i=1}^r t^i T_i$ 
for the $H^2(X)$-component of $\tau$ 
and set $\tau = \tau_2 + \tau'$. 
Because of the divisor equation in Gromov-Witten theory, 
the Novikov variable $Q$ and $\tau_2$ often appear 
in the combination $(Q e^{\tau_2})^d 
= Q^d e^{t_1 T_1(d) + \cdots + t_r T_r(d)}$. 
Therefore we can also work with the subring 
\[
\cc[\![Q e^{\tau_2},\tau']\!] 
= \cc[\![Q  e^{\tau_2}]\!][\![t^0,t^{r+1},\dots,t^s]\!] 
\subset \cc[\![Q,\tau]\!].  
\] 
The subrings 
$\cc[\![Qe^{\tau_2},\bfs,\tau']\!]\subset \cc[\![Q,\bfs,\tau]\!]$, 
$\cc[z][\![Qe^{\tau_2},\bfs, \tau']\!]\subset \cc[z][\![Q,\bfs,\tau]\!]$ 
are defined similarly. 

\subsection{Twisted quantum $D$-modules}
\label{subsec:tw,QDM}

Coates-Givental \cite{Givental-Coates-2007-QRR} 
introduced Gromov-Witten invariants 
twisted by a vector bundle and a multiplicative 
characteristic class. We consider the quantum 
$D$-module defined by genus-zero twisted 
Gromov-Witten invariants.  

\subsubsection{Twisted Gromov-Witten invariants and 
twisted quantum product}
\label{subsubsec:twist-quant-prod}

Let $X$ be a smooth projective variety and let $E$ be a 
vector bundle on $X$.  
Denote by $\eff(X)$ the subset of $H_{ 2}(X,\zz)$ 
of classes of effective curves.
For $d\in \eff(X)$ and $g,\ell \in \nn$, we denote by
$\overline{\mathcal{M}}_{g,\ell}(X,d)$ the moduli space of 
genus $g$ stable maps to $X$ of degree $d$ and with $\ell$ marked points. 
Recall that $\overline{\mathcal{M}}_{g,\ell}(X,d)$ is a proper 
Deligne-Mumford stack and is 
equipped with a virtual fundamental class 
$[\overline{\mathcal{M}}_{g,\ell}(X,d)]^{\vir}$ 
in $H_{2D}(\overline{\mathcal{M}}_{g,\ell}(X,d), \qq)$ 
with $D = (1-g)(\dim X -3)+ (c_1(X) \cdot d) + \ell$. 
In this paper, we only consider the genus-zero moduli spaces. 
The universal curve of $\overline{\mathcal{M}}_{0,\ell}(X,d)$ 
is $\overline{\mathcal{M}}_{0,\ell+1}(X,d)$: 
\begin{displaymath}
\xymatrix{\overline{\mathcal{M}}_{0,\ell+1}(X,d)\ar[d]^-{\pi}
\ar[rr]^-{\ev_{\ell+1}}&&X \\ \overline{\mathcal{M}}_{0,\ell}(X,d)}
\end{displaymath}
where $\pi$ is the map that  forgets the $(\ell+1)$-th marked point and
stabilizes, and $\ev_{\ell+1}$ is the evaluation map at the $(\ell+1)$-th
marked point.

The vector bundle $E$ defines a $K$-class 
$E_{0,\ell,d}:=\pi_{!}e_{\ell+1}^{*}E \in
K^{0}(\overline{\mathcal{M}}_{0,\ell}(X,d))$ 
on the moduli space, 
where $\pi_{!}$ denotes the $K$-theoretic push-forward. 
The restriction to a point $(f:C\to X) \in \overline{\mathcal{M}}_{0,\ell}(X,d)$ 
gives: 
\begin{align*} 
  E_{0,\ell,d}\mid_{(f:C\to X)}=[H^{0}(C,f^{*}E)]- [H^{1}(C, f^{*}E)].  
\end{align*}

For $i\in\{1,\ldots,\ell\}$, let $\mathfrak{L}_{i}$ denote the 
universal cotangent line bundle on $\overline{\mathcal{M}}_{0,\ell}(X,d)$ 
at the $i$-th marking.  
The fibre of $\mathfrak{L}_i$ at a point $(C,x_{1}, \ldots ,x_{\ell},$
$f\colon C\to X)$ is the cotangent space $T_{x_i}^*C$ at $x_i$. 
Put $\psi_i:=c_1(\mathfrak{L}_i)$ in 
$H^2(\overline{\mathcal{M}}_{0,\ell}(X,d),\qq)$.

The universal invertible multiplicative 
characteristic class $\bc(\cdot)$ is given by: 
\[
\bc(\cdot) = \exp\left (\sum_{k=0}^\infty s_k \ch_k(\cdot) \right)  
\]
with infinitely many parameters $s_0,s_1,s_2,\dots$. 
In the discussion of Coates-Givental's quantum Serre theorem, 
we treat $s_0,s_1,s_2,\dots$ as formal infinitesimal 
parameters. On the other hand, the result we obtain later 
sometimes makes sense for non-zero values of the parameters. 
For example, we will use the equivariant Euler class for $\bc(\cdot)$ 
in \S \ref{sec:non-equivariant,limit}. 

The \emph{genus-zero $\cE$-twisted Gromov-Witten invariants} are
defined by the following formula. 
For any $\gamma_{1}, \ldots,\gamma_{\ell}\in H^{\ev}(X)$ 
and any $k_{1}, \ldots ,k_{\ell}\in \nn$, we put:
\begin{displaymath}
\llangle \gamma_{1}\psi^{k_{1}}, \ldots ,\gamma_{\ell}
\psi^{k_{\ell}}\rrangle_{0,\ell,d}^{\cE}:=
\int_{[\overline{\mathcal{M}}_{0,\ell}(X,d)]^{\vir}}
\left(\prod_{i=1}^{\ell}\psi_{i}^{k_{i}}\ev_{i}^{*}\gamma_{i}\right) 
\bc(E_{0,\ell,d}). 
\end{displaymath}
We also use the following notation:
\begin{align*}
\bigcorr{
\gamma_{1}\psi^{k_{1}}, \ldots ,\gamma_{\ell}\psi^{k_{\ell}}
}^{\cE}_{\tau}
:=\sum_{d\in \eff(X)}\sum_{k\geq 0}\frac{Q^{d}}{k!}
\llangle \gamma_{1}\psi_{1}^{k_{1}}, \ldots ,
\gamma_{\ell}\psi_{\ell}^{k_{\ell}},\tau, \ldots ,\tau\rrangle_{0,\ell+k,d}^{\cE}
\end{align*} 
The genus-zero twisted Gromov-Witten potential is: 
\[
\cF_{\cE}^0(\tau) = \< \>^{\cE}_\tau 
= \sum_{d\in \eff(X)}\sum_{k\geq 0} 
\frac{Q^d}{k!} 
\llangle \tau,\dots,\tau\rrangle_{0,k,d}^{\cE}.  
\]
The genus-zero twisted potential $\cF_{\cE}^0(\tau)$ 
lies in $\cc[\![Q,\bfs,\tau]\!]$. By the divisor equation, we see 
that it lies in the subring $\cc[\![Q e^{\tau_2},\bfs,\tau']\!]$. 
Introduce the symmetric bilinear pairing $(\cdot,\cdot)_{\cE}$ 
on $H^{\ev}(X)\otimes \cc[\![\bfs]\!] $ by 
\begin{align*}
  (\gamma_{1},\gamma_{2})_{\cE}=\int_{X}
\gamma_{1}\cup\gamma_{2}\cup\bc(E).
\end{align*}
The $\cE$-twisted quantum product $\bullet_\tau^{\cE}$ 
is defined by the formula: 
\begin{equation} 
\label{eq:tw,quantum,prod} 
\left(T_\alpha \bullet_\tau^{\cE} T_\beta, T_\gamma\right)_{\cE} 
= \partial_\alpha \partial_\beta \partial_\gamma 
\cF_{\cE}^0 (\tau). 
\end{equation} 
The structure constants lie in $\cc[\![Qe^{\tau_2},\bfs,\tau']\!] 
\subset \cc[\![Q,\bfs,\tau]\!]$. 
The product is extended bilinearly over $\cc[\![Q,\bfs,\tau]\!]$ and 
defines the $\cE$-twisted quantum cohomology 
$(H^{\ev}(X)\otimes \cc[\![Q,\bfs,\tau]\!], \bullet_\tau^{\cE})$. 
It is associative and commutative, and has $T_0 = \bun$ as the identity.

\subsubsection{Twisted quantum $D$-module and fundamental solution} 
\label{subsubsec:twist-quant-diff}

\begin{defi} 
\label{def:quantumD-mod}
The \emph{$\cE$-twisted quantum $D$-module} is a triple 
\[
\QDM_{\cE}(X) = 
\left(H^{\ev}(X)\otimes \cc[z][\![Q,\bfs,\tau]\!], 
\nabla^{\cE}, S_{\cE} \right)
\]
where $\nabla^{\cE}$ is the connection defined by 
\begin{align*} 
\nabla^{\cE}_\alpha & \colon 
H^{\ev}(X) \otimes \cc[z][\![Q,\bfs,\tau]\!]
\to z^{-1} H^{\ev}(X)\otimes \cc[z][\![Q,\bfs,\tau]\!] \\ 
\nabla^{\cE}_\alpha & = \partial_\alpha + 
\frac{1}{z} \left( T_\alpha \bullet_\tau^{\cE} \right), \qquad 
\alpha = 0,\dots, s, 
\end{align*}
and $S_{\cE}$ is the `$z$-sesquilinear' pairing 
on $H^{\ev}(X)\otimes \cc[z][\![Q,\bfs,\tau]\!]$ 
defined by 
\begin{align*} 
S_{\cE}(u, v) = (u(-z), v(z))_{\cE} 
\end{align*} 
for $u, v \in H^{\ev}(X)\otimes \cc[z][\![Q,\bfs,\tau]\!]$. 
The connection $\nabla^{\cE}$ is called the \emph{quantum connection}. 
When $\bc=1$ and $E=0$, the triple 
\[
\QDM(X) = \left(H^{\ev}(X)\otimes \cc[z][\![Q,\tau]\!], 
\nabla = \nabla^{(\bc=1,E=0)}, S= S_{(\bc=1, E=0)}\right)
\] 
is called the \emph{quantum $D$-module} of $X$. 
\end{defi}

\begin{rem} 
The module $H^{\ev}(X)\otimes \cc[z][\![Q,\bfs,\tau]\!]$ 
should be viewed as the module of sections of a vector bundle 
over the formal neighbourhood of the point $Q=\bfs=\tau=z=0$. 
Since the connection $\nabla^{\cE}$ does not preserve 
$H^{\ev}(X)\otimes \cc[z][\![Q,\bfs,\tau]\!]$, 
quantum $D$-module is \emph{not} a $D$-module in the traditional sense.  
It can be regarded as a lattice in the $D$-module $H^{\ev}(X)\otimes 
\cc[z^\pm][\![Q,\bfs,\tau]\!]$ 
(see, e.g.~\cite[p.18]{Sdivf}). 
\end{rem} 
\begin{rem} 
\label{rem:specialization} 
As discussed, structure constants of the quantum product belong 
to the subring $\cc[\![Q e^{\tau_2}, \bfs,\tau']\!]$. 
Therefore the twisted quantum $D$-modules can be defined 
over $\cc[z][\![Q e^{\tau_2},\bfs,\tau']\!]$.  
This will be important when we specialize $Q$ to one 
in \S \ref{sec:non-equivariant,limit}. 
\end{rem} 

\begin{rem} 
\label{rem:connection_z}
For $\bc=1$ and $E=0$, we can complete the quantum connection 
$\nabla$ in the $z$-direction as a flat connection. We define 
\[
\nabla_{z\partial_z} 
= z\partial_z - \frac{1}{z}\left(\frE \bullet_\tau\right) 
+ \frac{\deg}{2} 
\]
where $\frE = \sum_{\alpha=0}^s 
(1-\frac{1}{2} \deg T_\alpha) t^\alpha T_\alpha 
+ c_1(TX)$ is the Euler vector field. 
\end{rem}

The quantum connection $\nabla^{\cE}$ is known to be flat and 
admit a fundamental solution. The fundamental solution 
of the following form was introduced by 
Givental \cite[Corollary 6.2]{Givental-Equivariant-GW}. 
We define $L_{\cE}(\tau,z) \in \End(H^{\ev}(X)) 
\otimes \cc[z^{-1}][\![Q,\bfs,\tau]\!]$ 
by the formula:  
\begin{align*}
L_{\cE}(\tau,z)\gamma =
\gamma-\sum_{\alpha=0}^{s}
\Bigcorr{\frac{\gamma}{z+\psi},T_{\alpha}}_{ \tau}^{\cE}
\frac{T^{\alpha}}{\bc(E)} 
\end{align*}
where $\gamma/(z+\psi)$ in the correlator should be expanded in 
the geometric series $\sum_{n=0}^\infty \gamma \psi^n (-z)^{-n-1}$.

\begin{prop}[{see, e.g.~\cite[\S 2]{Pand-after-Givental}, 
\cite[Proposition 2.1]{iritani_quantum_2011}}]
\label{prop:prop,QDM,S,flat,...}
The quantum connection $\nabla^{\cE}$ is flat and $L_{\cE}(\tau,z)$ 
gives the fundamental solution for $\nabla^{\cE}$. Namely we have 
\[
\nabla^{\cE}_\alpha \left( L_{\cE}(\tau,z) \gamma \right ) = 0  
\qquad 
\alpha = 0,\dots, s 
\] 
for all $\gamma \in H^{\ev}(X)$. 
Moreover $S_{\cE}$ is flat for $\nabla^{\cE}$ 
and $L_{\cE}(\tau,z)$ 
is an isometry for $S_{\cE}$: 
\begin{align*} 
d S_{\cE}(u, v) & = S_{\cE}\left( \nabla^{\cE}u, v 
\right) + S_{\cE}\left(u, \nabla^{\cE} v\right)  \\ 
S_{\cE}(u, v)&  = S_{\cE}\left(L_{\cE} u, L_{\cE}v\right) 
\end{align*} 
where $u, v \in H^{\ev}(X)\otimes \cc[z][\![Q,\bfs,\tau]\!]$. 
\end{prop}

From the last point of Proposition \ref{prop:prop,QDM,S,flat,...}, 
one can deduce that the inverse of $L_{\cE}$ is given by 
the adjoint of $L_{\cE}(\tau,-z)$. Explicitly: 
\begin{equation}
\label{eq:inverse_fundsol} 
L_{\cE}(\tau,z)^{-1} \gamma = \gamma + 
\sum_{\alpha=0}^s \Bigcorr{\frac{T_\alpha}{z-\psi}, \gamma}^{\cE}_\tau  
\frac{T_\alpha}{\bc(E)}. 
\end{equation}

\begin{defi} 
\label{defi:J}
The \emph{$\cE$-twisted $J$-function} is defined to be 
\begin{align}\label{eq:J,func}
\begin{split} 
    J_{\cE}(\tau,z)&:=z L_{\cE}(\tau,z)^{-1} \bun \\
&=z+\tau+\sum_{\alpha=0}^{s}
\Bigcorr{\frac{T_{\alpha}}{z-\psi}}_{\tau}^{\cE}
\frac{T^{\alpha}}{\bc(E)} 
\end{split} 
\end{align}
\end{defi} 
We deduce the following equality for $\alpha=0,\dots,s$: 
\begin{align}\label{eq:J,L}
  L_{\cE}(\tau,z)^{-1} T_{\alpha}=  
L_{\cE}(\tau,z)^{-1} z\nabla^{\cE}_{\alpha}\bun= 
\partial_{\alpha}J_{\cE}(\tau,z). 
\end{align}

\begin{rem} 
\label{rem:fullyflat} 
When $\bc=1$ and $E=0$, we can complete the quantum 
connection $\nabla$ in the $z$-direction as in Remark \ref{rem:connection_z}. 
The fundamental solution for flat sections, including in the $z$-direction, 
is given by 
$L(\tau,z) z^{-\deg/2}z^{c_1(TX)}$, see 
\cite[Proposition 3.5]{iritani_quantum_2011}.  
\end{rem} 

\begin{rem}
\label{rem:fundsol_divisor} 
The divisor equation for descendant invariants shows that 
\begin{equation} 
\label{eq:fundsol_divisor} 
L_{\cE}(\tau,z) \gamma 
= e^{-\tau_2/z} \gamma 
+ 
\sum_{\substack{(d,\ell)\neq (0,0) \\ d\in\eff(X), \ell \geq 0}} 
\frac{Q^d e^{\tau_2(d)}}{\ell!}  
\llangle \frac{e^{-\tau_2/z} \gamma}{-z-\psi}, 
\tau',\dots,\tau', T_\alpha \rrangle^{\cE}_{0,\ell+2,d} 
\frac{T^\alpha}{\bc(E)}. 
\end{equation} 
See, e.g.~\cite[\S 2.5]{iritani_quantum_2011}. 
In particular $L_{\cE}$ belongs to 
$\End(H^{\ev}(X)) \otimes 
\cc[z^{-1}][\![Qe^{\tau_2},\bfs,\tau']\!][\tau_2]$. 
\end{rem}

\subsection{Quantum Serre theorem in terms of quantum $D$-modules}
\label{subsec:interpr-quant-serre}
We formally associate to $\bc(\cdot)$ another multiplicative
class $\bc^{*}(\cdot)$ by the formula: 
\begin{align}\label{eq:sk*}
  \bc^{*}(\cdot)=\exp\left(\sum_{k\geq 0}(-1)^{k+1}
s_{k}\ch_{ k}(\cdot)\right). 
\end{align}
The class $\bc^*$ corresponds to the choice of parameters 
$\bfs^* = (s^*_0, s^*_1, s^*_2,\dots)$ 
with $s_{k}^{*}=(-1)^{k+1}s_{k}$.
For any vector bundle $G$, we have 
\begin{displaymath}
  \bc^{*}(G^{\vee})\bc(G)= \bun. 
\end{displaymath}

\begin{defi} 
Define the map $f \colon H^{\ev}(X)
\to H^{\ev}(X)$ by the formula: 
\begin{equation} 
\label{eq:f,map} 
f(\tau) = \sum_{\alpha=0}^{s}\< T^{\alpha}, 
\bc^{*}(E^{\vee}) \>_{\tau}^{\cE}T_{\alpha}. 
\end{equation} 
More precisely, the formula defines a morphism 
$f\colon \Spf \cc[\![Q,\bfs,\tau]\!] \to \Spf 
\cc[\![Q,\bfs,\tau]\!]$ of formal schemes. 
\end{defi} 

\begin{defi}
\label{def:QS_pairing} 
The \emph{quantum Serre pairing} $S^{\QS}$ is the 
$z$-sesquilinear pairing on $H^{\ev}(X) \otimes 
\cc[z][\![Q,\bfs,\tau]\!]$ defined by: 
\[
S^{\QS}(u,v) = \int_X u(-z) \cup v(z) 
\]
for $u, v\in H^{\ev}(X)\otimes \cc[z][\![Q,\bfs,\tau]\!]$. 
\end{defi}

\begin{thm} 
\label{thm:quantum,Serre,QDM}
{\rm (1)} The twisted quantum $D$-modules $\QDM_{\cE}(X)$ and 
$f^*\QDM_{\cEd}(X)$ are dual to each other with respect to 
$S^{\QS}$, i.e.~
\begin{equation}
\label{eq:flatness_S^QS}
\partial_{\alpha}S^{\QS}(u,v)
= S^{\QS}\big(\nabla^{\cE}_{\alpha}u,v \big)
+S^{\QS}\big(u, (f^{*}\nabla^{\cEd})_{\alpha}v \big) 
\end{equation}
for $u,v\in H^{\ev}(X)\otimes \cc[z][\![Q,\bfs,\tau]\!]$. 

{\rm (2)} The isomorphism of vector bundles 
  \begin{align*}
    \bc(E):\QDM_{\cE}(X) &\to f^{*}\QDM_{\cEd}(X) \\
\alpha &\mapsto \bc(E)\cup \alpha
  \end{align*}
intertwines the connections $\nabla^{\cE}$, $f^*\nabla^{\cEd}$ 
and the pairings $S_{\cE}$, $f^*S_{\cEd}$. 

{\rm (3)} The fundamental solutions satisfy the following properties: 
\begin{align*} 
\bc(E) L_{\cE}(\tau,z)  & = L_{\cEd}(f(\tau),z) \bc(E), \\ 
S^{\QS}(u,v)& = S^{\QS}(L_{\cE}u , (f^*L_{\cEd}) v).
\end{align*}  
\end{thm}

We will give a proof of this theorem in \S \ref{sec:quantum-Serre}. 

\begin{rem}
\label{rem:our,quantum,Serre,}
(1) The pull back $f^*\nabla^{\cEd}$ is defined to be 
\[
(f^*\nabla^{\cEd})_\alpha 
= 
\partial_\alpha + \frac{1}{z} 
\sum_{\beta=0}^s 
\frac{\partial f^\beta(\tau)}{\partial t^\alpha} 
\left( T_\beta \bullet^{\cEd}_{f(\tau)}   \right). 
\]
where we set $f(\tau) = \sum_{\alpha=0}^s f^\alpha(\tau) T_\alpha$. 
The flatness \eqref{eq:flatness_S^QS} of $S^{\QS}$ implies 
a certain complicated relationship between the quantum products 
$\bullet^{\cE}_\tau$, $\bullet^{\cEd}_{f(\tau)}$. 

(2) The map $\bc(E)$ in the above theorem is obtained as 
the composition of the quantum Serre duality and the self-duality: 
\[
\xymatrix{
& \ar[rd]^{S_{\cEd}}_{\cong} 
(-)^* \left(f^*\QDM_{\cEd}(X) \right)^\vee 
& 
\\
\ar[ru]^{S^{\QS}}_{\cong}  \ar[rr]^{\bc(E)} 
\QDM_{\cE}(X)  
& & 
f^*\QDM_{\cEd}(X) 
}
\]
where $(-)^*$ means the pull-back by the change 
$z\mapsto -z$ of sign and $(\cdots)^\vee$ means the dual 
as $\cc[z][\![Q,\bfs,\tau]\!]$-modules. 
Therefore part (1) of the 
theorem is equivalent to part (2). 
\end{rem}

\subsection{Quantum Serre theorem of Coates-Givental}
\label{subsec:twisted,cone}

Coates and Givental \cite{Givental-Coates-2007-QRR} 
stated quantum Serre theorem as an equality of Lagrangian cones. 
We review the language 
of Lagrangian cones and explain quantum Serre theorem. 

\subsubsection{Givental's symplectic vector space} 
Givental's symplectic vector space for the $(\bc,E)$-twisted 
Gromov-Witten theory is an infinite dimensional 
$\cc[\![Q,\bfs]\!]$-module: 
\[
\cH = H^{\ev}(X) \otimes \cc[z,z^{-1}][\![Q,\bfs]\!] 
\]
equipped with the anti-symmetric pairing 
\[
\Omega_{\cE} (f, g) = \Res_{z=0} (f(-z), g(z))_{\cE} dz.   
\]
The space $\cH$ has a standard polarization 
\[
\cH= \cH_+ \oplus \cH_- 
\]
where $\cH_\pm$ are $\Omega_{\cE}$-isotropic subspaces:
\begin{align*} 
\cH_+ &= H^{\ev}(X) \otimes \cc[z][\![Q,\bfs]\!] \\ 
\cH_- & = z^{-1} H^{\ev}(X) \otimes \cc[z^{-1}][\![Q,\bfs]\!].  
\end{align*} 
This polarization identifies $\cH$ with the total space 
of the cotangent bundle $T^*\cH_+$. 
A general element on $\cH$ can be written in the form\footnote
{Notice that the dual basis of $\{T_\alpha\}_{\alpha=0}^s$ 
with respect to the pairing $(\cdot,\cdot)_{\cE}$ is 
$\{\bc(E)^{-1}T^{\alpha}\}_{\alpha=0}^s$.}: 
\[
\sum_{k=0}^\infty \sum_{\alpha=0}^s 
q_k^\alpha T_\alpha z^k 
+ \sum_{k=0}^\infty \sum_{\alpha=0}^s 
p_{k,\alpha} \bc(E)^{-1}T^\alpha  \frac{1}{(-z)^{k+1}} 
\]
with $p_{k,\alpha}, q_k^\alpha \in \cc[\![Q,\bfs]\!]$. 
The coefficients $p_{k,\alpha}, q_k^\alpha$ here give Darboux 
co-ordinates on $\cH$ in the sense that 
$\Omega_{\cE} = \sum_{k,\alpha} dp_{k,\alpha} \wedge dq_k^\alpha$.

\subsubsection{Twisted Lagrangian cones}
\label{subsubsec:twisted,cones}

The genus-zero gravitational descendant Gromov-Witten 
potential is a function on the formal 
neighbourhood of $-z \bun$ in $\cH_+$ 
defined by the formula:  
\[
\cF^{0, \grav}_{\cE}(-z + \bt(z) ) =  
\< \>_{\bt(\psi)}^{\cE}  
= 
\sum_{d\in \eff(X)} \sum_{k=0}^\infty 
\frac{Q^d}{k!} 
\left\langle \bt(\psi_1),\dots, \bt(\psi_k) \right\rangle_{0,k,d} 
\]
where $\bt(z) = \sum_{k=0}^\infty t_k z^k$ with 
$t_k = \sum_{\alpha=0}^s 
t_k^\alpha T_\alpha$ is a formal variable in $\cH_+$.  
The variables $\{t_k^\alpha\}$ are related to the variables $\{q_k^\alpha\}$ 
by $t_k^\alpha = q_k^\alpha + \delta_{k,1} \delta_{\alpha,0}$.  

\begin{defi} 
The $\cE$-twisted Lagrangian cone $\cL_{\cE}\subset \cH$ is 
the graph of the differential $d \cF^{0,\grav}_{\cE} 
\colon \cH_+ \to T^* \cH_+ \cong\cH$.  
In terms of the Darboux co-ordinates above, 
$\cL_{\cE}$ is cut out by the equations 
\[ 
p_{k,\alpha} = 
\frac{\partial \cF^{0,\grav}_{\cE}}{\partial q_k^\alpha}. 
\]
In other words, it consists of points of the form: 
\begin{equation} 
\label{eq:point_on_the_cone} 
-z + \bt(z) + \sum_{\alpha=0}^s  
\Bigcorr{\frac{T^\alpha}{-z-\psi}}^{\cE}_{\bt(\psi)} \frac{T_\alpha}{\bc(E)}  
\end{equation} 
with $\bt(z) \in \cH_+$. 
\end{defi} 

Givental \cite{Givental:symplectic} showed that the submanifold 
$\cL_{\cE}$ is in fact a cone (with vertex at the origin of $\cH$). 
Moreover he showed the following geometric property of $\cL_{\cE}$: 
for every tangent space $T$ of $\cL_{\cE}$ ($T$ is a linear subspace 
of $\cH$), 
\begin{itemize} 
\item $zT = T \cap \cL_{\cE}$; 
\item the tangent space of $\cL_{\cE}$ 
at any point in $zT\subset \cL_{\cE}$ is $T$. 
\end{itemize} 
Note that the twisted $J$-function (\ref{eq:J,func}) is a family of 
elements lying on $\cL_{\cE}$: 
\begin{align*}
  J_{\cE}(\tau,-z)=-z+\tau+\sum_{\alpha=0}^{s}
\Bigcorr{\frac{T_{\alpha}}{-z-\psi}}_{\tau}^{\cE}\frac{T^{\alpha}}{\bc(E)}  
\end{align*}
obtained from \eqref{eq:point_on_the_cone} by setting $\bt(z) = \tau$. 

\begin{rem} 
In \cite[Appendix B]{CCIT-computingGW}, 
$\cL_{\cE}$ is defined as a formal scheme over $\cc[\![Q,\bfs]\!]$. 
For a complete Hausdorff topological $\cc[\![Q,\bfs]\!]$-algebra $R$, 
we have the notion of $R$-valued points on $\cL_{\cE}$. 
An $R$-valued point on $\cL_{\cE}$ 
is a point of the form \eqref{eq:point_on_the_cone}  
with $t_k^\alpha \in R$ such that $t_k^\alpha$ are topologically 
nilpotent, i.e.~$\lim_{n\to \infty} (t_k^\alpha)^n =0$.  
A $\cc[\![Q,\bfs]\!]$-valued point is given by $t_k^\alpha \in \cc[\![Q,\bfs]\!]$ 
with $t_k^\alpha|_{Q= \bfs=0} =0$. 
The $J$-function is a $\cc[\![Q,\bfs,\tau]\!]$-valued point 
on $\cL_{\cE}$. In what follows we mean by a point \eqref{eq:point_on_the_cone} 
on $\cL_{\cE}$ a $\cc[\![Q,\bfs]\!]$-valued point, but the discussion 
applies to a general $R$-valued point. 
\end{rem} 

\subsubsection{Tangent space to the twisted Lagrangian cone}
\label{subsubsec:tgt,space,twisted,cones}

Let $g= g(\bt)$ denote the point on $\cL_{\cE}$ given in equation 
\eqref{eq:point_on_the_cone}. Differentiating $g(\bt)$ in $t_k^\alpha$, 
we obtain the following tangent vector: 
\begin{align}
\label{eq:tangent_vector_cone}
\frac{\partial g(\bt)}{\partial t_k^\alpha} = 
T_{\alpha}z^{k}+
\sum_{\beta=0}^s 
\Bigcorr{T_{\alpha}\psi^{k},\frac{T^\beta}{-z-\psi}}^{\cE}_{\bt(\psi)}
\frac{T_\beta}{\bc(E)} 
\quad \text{in} \quad  T_{g} \cL_{\cE}. 
\end{align}
The tangent space $T_g \cL_{\cE}$ is spanned by these vectors. 
Since $T_{g} \cL_{\cE}$ is complementary to $\cH_-$, 
$T_g \cL_{\cE}$ intersects with $\bun+ \cH_-$ at a unique 
point. The intersection point is the one 
\eqref{eq:tangent_vector_cone} with $k=\alpha=0$: 
\[
(\bun + \cH_-) \cap T_g \cL_{\cE} = \left\{
\frac{\partial g(\bt)}{\partial t_0^0} = 
\bun-\frac{\ttau(\bt)}{z}+O(z^{-2}) \right\} 
\] 
where 
\begin{align}
\label{eq:tautilde}
\ttau(\bt) = 
\sum_{\alpha=0}^{s}\<\bun,T_{\alpha}\>_{\bt(z)}^{\cE} 
\frac{T^{\alpha}}{\bc(E)}. 
\end{align}
Givental \cite{Givental:symplectic} observed that each tangent 
space to the cone is uniquely parametrized by the value 
$\ttau(\bt)$, i.e.~the tangent spaces 
at $g(\bt_1)$ and $g(\bt_2)$ are equal if and only if $\ttau(\bt_1) 
= \ttau(\bt_2)$. The string equation shows that 
$\ttau(\bt) = \tau$ when $\bt(z) = \tau$. 
Hence $T_{g(\bt)} \cL_{\cE}$ equals the tangent 
space at $g(\ttau(\bt)) = J_{\cE}(\ttau(\bt),-z)$.

\begin{prop}[\cite{Givental:symplectic}, see also 
{\cite[Proposition B.4]{CCIT-computingGW}}]
\label{prop:tgt,cone,R,module}
Let $g = g(\bt)$ denote the point of $\cL_{\cE}$ given in equation 
\eqref{eq:point_on_the_cone}. 
The tangent space $T_g \cL_{\cE}$ is a free 
$\cc[z][\![Q,\bfs]\!]$-module generated by 
the derivatives  of the twisted $J$-function:  
    \begin{displaymath}
      \left. \frac{\partial J_{\cE}}{\partial t^\alpha}(\tau,-z) 
\right|_{\tau=\ttau(\bt)}.
    \end{displaymath}
\end{prop}
\begin{proof} 
As discussed, $T_g \cL_{\cE}$ equals 
the tangent space of $\cL_{\cE}$ at $J(\tau,-z)$ with $\tau=\ttau(\bt)$. 
On the other hand, the tangent space at $J_{\cE}(\tau,-z)$ is 
freely generated by the derivatives $\partial_\alpha J_{\cE}(\tau,-z)$ 
\cite[Lemma B.5]{CCIT-computingGW}, and the result follows. 
\end{proof}

\subsubsection{Relations between Lagrangian cones and $\QDM_{\cE}(X)$}
\label{subsubsec:cone,QDM}

Proposition \ref{prop:tgt,cone,R,module} means that 
the quantum $D$-module can be identified with the family of 
tangent spaces to the Lagrangian cone $\cL_{\cE}$ at 
the $J$-function $J_{\cE}(\tau,-z)$. 
\begin{align*} 
(-)^*\QDM_{\cE}(X)_\tau & \cong T_{J_{\cE}(\tau,-z)} \cL_{\cE} \\ 
T_\alpha & \mapsto \partial_\alpha J_{\cE}(\tau,-z) 
= L_{\cE}(\tau,-z)^{-1} T_\alpha
\end{align*} 
where $(-)^*$ denotes the pull-back by the sign change 
$z\mapsto -z$ and we used \eqref{eq:J,L}. 
This identification preserves the pairing 
$S_{\cE}$ and intertwines the quantum 
connection $\nabla$ on $\QDM_{\cE}(X)$ with the trivial 
differential $d$ on $\cH$: this follows from the properties of $L_{\cE}$ in 
Proposition \ref{prop:prop,QDM,S,flat,...}.

\subsubsection{Quantum Serre theorem of Coates-Givental}
\label{subsubsec:quantum,Serre}

\begin{thm}[Coates-Givental 
{\cite[Corollary 9]{Givental-Coates-2007-QRR}}]
  \label{thm:Quantum,Serre,Coates-Givental} The multiplication by
  $\bc(E)$ defines a symplectomorphism $\bc(E) \colon
  (\cH,\Omega_{\bc(E)}) \to
  (\cH,\Omega_{\bc^{*}(E^{\vee})})$ and identifies the twisted
  Lagrangian cones:
$$
\bc(E) \cL_{\cE} = \cL_{\cEd}. 
$$
\end{thm} 

For any $\gamma = \sum_{\alpha=0}^s \gamma^\alpha T_\alpha 
\in H^{\ev}(X)$, we write  
$\partial_{\gamma}:=\sum_{\alpha=0}^{s}\gamma^{\alpha}
\partial_{\alpha}$ for the directional derivative. 

\begin{cor}\label{cor:tgt,cone,QS}
Let $T_\tau$ denote the tangent space 
of $\cL_{\cE}$ at $J_{\cE}(\tau,-z)$ 
and let $T_\tau^*$ denote the tangent space 
of $\cL_{\cEd}$ at $J_{\cEd}(\tau,-z)$. 
Then we have: 

{\rm (1)} $\bc(E)T_\tau 
=T_{f(\tau)}^*$, where $f$ was defined in \eqref{eq:f,map};

{\rm (2)} 
$\bc(E)\partial_{\alpha}J_{\cE}(\tau,-z) 
=(\partial_{\bc(E)\cup T_{\alpha}}J_{\cEd})(\tau^*,-z)|_{\tau^* = f(\tau)}$ 
for $\alpha=0,1,\dots,s$. 
\end{cor}

\begin{rem} 
Exchanging the twist $\cE$ with $\cEd$ in the above Corollary, we obtain 
\[
 \bc^{*}(E^{\vee})z\partial_{\bc(E)}J_{\cEd}(\tau ,z ) 
=J_{\cE}(\widetilde{f}(\tau),z)
\]
  where $\widetilde{f}(\tau)= \sum_{\alpha=0}^{s}\<
  T_{\alpha},\bc(E) \>_{\tau}^{\cEd}T^{\alpha}$. 
This is exactly \cite[Corollary 10]{Givental-Coates-2007-QRR}.
\end{rem}

\begin{proof}[Proof of Corollary  \ref{cor:tgt,cone,QS}]
(1) Theorem \ref{thm:Quantum,Serre,Coates-Givental} implies that 
$\bc(E)T_\tau$ is a tangent space to the cone $\cL_{\cEd}$.  
Therefore, by the discussion in the previous section 
\S \ref{subsubsec:tgt,space,twisted,cones}, 
$\bc(E) T_\tau$ equals $T^*_\sigma$ with 
$\sigma$ given by the intersection point: 
\begin{equation} 
\label{eq:intersection_point_desired} 
(\bun+ \cH_-) \cap \bc(E) T_\tau = 
\left\{ \bun - \frac{\sigma}{z} + O(z^{-2}) \right \}. 
\end{equation} 
Note that 
\[
\partial_{\bc^*(E^\vee)} J_{\cE}(\tau,-z) = 
\bc^*(E^\vee) - \frac{1}{z} \sum_{\alpha=0}^s 
\<T^\alpha, \bc^*(E^\vee)\>_\tau \frac{T_\alpha}{\bc(E)} + O(z^{-2}) 
\]
lies in $T_\tau$. Multiplying this by $\bc(E)$, we obtain the 
intersection point in \eqref{eq:intersection_point_desired} 
and we have $\sigma = f(\tau)$ as required 
(recall that $\bc(E) \bc^*(E^\vee) = 1$).

(2) By part (1), the vector $\bc(E) \partial_\alpha J_{\cE}(\tau,-z)$ belongs to 
the tangent space $T_{f(\tau)}^*$ of $\cL_{\cEd}$. 
It has the following asymptotics: 
\[
\bc(E)\partial_\alpha J_{\cE}(\tau,-z) = \bc(E)\cup T_\alpha + O(z^{-1}). 
\]
By the description of tangent spaces in 
\S \ref{subsubsec:tgt,space,twisted,cones}, 
a tangent vector in $T_{f(\tau)}^*$ 
with this asymptotics is unique and 
is given by $(\partial_{\bc(E) \cup T_\alpha} J_{\cEd})(\tau^*,-z)$ 
with $\tau^* = f(\tau)$.  
\end{proof}

 \subsection{A proof of Theorem \ref{thm:quantum,Serre,QDM}}
 \label{sec:quantum-Serre}
We use the correspondence in \S \ref{subsubsec:cone,QDM} 
between quantum $D$-module and tangent spaces to the Givental cone. 
Then Corollary \ref{cor:tgt,cone,QS} implies that the map 
$\bc(E) \colon \QDM_{\cE}(X) \to f^*\QDM_{\cEd}(X)$ 
respects the quantum connection. Also it is obvious that 
the map $\bc(E)$ intertwines the pairings $S_{\cE}$ and $f^*S_{\cEd}$. 
This shows part (2) of the theorem. 
Also part (2) of Corollary \ref{cor:tgt,cone,QS} implies, in view of 
\eqref{eq:J,L} 
\[
\bc(E) L_{\cE}(\tau,-z)^{-1} T_\alpha = L_{\cEd}(f(\tau),-z)^{-1}
\left( \bc(E) \cup T_\alpha\right) 
\]
for $\alpha=0,1,\dots,s$. 
This implies the first equation of part (3). 
To see the second equation of part (3), we calculate: 
\begin{align*} 
S^{\QS}(u,v) &= S_{\cE}(u,\bc(E)^{-1}v) 
= S_{\cE}(L_{\cE}u, L_{\cE}\be(E)^{-1}  v) \\
& = S_{\cE}(L_{\cE}u, \bc(E)^{-1} (f^*L_{\cEd})v) 
= S^{\QS}(L_{\cE}, (f^*L_{\cEd})v)
\end{align*}
where we used Proposition \ref{prop:prop,QDM,S,flat,...}.  
Part (1) of the theorem is equivalent to part (2), as explained 
in Remark \ref{rem:our,quantum,Serre,}. 

\section{Quantum Serre duality for Euler twisted theory}
\label{sec:non-equivariant,limit}

In this section we apply 
Theorem \ref{thm:quantum,Serre,QDM} 
to the equivariant Euler class $\eqeuler$ 
and a convex vector bundle $E$. 
By taking the non-equivariant limit, we obtain a 
relationship among the quantum $D$-module 
twisted by the Euler class and the bundle $E$, 
the quantum $D$-module of the total space of $E^\vee$, 
and the quantum $D$-module 
of a submanifold $Z\subset X$ cut out by a regular section of $E$.

In order to ensure the well-defined non-equivariant limit, we 
assume that our vector bundle $E\to X$ is \emph{convex}, that is, for 
every genus-zero stable map $f\colon C\to X$ we have 
$H^1(C,f^*E) = 0$. The convexity assumption is satisfied, 
for example, 
if $\cO(E)$ is generated by global sections.

 \subsection{Equivariant Euler class}
\label{subsec:about,ground,ring,equivariant}

In this section we take $\bc$ to be the 
$\cc^\times$-equivariant Euler class $\eqeuler$. 
Given a vector bundle $G$, we let $\cc^\times$ act 
on $G$ by scaling the fibers and trivially on $X$. 
With respect to this $\cc^\times$-action we have 
\[
\eqeuler(G) = \sum_{i=0}^{\rank G} \lambda^{r-i} c_i (G)
\]
where $\lambda$ is the $\cc^\times$-equivariant parameter: 
the $\cc^\times$-equivariant cohomology of a point 
is $H^{*}_{\cc^\times}(\pt)=\cc[\lambda]$.  
Choosing $\eqeuler$ means the following specialization
\begin{displaymath}
  s_k:=\begin{cases} 
\log \lambda & \text{if $k=0$} \\ 
(-1)^{k-1}(k-1)!\lambda^{-k} & \text{if $k>0$} 
\end{cases}
\end{displaymath}
Although the parameters $s_k$ contain $\log \lambda$ and 
negative powers of $\lambda$, we will see that the 
$\eqeE$-twisted theory and $\eqeEi$-twisted 
theory are defined over the polynomial ring in $\lambda$, 
and hence admit the non-equivariant limit $\lambda \to 0$. 
Here the convexity of $E$ plays a role. 

\subsection{Specialization of the Novikov variable} 
\label{subsec:specialization} 
We henceforth specialize the Novikov variable $Q$ to one. 
By Remark \ref{rem:specialization}, 
the specialization $Q=1$ is well-defined: one has 
\begin{align*} 
& T_\alpha \bullet_\tau^{\cE} T_\beta\Bigr|_{Q=1} \in 
H^{\ev}(X)\otimes 
\cc[\![e^{\tau_2},\bfs,\tau']\!] \\ 
& L_{\cE}(\tau,z)\Bigr|_{Q=1} \in \End(H^{\ev}(X))\otimes 
\cc[z^{-1}][\![e^{\tau_2},\bfs,\tau']\!][\tau_2] \\ 
& f(\tau)|_{Q=1} \in H^{\ev}(X)\otimes \cc[\![e^{\tau_2},\bfs,\tau']\!] 
\qquad \text{see \eqref{eq:f,map}}
\end{align*} 
where $\cc[\![e^{\tau_2},\bfs,\tau']\!]$ is the completion of 
$\cc[e^{t^1},\dots,e^{t^s},t^0,t^{r+1},\dots,t^s, s_0,s_1,s_2,\dots]$. 
Since we chose $T_1,\dots,T_r$ to be a nef integral basis of $H_2(X,\zz)$, 
we have only nonnegative integral powers of $e^{t^1},\dots,e^{t^r}$ 
in the structure constants of quantum cohomology. 
The Euler-twisted quantum $D$-module will be defined over 
$\cc[z][\![e^{\tau_2},\tau']\!]$. 
In what follows, we shall omit $(\cdots)|_{Q=1}$ from the notation.

\subsection{Non-equivariant limit of $\QDM_{\eqeE}(X)$}
\label{subsec:non-eq,Euler,twist}

Let $\e = \lim_{\lambda \to 0} \eqeuler$ 
denote the non-equivariant Euler class. 
We first discuss the non-equivariant limit of 
$\QDM_{\eqeE}(X)$. 
Recall the $K$-class $E_{0,\ell,d}$ on the moduli space 
$\overline{\mathcal{M}}_{0,\ell}(X,d)$ introduced in 
\S\ref{subsubsec:twist-quant-prod}. 
The convexity assumption for $E$ implies that 
$E_{0,\ell,d}$ is represented by a vector bundle. 
Moreover the natural evaluation morphism 
$E_{0,\ell,d}\to \ev_{j}^{*}E$ at the $j$th marking 
is surjective for all $j\in \{1,\dots,\ell\}$. 
Define $E_{0,\ell,d}(j)$ by the exact sequence: 
\begin{align} 
\label{eq:exact_E(j)}
\xymatrix{0\ar[r]&E_{0,\ell,d}(j)\ar[r]&E_{0,\ell,d}
\ar[r]&\ev_{j}^{*}E\ar[r] &0}.  
\end{align}
We use the following variant of $\eqeE$-twisted 
invariants (see \cite{Pand-after-Givental}). 
For any $\gamma_{1}, \ldots,\gamma_{\ell}\in H^{\ev}(X)$ 
and any $k_{1}, \ldots ,k_{\ell}\in \nn$, we put: 
\begin{displaymath}
  \llangle \gamma_{1}\psi_{1}^{k_{1}}, \ldots, 
\widetilde{\gamma_{j}\psi_{j}^{k_{j}}}, \ldots, 
\gamma_{\ell}\psi_{\ell}^{k_{\ell}}\rrangle_{0,\ell,d}^{\eqeE}
:=\int_{[\overline{\mathcal{M}}_{0,\ell}(X,d)]^{\vir}}
\left(\prod_{i=1}^{\ell}\psi_{i}^{k_{i}}\ev_{i}^{*}\gamma_{i}\right) 
\eqeuler(E_{0,\ell,d}(j)). 
\end{displaymath}
This lies in the polynomial ring $\cc[\lambda]$. 

\begin{lem}[\cite{Pand-after-Givental}]
\label{lem:noneqlim_Euler} 
Suppose that $E$ is convex. 
The $\eqeE$-twisted quantum product $T_\alpha \bullet^{\eqeE}_\tau T_\beta$ 
lies in $H^{\ev}(X) \otimes \cc[\lambda][\![e^{\tau_2},\tau']\!]$ 
and admits the non-equivariant limit 
$T_\alpha\bullet^{\eE}_\tau T_\beta := \lim_{\lambda \to 0} 
(T_\alpha \bullet^{\eqeE}_\tau T_\beta)$. 
\end{lem} 
\begin{proof} 
Recall that the twisted quantum product \eqref{eq:tw,quantum,prod} 
is given by: 
\begin{displaymath}
\gamma_{1}\eqeEprod{\tau}\gamma_{2} =
\sum_{\alpha=0}^{s} 
\Bigcorr{\gamma_{1},\gamma_{2},\frac{T_{\alpha}}{\eqeuler(E)}}
_{\tau}^{\eqeE}T^{\alpha}. 
\end{displaymath}
From the exact sequence \eqref{eq:exact_E(j)}, we deduce that 
\[
\frac{\eqeuler(E_{0,\ell,d})}{\ev_{3}^*\eqeuler(E)} 
= \eqeuler(E_{0,\ell,d}(3)). 
\]
Therefore we have 
\begin{displaymath}
  \gamma_{1}\bullet_{\tau}^{\eqeE}\gamma_{2}=
\sum_{\alpha=0}^{s}
\bigcorr{\gamma_{1},\gamma_{2},\widetilde{T}_{\alpha}}_{\tau}^{\eqeE}
T^{\alpha} 
\end{displaymath}
and this lies in $H^{\ev}(X) \otimes \cc[\lambda][\![e^{\tau_2}, \tau']\!]$. 
\end{proof}

\begin{lem}\label{lem:equ,limit,f,inverse} 
Suppose that $E$ is convex. 
For $\cE = \eqeE$, 
the map $f(\tau)$ in \eqref{eq:f,map} 
lies in $H^{\ev}(X)\otimes \cc[\lambda][\![e^{\tau_2},\tau']\!]$ 
and admits the non-equivariant limit $\overline{f}(\tau) := 
\lim_{\lambda \to 0} f(\tau)$. 
\end{lem} 
\begin{proof} 
Note that $\bc^*(E^\vee)= \eqeulerstar(E^\vee)$ 
in \eqref{eq:f,map} equals $\eqeuler(E)^{-1}$.
Arguing as in Lemma \ref{lem:noneqlim_Euler}, 
we have 
\begin{align*}
f(\tau)=\sum_{\alpha=0}^{s} 
\bigcorr{T_{\alpha},\widetilde{1}}_{\tau}^{\eqeE}T^{\alpha}. 
\end{align*}
The conclusion follows. 
\end{proof}


By Lemma \ref{lem:noneqlim_Euler}, we deduce that 
the non-equivariant limit $\nabla^{\eE} = \lim_{\lambda \to 0}
\nabla^{\eqeE}$ of the quantum connection exists. 
Moreover it can be completed in the $z$-direction 
in a flat connection: 
\begin{align*}
  \nabla^{\eE}_{z\partial_{z}}& 
= z\partial_{z}-\frac{1}{z} \left( 
\frE^{\eE}\eEprod{\tau}\right)  +\frac{\deg}{2}
\end{align*}
where $\frE^{\eE}$ is the Euler vector field: 
\begin{equation} 
\label{eq:tw_Eulervf}
\frE^{\eE} =\sum_{\alpha=0}^{s}\left(1-\frac{\deg
      T_{\alpha}}{2}\right)t^{\alpha}T_{\alpha} + c_{1}(TX) -c_{1}(E). 
\end{equation} 
The non-equivariant limit $S_{\eE}(u,v) := \int_X u(-z) \cup v(z) \cup \e(E)$ 
of the pairing $S_{\eqeE}$ becomes degenerate. 
The pairing $S_{\eE}$ is not flat in the $z$-direction, 
but satisfies the following equation: 
\[
z\partial_{z} S_{\eE}(u,v) - S_{\eE}\left(\nabla^{\eE}_{z\partial_z} u, v\right) 
- S_{\eE}\left(u, \nabla^{\eE}_{z\partial_z} v\right) 
= - (\dim X -\rank E) S_{\eE}(u,v). 
\]
We refer to this by saying that \emph{$S_{\eE}$ is of weight $-(\dim X -\rank E)$}. 
Note that $z^{\dim X -\rank E} S_{\eE}$ is flat in both $\tau$ and $z$.

\begin{defi}[cf.~Definition \ref{def:quantumD-mod}] 
\label{defi:non,equivariant,limit,}
   We call the triple 
   \begin{displaymath}
     \QDM_{\eE}(X):=
\left( H^{\ev}(X)\otimes\mathbb{C}[z][\![e^{\tau_{2}},\tau']\!], 
\nabla^{\eE}, S_{\eE} \right)
   \end{displaymath}
the \emph{$\eE$-twisted quantum $D$-module}. 
\end{defi}

\begin{rem}
\label{rem:fullyflat_Euler}
By a similar argument, the fundamental solution 
$L_{\eqeE}$ in Proposition \ref{prop:prop,QDM,S,flat,...} 
can be written as: 
\[
L_{\eqeE}(\tau,z) \gamma 
= \gamma - \sum_{\alpha=0}^s 
\Bigcorr{\frac{\gamma}{z+\psi}, \widetilde{T_\alpha}}_{\tau}^{\eqeE}
T^\alpha 
\]
and therefore admits the non-equivariant limit $L_{\eE}$. 
The fundamental solution for $\nabla^{\eE}$, 
including in the $z$-direction,  
is given by $L_{\eE}(\tau,z)z^{-\frac{\deg}{2}}z^{c_{1}(TX)-c_{1}(E)}$.  
All the properties of Proposition
  \ref{prop:prop,QDM,S,flat,...} are true for the limit 
(see \cite[\S 2]{Mann-Mignon-2011-QDM-lci-nef} 
for a more precise statement). 
\end{rem}

\subsection{Quantum $D$-module of a section of $E$}
In this section we describe a relationship 
between the $\eE$-twisted quantum $D$-module 
and the quantum $D$-module of a submanifold $Z\subset X$ 
cut out by a regular section of $E$. 

Let $\iota \colon Z \to X$ denote the natural inclusion. 
The functoriality of virtual classes \cite{Kim-Kresch-Pantev} 
\[
[\overline{\mathcal{M}}_{0,\ell}(X,d)]^{\vir} \cap \e(E_{0,\ell,d}) 
= \sum_{\iota_*(d') =d} \iota_* 
[\overline{\mathcal{M}}_{0,\ell}(Z,d')]^{\vir} 
\]
together with the argument in \cite{Pand-after-Givental}, 
\cite[Corollary 2.5]{iritani_quantum_2011} 
shows that 
\begin{align}
\label{eq:iota_pullback}
  \iota^{*}\left(\gamma_{1}\bullet_{\tau}^{\eE}\gamma_{2}\right)
=(\iota^{*}\gamma_{1})\bullet^{Z}_{\iota^{*}\tau}
  (\iota^{*}\gamma_{2}) 
\end{align} 
for $\gamma_1,\gamma_2 \in H^{\ev}(X)$. 
Define the ambient part of the cohomology of $Z$ 
by $H_{\amb}^*(Z) = \im (\iota \colon H^*(X) \to H^*(Z))$.  
Equation \eqref{eq:iota_pullback} 
shows that the ambient part $H_{\amb}^*(Z)$ is closed under 
the quantum product $\bullet^Z_\tau$ of $Z$ 
as long as $\tau$ lies in the ambient part. 
\begin{defi} 
\label{def:ambientQDM}
The \emph{ambient part quantum $D$-module} 
of $Z$ is a triple 
\[ 
\QDM_{\amb}(Z):= \left( 
H^{\ev}_{\amb}(Z)\otimes 
\mathbb{C}[z][\![e^{\tau_{2}},\tau']\!], 
\nabla^{Z},  S_{Z} \right)   
\]
where the parameter $\tau = \tau_2 + \tau'$ is restricted to lie 
in the ambient part $H^{\ev}_{\amb}(Z)$ 
and $S_Z(u,v) = \int_Z u(-z) \cup v(z)$. 
We complete the quantum connection $\nabla^Z$ in the 
$z$-direction as in Remark \ref{rem:connection_z}; 
then $S_Z$ is of weight $(-\dim Z)$. 
\end{defi} 

Equation \eqref{eq:iota_pullback} proves the following 
proposition: 

\begin{prop}
\label{prop:QDM,amb}
The restriction map $\iota^* \colon H^{\ev}(X) \to H^{\ev}_{\amb}(Z)$  
induces a morphism between the quantum $D$-modules 
\begin{align*}
\iota^* \colon \QDM_{(\e,E)}(X) &\to (\iota^*)^* \QDM_{\amb}(Z) 
\end{align*}
which is compatible with the connection and the pairing. 
\end{prop}

 \subsection{Quantum $D$-module of the total space of $E^{\vee}$}
\label{subsec:non-equiv-limit,inverse,euler}

We explain that the non-equivariant limit of 
the $\eqeEi$-twisted quantum $D$-module 
is identified with the quantum $D$-module of the total space of $E^\vee$. 

The $\eqeEi$-twisted Gromov-Witten invariants admit 
a non-equivariant limit 
under the \emph{concavity}\footnote
{A bundle $E^\vee$ is said to be \emph{concave} if for every 
non-constant genus-zero stable map $f \colon C \to X$, one has 
$H^0(C,f^* E^\vee)=0$.} assumption for $E^\vee$ 
and they are called \emph{local Gromov-Witten invariants} 
\cite{Givental-Equivariant-GW, Givental-Elliptic-1998, 
LocalMirrorSymmetry-Chiang-Klemm-Yau-Zaslow-99}. 
In this paper, we only impose the weaker assumption that 
$E$ is convex (see Remark \ref{rem:concavity} below). 
In this case, a non-equivariant 
limit of $\eqeEi$-twisted invariants may not exist, 
but a non-equivariant limit of the twisted quantum product 
is still well-defined.

The virtual localization formula  
\cite{Graber-Pandaripande-2003-virtual-localisation} 
gives the following: 
  
\begin{prop}\label{prop:GW,tot,space,GW,twist,X} 
For $\gamma_1,\dots,\gamma_\ell \in H^{\ev}(X)$ 
and non-negative integers $k_1,\dots,k_\ell$, we have 
\[
\llangle 
\gamma_1\psi^{k_1},\dots,\gamma_\ell \psi^{k_\ell}
\rrangle_{0,\ell,d}^{\eqeEi} 
= \llangle \gamma_1 \psi^{k_1},\dots,\gamma_{\ell} \psi^{k_\ell} 
\rrangle_{0,\ell,d}^{E^\vee, \cc^\times}
\] 
where the right-hand side is the $\cc^\times$-equivariant Gromov-Witten 
invariant of $E^\vee$ with respect to the $\cc^\times$-action on $E^\vee$ 
scaling the fibers. 
\end{prop}

The non-equivariant Gromov-Witten invariants for $E^\vee$ 
are ill-defined in general because the moduli space 
$\overline{\mathcal{M}}_{0,\ell}(E^\vee,d)$ can be non-compact. 
The following lemma, however, shows the existence 
of the non-equivariant quantum product of $E^\vee$. 

\begin{lem} 
\label{lem:properness}
Let $E$ be a vector bundle on $X$ such that 
$f^*E$ is generated by global sections 
for any stable maps $f\colon C\to X$ of genus $g$. 
Then the evaluation map $
\ev_i \colon \overline{\mathcal{M}}_{g,\ell}(E^\vee,d) \to E^\vee$ 
is proper for all $i\in \{1,\dots,\ell\}$. 
In particular, when $E$ is convex, $\ev_i \colon \overline{\mathcal{M}}_{0,\ell}
(E^\vee, d) \to E^\vee$ is proper. 
\end{lem} 
\begin{proof} 
The convexity of $E$ implies that, for any map $u\colon \pp^1 \to X$, 
$u^*E$ is isomorphic to $\bigoplus_{i=1}^r \cO(k_i)$  
with $k_i \geq 0$. Thus the latter statement follows from 
the former. 

Let us prove the former statement. 
We start with the remark that, for every stable map $f\colon C \to X$, 
the evaluation map $\ev_i \colon H^0(C,f^*E^\vee) \to E_{f(x_i)}^\vee$ 
at the $i$-th marking $x_i\in C$ is injective. 
Suppose that a section $s\in H^0(C,f^*E^\vee)$ vanishes 
at $x_i$, i.e.~$\ev_i(s) = 0$. For every $u \in H^0(C,f^*E)$, 
the pairing $\langle s,u\rangle$ is a global section of $\cO_C$ 
which vanishes at $x_i$. Then $\langle s, u\rangle$ must be identically zero 
on $C$. Since $f^*E$ is generated by global sections, 
this implies that $s=0$. Hence we have shown that the evaluation map 
$H^0(C,f^*E^\vee) \to E^\vee_{f(x_i)}$ is injective. 

Giving a stable map to $E^\vee$ is equivalent to giving 
a stable map $f\colon C\to X$ and a section of $H^0(C, f^*E^\vee)$. 
Therefore, by the preceding remark, 
the moduli functor $\overline{\mathcal{M}}_{g,\ell}(E^\vee,d)$ 
is a subfunctor of $\overline{\mathcal{M}}_{g,\ell}(X,d) 
\times E^\vee$ 
via the natural projection $\overline{\mathcal{M}}_{g,\ell}(E^\vee,d) 
\to \overline{\mathcal{M}}_{g,\ell}(X,d)$ 
and the evaluation map 
$\ev_i \colon \overline{\mathcal{M}}_{g,\ell}(E^\vee,d) \to E^\vee$. 
Since $\overline{\mathcal{M}}_{g,\ell}(X,d)$ is proper, 
it suffices to show that the map $
\overline{\mathcal{M}}_{g,\ell}
(E^\vee,d) \to \overline{\mathcal{M}}_{g,\ell}(X,d) \times E^\vee$ 
is proper. 
We use the valuative criterion for properness 
(see \cite[Theorem 4.19]{Deligne-Mumford}). 
Let $R$ be a DVR. Suppose that we are given a stable map $f\colon C_R \to X$ 
over $\Spec(R)$ and an $R$-valued point $v\in E^\vee(R)$. 
These data $(f,v)$ give a map $\Spec(R) \to \overline{\mathcal{M}}_{g,\ell}(X,d) 
\times E^\vee$. 
Suppose moreover that there exists a section $s\in H^0(C_K,f^*E^\vee)$ over 
the field $K$ of fractions of $R$ such that $\ev_i(s) = v$ in $E^\vee(K)$, 
where $C_K = C_R \times_{\Spec(R)} \Spec(K)$.  
Then $(f,s)$ defines a map $\Spec(K) \to \overline{\mathcal{M}}_{g,\ell}(E^\vee,d)$ 
such that the following diagram commutes:
\[
\xymatrix
{\overline{\mathcal{M}}_{g,\ell}(E^\vee,d) 
\ar[r] & 
\overline{\mathcal{M}}_{g,\ell}(X,d) \times E^\vee \\  
\Spec(K) \ar[r] \ar[u]^{(f,s)}
& \Spec(R). \ar[u]_{(f,v)}\ar@{-->}[ul]
} 
\]
We will show that there exists a morphism 
$\Spec(R) \to \overline{\mathcal{M}}_{g,\ell}(E^\vee,d)$ 
which commutes with the maps in the above diagram. 
Since $\overline{\mathcal{M}}_{g,\ell}(E^\vee,d)$ 
is a subfunctor of $\overline{\mathcal{M}}_{g,\ell}(X,d) 
\times E^\vee$, it suffices to show the existence of a 
morphism $\Spec(R) \to \overline{\mathcal{M}}_{g,\ell}(E^\vee,d)$ 
which makes the upper-right triangle commutative, 
i.e.~$v$ is the image of a section 
in $H^0(C_R,f^*E^\vee)$. Let $\pi \colon C_R \to \Spec (R)$ denote  
the structure map and $x_i \colon \Spec(R) \to C_R$ denote the 
$i$-th marking. 
Note that the composition $\Spec(R) \xrightarrow{v} E^\vee \to X$ 
coincides with $f\circ x_i$, since the two maps coincide 
when we compose them with $\Spec(K) \to \Spec(R)$ 
(by the existence of $s$) 
and by the separatedness of $X$. 
Thus $v$ defines a section of $x_i^*f^* E^\vee$, which 
we denote again by $v$. 
We need to show that $v$ is in the image of 
$R^0\pi_* f^*E^\vee \to x_i^*f^*E^\vee$. 
Let $p\in \Spec(R)$ denote the unique closed point  
and let $k(p)$ be the residue field at $p$. 
We claim that the maps $R^0\pi_* f^*E^\vee \otimes_R k(p) 
\to H^0(C_p, f^*E^\vee)$, 
$H^0(C_p,f^*E^\vee) \to (x_i^*f^*E^\vee) \otimes_R k(p)$ 
are injective.  The injectivity of the latter map has been 
shown. To see the injectivity of the former, we take the 
so-called Grothendieck complex \cite[\S 5, p.46]{Mumford:abelian}: a complex 
$G^0\to G^1$ of finitely generated free $R$-modules 
such that the sequences 
\begin{align*} 
\begin{CD} 
0@>>> R^0 \pi_*f^*E^\vee @>>> G^0 @>{d^0}>> G^1 \\
0@>>> H^0(C_p, f^*E^\vee) @>>> G^0\otimes_R k(p) @>>> G^1 \otimes_R k(p) 
\end{CD} 
\end{align*} 
are exact. Since $R$ is a PID, the image of $d^0$ is a free $R$-module. 
Therefore $\Tor_1^R(\im d^0, k(p))=0$ and we obtain the exact 
sequence: 
\[
\begin{CD}
0 @>>> (R^0 \pi_* f^*E^\vee) \otimes_R k(p) @>>> G^0\otimes_R k(p) 
@>>> (\im d^0)\otimes_R k(p) @>>> 0.  
\end{CD} 
\]
Now the claim follows. The claim implies that $R^0\pi_*f^* E^\vee 
\to x_i^* f^* E^\vee$ is injective at the fiber of $p$. 
Then it follows that the cokernel $M$ of $R^0\pi_*f^* E^\vee 
\to x_i^* f^* E^\vee$ is a free $R$-module. 
In fact, let $N$ be the image of $R^0\pi_* f^* E^\vee \to x_i^*f^* E^\vee$; 
then the inclusion $N \subset x_i^*f^*E^\vee$ induces 
an injection $N\otimes_R k(p) \to (x_i^*f^*E^\vee) \otimes_R k(p)$. 
Because $x_i^*f^*E^\vee$ is a free $R$-module, we 
have the exact sequence 
\[
\xymatrix{ 
0 \ar[r] & 
\Tor_1^R(M,k(p)) \ar[r] & 
N \otimes_R k(p) \ar[r] & 
x_i^*f^*E^\vee \otimes_R k(p) \ar[r] &  
M\otimes_R k(p) \ar[r] & 0. 
}  
\]
Therefore $\Tor_1^R(M,k(p)) = 0$; thus $M$ is free. 
We know by assumption that the image of $v$ 
in $M\otimes_R K$ vanishes. Thus $v$ has to vanish 
in $M$. The conclusion follows. 
\end{proof} 

\begin{rem} 
\label{rem:concavity} 
The concavity of $E^\vee$ implies the convexity of $E$. 
This can be proved as follows. 
For a stable map $f\colon C\to X$ of genus zero, we have 
$H^1(C,f^*E) = H^0(C,f^*E^\vee \otimes \omega_C)^\vee$ 
by Serre duality, where $\omega_C$ is the dualizing sheaf on 
$C$. Suppose that $E^\vee$ is concave.  
Since the degree of $\omega_C$ on a tail component of 
$C$ is negative, a section of $f^* E^\vee \otimes \omega_C$  
has to vanish on tail components and defines a section of 
$f^*E^\vee \otimes \omega_{C'}$ where $C'$ is obtained 
from $C$ by removing all its tail components. By induction 
on the number of components, we can see that 
a section of $f^*E^\vee \otimes \omega_C$ vanishes 
and $H^0(C,f^*E^\vee \otimes\omega_C) = 0$. 
\end{rem} 

By the above lemma, we can define the quantum product 
of $E^\vee$ using the push-forward along the evaluation 
map $\ev_3$: 
\begin{equation} 
\label{eq:qprod_pushforward}
T_\alpha \bullet^{E^\vee}_\tau T_\beta 
= \sum_{d\in \eff(E^\vee)} \sum_{\ell=0}^\infty \frac{1}{\ell!} 
\ev_{3*}\left(\ev_1^*(T_\alpha) \ev_2^*(T_\beta) \prod_{j=4}^{\ell+3} 
\ev^*_j(\tau) \cap [\overline{\mathcal{M}}_{0,3+\ell}(E^\vee,d)]^{\vir} 
\right). 
\end{equation}
The quantum product $\bullet^{E^\vee}_\tau$ defines a flat quantum connection 
$\nabla^{E^\vee}$ as in Definition \ref{def:quantumD-mod}. 
\begin{defi} 
\label{defi:QDM_Evee}
The (non-equivariant) \emph{quantum $D$-module} of $E^\vee$ is a pair 
\[
\QDM(E^\vee) = \left(H^{\ev}(X)\otimes \cc[z][\![e^{\tau_2},\tau']\!], 
\nabla^{E^\vee}\right) 
\]
where the connection $\nabla^{E^\vee}$ is completed in the $z$-direction 
as in Remark \ref{rem:connection_z}: 
\[
\nabla^{E^\vee}_{z\partial_z} = z\partial_z - \frac{1}{z} 
(\frE^{E^\vee} \bullet_\tau) + 
\frac{\deg}{2} 
\]
where $\frE^{E^\vee} := \sum_{\alpha=0}^s (1- \frac{1}{2} \deg T_\alpha) 
t^\alpha T_\alpha + c_1(TX) - c_1(E)$ is the Euler vector field 
(note that this is the same as $\frE^{\eE}$ in \eqref{eq:tw_Eulervf}).   
Here the standard identification $H^{\ev}(E^\vee) \cong H^{\ev}(X)$ 
is understood. 
\end{defi} 

We conclude the following: 
\begin{prop} 
\label{prop:noneqlim_invEuler}
Suppose that $E$ is convex. Define the map 
$h\colon H^{\ev}(X) \to H^{\ev}(X)$ by 
\begin{equation} 
\label{eq:h} 
h(\tau) = \tau + \pi \sqrt{-1} c_1(E)
\end{equation} 
Then we have: 
\begin{enumerate} 
\item The non-equivariant limit of 
$\nabla^{\eqeEi}$ exists and coincides with $\nabla^{E^\vee}$. 
\item The non-equivariant limit of $\nabla^{(\eqeulerstar, E^\vee)}$ exists 
and coincides with $h^* \nabla^{E^\vee}$.   
\end{enumerate} 
\end{prop} 
\begin{proof} 
Part (1) follows from Proposition \ref{prop:GW,tot,space,GW,twist,X} 
and the existence of the non-equivariant quantum product for $E^\vee$. 
To see part (2), notice a small difference between $\eqeulerstar$ 
and $\eqeulerinv$: 
for a vector bundle $G$ we have 
\begin{align*} 
\eqeulerstar(G) = \frac{1}{\eqeuler(G^\vee)}  = 
(-1)^{\rank G} \frac{1}{\ctop_{-\lambda}(G)}.  
\end{align*} 
Since the virtual rank of $(E^\vee)_{0,\ell,d}$ equals  
$\rank E + c_1(E)(d)$, we have: 
\begin{align*} 
\sum_{n=0}^\infty \frac{1}{n!}
\langle T_\alpha,T_\beta,T_\gamma,\tau,\dots,\tau
\rangle^{(\eqeulerstar,E^\vee)}_{0,n+3,d} 
\frac{T^\gamma}{\eqeulerstar(E^\vee)} 
& = \sum_{n=0}^\infty \frac{1}{n!} 
(-1)^{\rank E + c_1(E)(d)} 
\langle T_\alpha, T_\beta, T_\gamma,\tau,\dots,\tau 
\rangle^{(\ctop_{-\lambda}^{-1},E^\vee)}_{0,n+3,d}
\frac{T^\gamma}{\eqeulerstar(E^\vee)}
\\ 
&  = \sum_{n=0}^\infty 
\frac{1}{n!} \langle 
T_\alpha, T_\beta, T_\gamma,h(\tau),\dots,h(\tau)  
\rangle^{(\ctop_{-\lambda}^{-1},E^\vee)}_{0,n+3,d}  
\be_{-\lambda}(E^\vee) T^\gamma  
\end{align*} 
where we used the divisor equation in the second line. 
This implies that the $(\eqeulerstar,E^\vee)$-twisted quantum product 
is the pull-back of the $(\ctop_{-\lambda}^{-1},E^\vee)$-twisted 
quantum product by $h$. The conclusion follows by 
taking the non-equivariant limit $\lambda \to 0$. 
\end{proof} 

\begin{rem} 
The pairing $S_{(\eqeulerstar,E^{\vee})}$ does not have a non-equivariant limit. 
\end{rem} 
\begin{rem} 
\label{rem:fullyflat_Evee} 
We can define the fundamental solution for the quantum 
connection of $E^\vee$ 
using the push-forward along an evaluation map 
similarly to \eqref{eq:qprod_pushforward}. Therefore 
the fundamental solution $L_{\eqeEi}$ admits 
a non-equivariant limit $L^{E^\vee}$. 
Using the formula \eqref{eq:fundsol_divisor} and 
an argument similar to Proposition \ref{prop:noneqlim_invEuler}, 
we find that  
\[ 
L_{(\be_{-\lambda}^{-1},E^\vee)}(h(\tau),z)   = 
L_{(\eqeulerstar,E^\vee)}(\tau,z) \circ e^{-\pi \sqrt{-1} c_1(E)/z} 
\]
and thus 
\begin{equation} 
\label{eq:LEvee_Leulerstar} 
L^{E^\vee}(h(\tau),z) = \lim_{\lambda \to 0} 
L_{(\eqeulerstar,E^\vee)} (\tau,z) \circ e^{-\pi \sqrt{-1} c_1(E)/z}. 
\end{equation} 
Then $L^{E^{\vee}}(\tau,z)z^{-\frac{\deg}{2}}z^{c_{1}(TX)-c_{1}(E)}$ 
is a fundamental solution of $ \nabla^{E^{\vee}}$ 
including in the $z$-direction. 
\end{rem} 

 \subsection{Non-equivariant limit of quantum Serre duality}
\label{subsec:non-equivariant,euler,twist}

We will state a non-equivariant limit of Theorem \ref{thm:quantum,Serre,QDM} 
when $\bc$ is $\eqeuler$ and $E$ is a convex vector bundle. 
From \S\ref{subsec:non-eq,Euler,twist} and 
\S\ref{subsec:non-equiv-limit,inverse,euler}, 
the quantum $D$-modules $\QDM_{(\eqeuler,E)}(X)$ and
$\QDM_{(\eqeulerstar,E^{\vee})}(X)$ have non-equivariant limits, 
and the limits are respectively $\QDM_{\eE}(X)$ and $h^*\QDM(E^\vee)$. 
The map $f$ in \eqref{eq:f,map} also admits a non-equivariant 
limit by Lemma \ref{lem:equ,limit,f,inverse}.  
The quantum Serre pairing in Definition \ref{def:QS_pairing} 
has an obvious non-equivariant limit: 
\begin{equation*} 
S^{\QS}: \QDM_{\eE}(X) 
\times (h\circ\overline{f})^{*}\QDM(E^{\vee}) \to 
\mathbb{C}[z][\![e^{\tau_{2}},\tau']\!]
\end{equation*} 
defined by $S^{\QS}(u,v)=
\int_{X} u(-z)  \cup v(z)$.
  
\begin{thm}
\label{thm:euler,quantum,Serre,QDM} 
Let $E$ be a convex vector bundle on $X$. 
Let $h$ be the map in \eqref{eq:h} and let $\overline{f}$ be 
the map in Lemma \ref{lem:equ,limit,f,inverse}. 

{\rm (1)} The pairing $S^{\QS}$ is flat in the $\tau$-direction 
and is of weight $(-\dim X)$. 

{\rm (2)} The map $\e(E)\cup:\QDM_{\eE}(X) \to 
(h\circ\overline{f})^{*}\QDM(E^{\vee})$, 
$\alpha  \mapsto \e(E)\cup \alpha$ 
respects the quantum connection in the $\tau$-direction 
and is of weight $\rank E$, that is, 
\begin{align*} 
\nabla'_\alpha \e(E) &= \e(E) \nabla_\alpha \\ 
\nabla'_{z\partial_z} \e(E) &= \e(E)\nabla_{z\partial_z} + \rank(E) \e(E) 
\end{align*} 
for $\nabla' = (h\circ\overline{f})^* \nabla^{E^\vee}$ 
and $\nabla = \nabla^{\eE}$. 

{\rm (3)} The fundamental solutions in Remarks \ref{rem:fullyflat_Euler}, 
\ref{rem:fullyflat_Evee} satisfy the following 
relations: 
\begin{align*} 
\e(E) \circ L_{\eE}(\tau,z) & = L^{E^\vee}(h\circ \overline{f}(\tau),z) 
e^{\pi\sqrt{-1}c_1(E)/z} \circ \e(E) \\ 
\left(\gamma_1, e^{-\pi\sqrt{-1} c_1(E)/z} \gamma_2\right) & = \left(
L_{\eE}(\tau,-z) \gamma_1, L^{E^\vee}(h(\overline{f}(\tau)),z) 
\gamma_2 \right). 
\end{align*} 
where $(u,v) = \int_X u\cup v$ is the Poincar\'{e} pairing. 
\end{thm}

\begin{proof}[Proof of Theorem \ref{thm:euler,quantum,Serre,QDM}] 
Almost all the statements follow by taking the non-equivariant limit 
of Theorem \ref{thm:quantum,Serre,QDM}. 
Notice that part (3) follows from Theorem \ref{thm:quantum,Serre,QDM} (3), 
Remarks \ref{rem:fullyflat_Euler}, \ref{rem:fullyflat_Evee} 
and equation \eqref{eq:LEvee_Leulerstar}. 
What remains to show is the statement about weights 
of $S^{\QS}$ and $\e(E)$. 
Regarding $\frE^{\eE}$, $\frE^{E^\vee}$ 
(see \eqref{eq:tw_Eulervf}) as vector fields on $H^{\ev}(X)$, 
we can check that $(h\circ \overline{f})_* 
\frE^{\eE} = \frE^{E^\vee}$. 
Therefore: 
\[
\Gr :=  \nabla_{z\partial_z}^{\eE} + \nabla_{\frE^{\eE}}^{\eE}  
= z\partial_z + \frE^{\eE} + \frac{\deg}{2} 
= \left((h\circ\overline{f})^*\nabla^{E^\vee}\right)_{z\partial_z} + 
\left((h\circ \overline{f})^*\nabla^{E^\vee}\right)_{\frE^{\eE}}. 
\]
On the other hand, we can check that 
\[
(z\partial_z + \frE^{\eE}) S^{\QS}(u,v) 
- S^{\QS}(\Gr u, v)- S^{\QS}(u,\Gr v) = -(\dim X)  
S^{\QS}(u,v).  
\]
The flatness of $S^{\QS}$ in the $\frE^{\eE}$-direction 
shows that $S^{\QS}$ is of weight $-\dim X$. The discussion 
for $\e(E)$ is similar. 
\end{proof} 

Let $Z \subset X$ be the zero-locus of a transverse section of $E$ 
and let $\iota \colon Z \to X$ be the inclusion map.  
We consider the following conditions for $Z$: 
\begin{lem} 
\label{lem:cond_Z}
The following conditions are equivalent: 
\begin{enumerate} 
\item the Poincar\'{e} pairing on $H_{\amb}^{\ev}(Z) 
= \im(\iota^* \colon H^{\ev}(X) \to H^{\ev}(Z))$ is non-degenerate; 

\item we have the decomposition 
$H^{\ev}(Z) = \Ker \iota_* \oplus \im \iota^*$; 

\item $\iota^*$ induces an isomorphism 
$H^{\ev}(X)/\Ker(\e(E) \cup)  \cong H^{\ev}_{\amb}(Z)$. 

\end{enumerate} 
\end{lem} 
\begin{proof} 
(1) $\Rightarrow$ (2): it suffices to see that $\Ker \iota_* \cap \im \iota^* 
= \{0\}$. Suppose that $\alpha \in \Ker \iota_* \cap \im\iota^*$. 
Then for every $\iota^* \beta \in H_{\amb}^{\ev}(Z)$ we have 
$(\iota^* \beta, \alpha) = (\beta, \iota_* \alpha) = 0$. 
By assumption we have $\alpha=0$. 
(2) $\Rightarrow$ (3): we have $\iota^* \alpha=0$ if and only if 
$\iota_* \iota^* \alpha =\e(E)\cup \alpha = 0$. 
Therefore $\Ker(\iota^*) = \Ker(\e(E)\cup)$ and part (3) follows.  
(3) $\Rightarrow$ (1): since $(\iota^*\alpha,\iota^*\beta) = 
(\alpha,\iota_*\iota^*\beta) = (\alpha, \e(E)\cup \beta)$, 
the kernel of the Poincar\'{e} pairing on $H_{\amb}^{\ev}(Z)$ 
is $\iota^*(\Ker(\e(E)\cup))$, which is zero. 
\end{proof} 

\begin{rem} 
The conditions in Lemma \ref{lem:cond_Z} hold if $E$ is the direct 
sum of ample line bundles by the Hard Lefschetz theorem. They also hold 
if $X$ is a toric variety and $Z$ is a regular hypersurface 
with respect to a semiample line bundle $E$ on $X$ by a result of 
Mavlyutov \cite{mavlyutov_chiral_2000}. 
\end{rem}

\begin{cor} 
\label{cor:Z}
Suppose that $E$ is a convex vector bundle on $X$ 
and $Z\subset X$ be the zero-set of a regular section of $E$ 
satisfying one of the conditions in Lemma \ref{lem:cond_Z}. 
Then the morphism $\e(E)\cup \colon \QDM_{\eE}(X) \to 
(h\circ \overline{f})^*\QDM(E^\vee)$ in 
Theorem \ref{thm:euler,quantum,Serre,QDM} 
factors through $\QDM_{\amb}(Z)$ as: 
\begin{displaymath}
\xymatrix{\QDM_{\eE}(X) \ar@{->>}[d]_-{\iota^*} 
\ar[r]^-{\e(E)\cup} 
&  (h\circ\overline{f})^{*}\QDM(E^{\vee})\\ 
(\iota^*)^*\QDM_{\amb}(Z) \ar@{^{(}->}_-{\iota_*}[ru]}
  \end{displaymath}
In particular, $\iota_* \colon (\iota^*)^*\QDM_{\amb}(Z) 
\to (h\circ \overline{f})^* \QDM(E^\vee)$ respects 
the quantum connection in the $\tau$-direction and 
is of weight $\rank E$. 
\end{cor}
\begin{proof} 
We already showed that $\iota^*$ is a morphism of flat connections 
in Proposition \ref{prop:QDM,amb}. It suffices to invoke the 
factorization of the linear map $\e(E)\cup$: 
\begin{equation}
\label{eq:6}
\begin{aligned} 
\xymatrix{H^{\ev}(X) \ar@{->>}[d]_-{\iota^{*}} \ar[r]^-{\e(E)\cup} 
& H^{\ev}(X)\\
      H^{\ev}_{\amb}(Z)\cong H^{\ev}(X)/\Ker(\e(E)\cup) 
\ar@{^{(}->}^-{\iota_{*}}[ru]}
\end{aligned}
\end{equation} 
\end{proof} 

 \begin{rem}
Recall that for a general non-compact space, the Poincar\'e
duality pairs the cohomology with the cohomology with compact
support.  This analogy leads us to think of $\QDM_{\eE}(X)$ as
the quantum $D$-module with compact support of the total
space $E^{\vee}$. 
\end{rem}

\begin{rem} 
It would be interesting to study if $\iota_*$ always defines a 
morphism of quantum $D$-modules without assuming 
the conditions in Lemma \ref{lem:cond_Z}. 
\end{rem}

\section{Quantum Serre duality and integral structures}
\label{sec:quant-serre-integr}
In this section we study a relation between quantum Serre duality 
for the Euler-twisted theory and 
the $\hGamma$-integral structure studied in 
\cite{Iritani-2009-Integral-structure-QH, 
Katzarkov-Pantev-Kontsevich-ncVHS, iritani_quantum_2011, 
Mann-Mignon-2011-QDM-lci-nef}. 
The $\hGamma$-integral structure is a lattice in the space of 
flat sections for the quantum connection, which is isomorphic 
to the Grothendieck group $K(X)$ of vector bundles on $X$. 
After introducing a similar integral structure in the Euler-twisted 
theory, we see that the quantum Serre pairing is identified with 
the Euler pairing on $K$-groups, and that the morphisms 
of flat connections in Corollary \ref{cor:Z} are induced by 
natural maps between $K$-groups. 
In this section, the Novikov variable $Q$ is specialized 
to one, see \S \ref{subsec:specialization}.

\begin{recall} 
Recall the classical self-intersection formula in $K$-theory.
Let $j:X \hookrightarrow Y$ be a closed embedding with 
normal bundle $N$ between quasi-compact and
quasi-separated schemes. In Theorem 3.1 of 
\cite{Thomason-excess-intersection-K-theory}, Thomason proves 
that we have for any $[V]\in K(X)$
\begin{align}\label{eq:3}
  j^{*}j_{*}[V]=[\lambda_{-1}N^\vee]\cdot[V]
\end{align}
where
\begin{displaymath}
  [\lambda_{-1}N^\vee] := \sum_{k\geq 0} (-1)^{k}[\wedge^{k}N^\vee] \in K(X).
\end{displaymath}
\end{recall}

\begin{defi} 
For a vector bundle $G$ with Chern roots $\delta_{1}, \ldots, \delta_{r}$, 
we define the $\hGamma$-class to be 
\begin{displaymath}
  \hGamma(G)=\prod_{i=1}^{r} \Gamma(1 + \delta_{i}).
\end{displaymath}
We also define a $(2\pi \sqrt{-1})$-modified 
Chern character by: 
\[
\Ch(G) = (2\pi \sqrt{-1})^{\frac{\deg}{2}} 
\ch(G) = \sum_{i=1}^{r} e^{2\pi \sqrt{-1} \delta_i}.  
\]
\end{defi}

Suppose that $E$ is a convex vector bundle on $X$. 
Let $Z\subset X$ be a submanifold cut out by a transverse 
section $s$ of $E$. 
For the (twisted) quantum connection $\nabla$, 
we write $\Ker \nabla$ for the space of flat sections: 
\begin{align*} 
\Ker \nabla^{\eE} &= \left\{ s \in 
H^{\ev}(X) \otimes \cc[z^\pm][\![e^{\tau_2},\tau']\!][\log z] : 
\nabla^{\eE} s = 0 \right\}, \\ 
\Ker \nabla^{Z} & = \left\{ s \in 
H^{\ev}_{\amb}(Z) \otimes \cc[z^\pm][\![e^{\tau_2},\tau']\!][\log z] : 
\nabla^Z s = 0 \right\}, \\ 
\Ker \nabla^{E^\vee} 
& = \left\{ s \in 
H^{\ev}(X) \otimes \cc[z^\pm][\![e^{\tau_2},\tau']\!][\log z] : 
\nabla^{E^\vee} s = 0 \right\}.   
\end{align*} 
\begin{defi} The \emph{$K$-group framing} is a map from 
a $K$-group to the space of flat sections defined as follows: 

\noindent 
(1) for the twist $\eE$, the $K$-group framing 
$Z^{\eE} \colon K(X) \to \Ker \nabla^{\eE}$ is: 
\begin{align*}
Z^{\eE}(V) &= \frac{1}{(2\pi\sqrt{-1})^{\dim X -\rank E}} 
L_{\eE}(\tau,z) 
z^{-\frac{\deg}{2}}z^{c_{1}(TX)-c_1(E)}
\frac{\hGamma(TX)}{\hGamma(E)} \Ch(V); \\
\intertext
{(2) for a smooth section $Z\subset X$ of $E$, the $K$-group 
framing $Z^{\amb} \colon K_{\amb}(Z) \to \Ker \nabla^Z$ is:} 
Z^{\amb}(V) & = 
\frac{1}{(2\pi\sqrt{-1})^{\dim Z}} L^Z(\tau,z) z^{-\frac{\deg}{2}} 
z^{c_1(TZ)} \hGamma(TZ) \Ch(V); \\
\intertext
{(3) for the total space $E^{\vee}$, the $K$-group framing 
$Z^{E^{\vee}} \colon K(X) \to \Ker \nabla^{E^{\vee}}$ is:}  
Z^{E^\vee}(V) &= \frac{1}{(2\pi \sqrt{-1})^{\dim E^\vee}} 
L^{E^{\vee}}
(\tau,z)z^{-\frac{\deg}{2}}z^{c_{1}(TX)- c_1(E)}
\hGamma(TE^{\vee})\Ch(V). 
\end{align*}
where $K_{\amb}(Z) = \im(\iota^* \colon K(X) \to K(Z))$ and 
$L^Z, L^{E^\vee}$ are the fundamental solutions 
for $Z$ and $E^\vee$ respectively. 
Recall from Remarks \ref{rem:fullyflat}, 
\ref{rem:fullyflat_Euler}, \ref{rem:fullyflat_Evee} 
that these formula define a section which is flat in both $\tau$ and $z$. 
\end{defi} 

 \begin{prop}\label{prop:compa,integral}
For any vector bundles $V,W$ on $X$, we have 
\begin{displaymath}
 \chi(V\otimes W^{\vee})=
(-2\pi \sqrt{-1} z)^{\dim X} {S}^{\QS}
\left(Z^{(\e,E)}(V)(\tau, e^{\pi \sqrt{-1}}z), 
Z^{E^{\vee}}(W)(h\circ\overline{f}(\tau),z)\right)
\end{displaymath}
where $\chi(V\otimes W^\vee) = 
\sum_{i=0}^{\dim X} (-1)^i \dim \Ext^i(W,V)$ 
is the holomorphic Euler characteristic. 
\end{prop}
\begin{proof} 
This is analogous to \cite[Proposition 2.10]{Iritani-2009-Integral-structure-QH}. 
Since the pairing $z^{\dim X} S^{\QS}$ is flat, the right-hand side 
is constant with respect to $\tau$ and $z$. Evaluating the right-hand side 
at $z=1$, we obtain 
\[
\frac{1}{(-2\pi \sqrt{-1})^{\dim X}} 
\left(L_{\eE}(\tau,-1)e^{-\pi\sqrt{-1} \frac{\deg}{2}} 
e^{\pi\sqrt{-1}(c_1(TX)-c_1(E))} \gamma_1, 
L^{E^\vee}(h\circ \overline{f}(\tau),1) \gamma_2 
 \right) 
\]
with $\gamma_1 = \frac{\hGamma(TX)}{\hGamma(E)} \Ch(V)$, 
$\gamma_2 = \hGamma(TX) \hGamma(E^\vee) \Ch(W)$, 
where $(\cdot,\cdot)$ is the Poincar\'{e} pairing on $X$. 
By Theorem \ref{thm:euler,quantum,Serre,QDM} (3), 
we find that this equals 
\[
\frac{1}{(-2\pi\sqrt{-1})^{\dim X}} 
\left ( e^{-\pi\sqrt{-1} \frac{\deg}{2}} 
e^{\pi\sqrt{-1}(c_1(TX)-c_1(E))} \gamma_1, 
e^{-\pi \sqrt{-1} c_1(E)}\gamma_2 
 \right).
\]
Since the adjoint of $\frac{\deg}{2}$ is $\dim X - \frac{\deg}{2}$, 
this is:  
\[
\frac{1}{(2\pi\sqrt{-1})^{\dim X}} 
\left (
e^{\pi\sqrt{-1}c_1(TX)} \gamma_1, 
e^{\pi\sqrt{-1}\frac{\deg}{2}}\gamma_2 
 \right). 
\]
Using the following identities:
\begin{align*}
&e^{\pi\sqrt{-1}\frac{\deg}{2}} 
\Gamma(1+\delta )=\Gamma(1-\delta) \quad \text{and} \quad 
e^{\pi\sqrt{-1}\frac{\deg}{2}}\ch(W)=\ch(W^{\vee}) \\ 
&(2\pi\sqrt{-1})^{-\dim X} 
\int_{X}\gamma= \int_{X}(2\pi\sqrt{-1})^{-\frac{\deg}{2}}\gamma\\
&(2\pi\sqrt{-1})^{-\frac{\deg}{2}}\Gamma(1+\delta)=
(2\pi\sqrt{-1})^{\frac{\deg}{2}} 
\Gamma\left(1+\textstyle\frac{\delta}{2\sqrt{-1}\pi}\right)
\end{align*}
with $\delta$ a degree-two cohomology class, 
we deduce that the right hand side of the proposition is
\begin{align*}
\int_{X}\ch(V\otimes
  W^{\vee}) e^{\rho/2}
  \prod_{i=1}^{n}
\Gamma\left(1+\tfrac{\rho_{i}}{2\sqrt{-1}\pi}\right)
\Gamma\left(1-\tfrac{\rho_{i}}{2\sqrt{-1}\pi}\right)
\end{align*}
where $\rho_{1}, \ldots ,\rho_{n}$ are the Chern roots of $TX$ 
and $\rho = c_1(TX) = \rho_1 + \cdots + \rho_n$.  
Finally, we use
$\Gamma(x)\Gamma(1-x)=\pi/\sin(\pi x)$ to get
\begin{align}\label{eq:53}
  e^{\rho/2}
  \prod_{i=1}^{n}\Gamma\left(1+\tfrac{\rho_{i}}{2\sqrt{-1}\pi}\right)\Gamma\left(1-\tfrac{\rho_{i}}{2\sqrt{-1}\pi}\right)= \Td(TX).
\end{align}
We conclude the proposition by the theorem of Hirzebruch-Riemann-Roch.
\end{proof}

The following proposition shows that the integral structures are compatible 
with the diagram in Corollary \ref{cor:Z}: 
\begin{prop}\label{prop:integral} 
Let $E$ be a convex vector bundle and $Z \subset X$ be a 
submanifold cut out by a regular section of $E$. 
Let $\iota \colon Z \hookrightarrow X$, $j\colon X \hookrightarrow 
E^\vee$ denote the natural inclusions.  
Assume that $Z$ satisfies one of the conditions in 
Lemma \ref{lem:cond_Z}. 
Then the diagram in Corollary \ref{cor:Z} can be extended 
to the following commutative diagram
\begin{displaymath}
\xymatrix{
K(X) \ar@/_3pc/[dddr]^-{\iota^{*}} \ar[rd]^{Z^{\eE}} 
\ar[rrrr]^-{j^{*}j_{*} }
& & & & K(X) \ar[dl]_-{Z^{E^{\vee}}} \\
&\Ker \nabla^{\eE}\ar[d]^-{\iota^{*}} 
\ar[rr]^-{c(z) \e(E)\cup}
&& \Ker (\overline{f}\circ h)^{*}\nabla^{E^{\vee}}&  \\
&\Ker (\iota^*)^*\nabla^{\amb} \ar@{^{(}->}[rru]_-{c(z) \iota_{*}} & \\
&K_{\amb}(Z) \ar[u]_-{Z^{\amb}} 
\ar@/_4pc/[rrruuu]_{\phantom{ABC}(-1)^{\rank E}\det(E) \otimes \iota_{*}}}
  \end{displaymath}
  where $c(z) = 1/(-2\pi\sqrt{-1}z)^{\rank E}$. 
\end{prop}

\begin{proof} 
We first prove that the top square is commutative.
Recall the following equation 
from part (3) of Theorem \ref{thm:euler,quantum,Serre,QDM}: 
\begin{align*}
 \e(E)L_{\eE}(\tau,z)&=
L^{E^{\vee}}(h\circ\overline{f}(\tau),z) 
e^{\pi \sqrt{-1}c_{1}(E)/z}\e(E). 
\end{align*}
So it remains to prove that for any $V\in K(X)$, we have
\begin{align*}
(-2\pi\sqrt{-1})^{\rank E} 
e^{\pi\sqrt{-1}c_{1}(E)} \e(E) \hGamma(TX)
\hGamma(E)^{-1} \Ch(V) 
=\hGamma(TX) \hGamma(E^{\vee}) 
\Ch(j^*j_{*}V)
\end{align*}
This follows from  \eqref{eq:53} applied  to the vector bundle $E$ and  from 
\begin{displaymath}
  \ch j^{*}j_{*}V= \e(E^\vee) \Td(E^\vee)^{-1}\ch(V),\quad 
\text{see \eqref{eq:3}}.
\end{displaymath}
The commutativity of the left square follows from the properties 
of the $\hGamma$-class and the following facts 
(see \cite[Proposition 2.4]{iritani_quantum_2011} for the 
second property): 
\begin{align*}
& \xymatrix{0\ar[r]&TZ\ar[r]&\iota^{*}TX\ar[r]&\iota^{*}E\ar[r]&0} 
\text{ is exact;}\\
& L^{Z}(\iota^* \tau,z)\iota^{*}\gamma=\iota^{*} 
\left( L_{\eE}(\tau,z) \gamma \right), 
\quad \forall\gamma\in H^{\ev}(X). 
\end{align*}
The identity $j^{*}j_{*}=(-1)^{\rank E} \det(E)\otimes 
\iota_{*}\iota^{*}$ implies that the right square is commutative. 
\end{proof}

 \section{Quantum Serre duality and abstract Fourier-Laplace transform}
\label{sec:QS,KX,abstract,Fourier}

In this section we study quantum Serre duality with respect to 
the anticanonical line bundle $K_X^{-1}$. 
We consider the $(\e, K_X^{-1})$-twisted quantum 
$D$-module of $X$ and the quantum $D$-module of 
the total space of $K_X$.  
On the small quantum cohomology locus $H^2(X)$, 
we identify these quantum $D$-modules with 
Dubrovin's  second structure connections with 
different parameters $\sigma$. 
We show that the duality between them is given by the 
second metric $\check{g}$. 
Throughout the section, we assume that the anticanonical 
class $-K_X = c_1(X)$ of $X$ is nef and the Novikov variable  
is specialized to one (\S\ref{subsec:specialization}). 
We also set $n:=\dim_\cc X$ and $\rho:=c_1(X)$. 

\subsection{Convergence assumption}
\label{subsec:convergence} 
In this section \S \ref{sec:QS,KX,abstract,Fourier}, 
we assume certain analyticity of quantum cohomology of $X$. 
In \S \ref{subsec:qconn_parameter}--\ref{subsec:second_str_conn}, 
we assume that the big quantum cohomology of $X$ is convergent, 
that is, the quantum product 
(with Novikov variables specialized to one, see \S \ref{subsec:specialization}) 
\[
T_\alpha \bullet_{\tau} T_\gamma \in \cc[\![e^{\tau_2},\tau']\!] 
= \cc[\![e^{t^1},\dots,e^{t^r},t^0,t^{r+1},\dots,t^s]\!] 
\]
converges on a region $U\subset H^{\ev}(X,\cc)$ of the form: 
\[
U = \left\{\tau \in H^{\ev}(X,\cc) : |e^{t^i}| < \epsilon \ (1\le i \le r), 
\  |t^j| < \epsilon \ (r+1 \le j\le s) \right\}.
\]
For the main results in this section, we only need the convergence 
of the \emph{small} quantum product. 
This means that the quantum product $T_\alpha \bullet_\tau T_\beta$ 
restricted to $\tau = \tau_2$ to lie in $H^2(X,\cc)$ 
converges on a region $U_{\rm sm}\subset H^2(X,\cc)$ of the form 
\begin{equation} 
\label{eq:Usmall}
U_{\rm sm} = 
\left\{\tau_2 \in H^2(X,\cc) : |e^{t^i}| < \epsilon \ (1\le i\le r) \right\}. 
\end{equation} 
When $X$ is Fano, i.e.~if $-K_X$ is ample, the convergence 
of small quantum cohomology 
is automatic because the structure constants  
are polynomials in $e^{t^1},\dots, e^{t^{r}}$ for degree reason.

\subsection{Quantum connection with parameter $\sigma$}  
\label{subsec:qconn_parameter}
We introduce a variant of the quantum connection 
parametrized by a complex number $\sigma$. 
Consider the trivial vector bundle 
\[
F = H^{\ev}(X) \times (U\times \cc_z) 
\]
over $U\times \cc_z$, where $U$ is the convergence domain 
of the big quantum product in \S \ref{subsec:convergence}, 
and define a meromorphic flat connection $\nabla^{(\sigma)}$ 
of $F$ 
by the formula (cf.~Definition \ref{def:quantumD-mod} and 
Remark \ref{rem:connection_z})
\begin{align*} 
\nabla_\alpha^{(\sigma)} & = 
\partial_\alpha + \frac{1}{z} (T_\alpha\bullet_\tau) \\ 
\nabla_{z\partial_z}^{(\sigma)} & = 
z \partial_z - \frac{1}{z} (\frE\bullet_\tau) + 
\left(\mu - \frac{1}{2} -\sigma \right) 
\end{align*} 
where $\mu$ is an endomorphism of $H^{\ev}(X)$ defined by 
\[
\mu(T_\alpha) = \left(|\alpha| - \frac{n}{2}\right) T_\alpha \qquad 
\text{with} \ \ |\alpha| = \frac{1}{2} \deg T_\alpha, \ 
n=\dim_\cc X.  
\]
Let $(-) \colon U\times \cc_z \to U \times \cc_z$ denote the map 
sending $(\tau,z)$ to $(\tau,-z)$. We note the following facts: 

\begin{prop}[{\cite[Theorem 9.8 (c)]{Hfm}}] 
The $\cO_{U\times \cc_z}$-bilinear pairing 
\begin{align*} 
g \colon (-)^*\big(F,\nabla^{(\sigma)}\big) 
\times \big(F,\nabla^{(-1-\sigma)}\big) & \to \cO_{U\times \cc_z}
\end{align*} 
defined by $g(T_\alpha,T_\beta) = \int_X T_\alpha \cup T_\beta$ 
is flat.  
\end{prop} 

\begin{prop}[see e.g.~{\cite[Proposition 2.4]{Iritani-2009-Integral-structure-QH}}] 
Let $L(\tau,z)$ be the fundamental solution for the quantum 
connection of $X$ from Proposition \ref{prop:prop,QDM,S,flat,...} 
(with $\bc=1$, $E=0$). 
We have that $ L(\tau,z)z^{-(\mu-\frac{1}{2}-\sigma)}z^{c_{1}(TX)}$
is a fundamental solution of $\nabla^{(\sigma)}$ including in the $z$-direction. 
\end{prop} 

\begin{rem} 
The variable $z$ in this paper corresponds to $z^{-1}$ 
in Hertling's book \cite[\S 9.3]{Hfm}. 
For convenience of the reader, 
we made a precise link of notation with the book of Hertling: 
\begin{equation}
\label{eq:not,hertling}
\cU= \frE\bullet_{\tau}, \quad  
D=2-n, \quad   
\cV=-\mu-\frac{n}{2}. 
\end{equation}
\end{rem}

Using the divisor equation, the inverse of the fundamental solution 
$L(\tau,z)$ for $X$ (see \eqref{eq:inverse_fundsol})  
can be written in the form: 
\begin{align}
\label{eq:83}
L(\tau,z)^{-1}T_\alpha=e^{\tau_{2}/z}
\left(T_\alpha + \sum_{
\substack{(d,l)\neq (0,0) \\ \beta\in\{0, \ldots,s\}}}
\left\langle T_\alpha,\tau', \ldots ,\tau',\frac{T_{\beta}}{z-\psi}
\right\rangle_{0,l+2,d}e^{\tau_{2}(d)}\frac{T^{\beta}}{l!}\right) 
\end{align}
Denote by $K^{(\sigma)}_\alpha$ the $\alpha$th column of 
the inverse fundamental solution matrix for $\nabla^{(\sigma)}$: 
\[ 
  K^{(\sigma)}_{\alpha}(\tau,z):=
z^{-c_{1}(TX)}z^{\mu-\frac{1}{2}-\sigma}
L(\tau,z)^{-1}T_{\alpha}.
\]
If we restrict $\tau$ to lie in $H^{2}(X)$, 
we have the following expression. 

\begin{lem}\label{lem:inverse,funda,sol}
For any $\alpha\in\{0, \ldots ,s\}$ and $\tau_2 \in H^2(X)$, we have
\begin{align*}
    K^{(\sigma)}_{\alpha}(\tau_{2},z)&= 
    \sum_{d\in \eff(X)}N_{\alpha,d}(1) 
\frac{e^{\tau_{2}+\tau_{2}(d)}}
{z^{\rho+\rho(d)-|\alpha|+\frac{n+1}{2}+\sigma}} 
 \end{align*}
where $\rho = c_1(X)$ and 
\begin{equation}
\label{eq:N_alphad}
N_{\alpha,d}(z) := 
\begin{cases} 
  \sum_{\beta=0}^s \llangle T_\alpha, 
\frac{T_\beta}{z-\psi} 
\rrangle_{0,2,d}{T^\beta} &\text{if $d\neq 0$;}\\
T_\alpha &\text{if $d= 0$.}
\end{cases}
\end{equation}
\end{lem}

\begin{proof} Restricting to $H^2(X)$ means setting 
$\tau'=0$ in \eqref{eq:83}. 
For $d\neq 0\in \eff(X)$, we have 
\begin{align*}
N_{\alpha,d}(z) = \sum_{\beta=0}^{s}\llangle
 T_{\alpha}, \frac{T_{\beta}}{z-\psi}\rrangle_{0,2,d}T^{\beta}
=\sum_{k\geq 0}\sum_{\beta=0}^{s}\llangle
  T_{\alpha},\psi^{k}T_{\beta}\rrangle_{0,2,d}
\frac{T^{\beta}}{z^{k+1}}. 
\end{align*}
By the degree axiom for Gromov-Witten invariants, 
only the term with $ k=n-|\beta|- |\alpha| +\rho(d)-1$ 
contributes. Therefore $z^{\mu}N_{\alpha,d}(z)= 
N_{\alpha,d}(1) z^{-(\rho(d)+\frac{n}{2}-|\alpha|)}$.  
Noting that $z^{\mu}\circ e^{\tau_{2}/z}=e^{\tau_{2}} \circ z^{\mu}$, 
we deduce the formula of the lemma. 
\end{proof}

\subsection{The second structure connection} 
\label{subsec:second_str_conn} 

We introduce the \emph{second structure connection} 
\cite[lecture 3]{Dtft}, \cite[\S 2.3]{Dubrovin-Almost-Frob-2004}, 
\cite[II, \S 1]{Mfm}, \cite[\S 9.2]{Hfm}. 
Let $x$ be the variable Laplace-dual to $z^{-1}$ and 
let $\cc_x$ denote the complex plane with co-ordinate $x$. 
Consider the trivial vector bundle 
\[
\check{F} = H^{\ev}(X)\times (U\times \cc_x)
\]
over $U\times \cc_x$. 
The \emph{second structure connection} is a meromorphic flat connection 
on the bundle $\check{F}$ defined by 
\begin{align}
\label{eq:second_str_conn} 
\begin{split} 
\chnabla^{(\sigma)}_{\alpha}&=\partial_{{\alpha}}+
\left(\mu-\frac{1}{2}-\sigma\right)
\left( (\frE\bullet_\tau)-x \right)^{-1} (T_\alpha\bullet_\tau) \\
\chnabla^{(\sigma)}_{\partial_{x}}&=
\partial_{x} -\left(\mu-\frac{1}{2}-\sigma\right)
\left((\frE\bullet_\tau) -x\right)^{-1}.
\end{split} 
\end{align}
The connection has a singularity along the divisor 
$\Sigma \subset U\times \cc$: 
\begin{displaymath}
\Sigma:=
\left\{(\tau,x)\in U \times \cc_{x} \mid 
\det((\frE\bullet_{\tau})-x) = 0 \right \}.
\end{displaymath} 
The second structure connection has an invariant pairing 
called the \emph{second metric} (or the \emph{intersection form}). 

\begin{prop}[{\cite[Theorem 9.4.c]{Hfm}}]
The $\cO_{U\times \cc_x}$-bilinear pairing 
\[
  \check{g}: (\check{F},\chnabla^{(\sigma)})
  \times (\check{F},\chnabla^{(-\sigma)}) \to 
  \cO_{ U\times\cc_{x}}(\Sigma)  
  \]
defined by $\check{g}(T_{\alpha},T_{\beta})
= \int_{X}T_{\alpha}
\cup ((\frE\bullet_{\tau})-x)^{-1}T_{\beta}$ 
is flat. This is called the \emph{second metric}. 
\end{prop}  

We now explain how the second structure connection $\chnabla^{(\sigma)}$ 
arises from the Fourier-Laplace transformation of the 
quantum connection $\nabla^{(\sigma-1)}$ (see \cite[V]{Sdivf}, 
\cite[1.b]{DSgm1}). 
Consider the module $M=H^{\ev}(X) \otimes \cO_U[z]$ 
of sections of the trivial bundle $F$ 
which are polynomials in $z$. 
The quantum connection $\nabla^{(\sigma-1)}$ 
equips $M[z^{-1}]$ with the structure of an  
$\cO_U\langle \partial_\alpha, z^\pm,\partial_z \rangle$-module 
by the assignment: 
\[
\partial_z \mapsto\nabla_{\partial_z}^{(\sigma-1)} 
\qquad 
\partial_\alpha \mapsto \nabla_\alpha^{(\sigma-1)}. 
\]
Consider the isomorphism of the rings of differential operators: 
\begin{align*} 
\cO_U\langle \partial_\alpha, z^{-1}, \partial_{z^{-1}}\rangle 
\cong 
\cO_U\langle \partial_\alpha, x, \partial_x \rangle 
\end{align*} 
sending $\partial_{z^{-1}} = -z^2\partial_z$ to $x$ 
and $z^{-1}$ to $-\partial_x$. Via this isomorphism,  
we may regard $M[z^{-1}]$ as an 
$\cO_U[x] \langle \partial_\alpha, \partial_x \rangle$-module. 
This is called the \emph{abstract Fourier-Laplace transform}. 
The subset $M\subset M[z^{-1}]$ is closed under 
the action of $x = -z^2 \partial_z$, and thus 
becomes an $\cO_U[x]$-submodule of $M[z^{-1}]$. 
Note that $M[z^{-1}]$ is generated by $M$ 
over $\cO_U\langle x, \partial_x \rangle$ 
since $z^{-1} =  - \partial_x$. 
Regard $T_\alpha\in H^{\ev}(X)$ as an element of $M$. 
Under the abstract Fourier-Laplace transformation, 
we have 
\begin{align*} 
(\partial_x x) \cdot  T_\alpha &= \nabla^{(\sigma-1)}_{z\partial_z} T_\alpha 
 = \partial_x \cdot( \frE\bullet_\tau T_\alpha ) + 
\left(\mu +\frac{1}{2} - \sigma\right)T_\alpha \\ 
\partial_\beta \cdot T_\alpha & = \nabla^{(\sigma-1)}_\beta T_\alpha 
= -\partial_x  \cdot(T_\beta \bullet_\tau T_\alpha) 
\end{align*} 
Regarding $(\frE\bullet_\tau)$, $\mu$, $(T_\beta\bullet_\tau)$ 
as matrices written in the basis $\{T_\alpha\}$, we obtain 
\begin{align*} 
[\partial_x T_0,\dots,\partial_x T_s] (x - \frE\bullet_\tau) 
& = [T_0,\dots,T_s] \left(\mu - \frac{1}{2} - \sigma\right) \\ 
[\partial_\beta T_0,\dots,\partial_\beta T_s]  
& = -[\partial_x T_0,\dots, \partial_x T_s] (T_\beta \bullet_\tau)
\end{align*} 
Inverting $(x- \frE\bullet_\tau)$ in the first equation, 
we obtain the connection matrices for the second structure connection 
$\chnabla^{(\sigma)}$. 
In other words, writing $\cO(\check{F})$ for the sheaf of 
holomorphic sections of $\check{F}$ which are polynomials in $x$, 
the natural $\cO_U[x]$-module map 
\begin{equation} 
\label{eq:2ndconn_FL}
\cO(\check{F}) \to M,\qquad T_\alpha \mapsto T_\alpha 
\end{equation}
intertwines the meromorphic connection $\chnabla^{(\sigma)}$ 
on $\check{F}$ 
with the action of $\partial_x$, $\partial_\alpha$ 
on $M$ after inverting $\det(x-\frE\bullet_\tau)\in \cO_U[x]$. 
On the other hand, $\{T_\alpha\}$ does not always give an 
$\cO_U[x]$-basis of $M$ and the map \eqref{eq:2ndconn_FL} 
is not always an isomorphism. 
A sufficient condition for the map \eqref{eq:2ndconn_FL} 
to be an isomorphism is given by  
a result of Sabbah \cite{Sdivf}. 

\begin{prop}[{\cite[Proposition V.2.10]{Sdivf}}]
\label{prop:Sabbah}
Suppose that $\sigma\notin -\frac{n-1}{2} + \zz_{\ge 0}$ 
with $n= \dim_{\cc} X$. 
Then the second structure connection $(\check{F}, \chnabla^{(\sigma)})$ 
coincides with the abstract Fourier transform of the quantum connection 
$(F,\nabla^{(\sigma-1)})$, i.e.~the map \eqref{eq:2ndconn_FL} 
is an isomorphism. 
\end{prop}

\subsection{Fundamental solution for the second structure connection} 
We will henceforth restrict ourselves to the small quantum 
cohomology locus $H^2(X)$. 
We find an inverse fundamental solution for the second structure connection 
using a truncated Laplace transformation. 

\begin{defi}\label{defi:Laplace}  
Consider a cohomology-valued power series of the form: 
\[
K(z)=z^{-\gamma} \sum_{k} a_k z^{-k}
\]
with $a_k\in H^{\ev}(X)$ and $\gamma \in H^{\ev}(X)$, 
where $z^{-\gamma} = e^{-\gamma \log z}$. 
We assume that the exponent $k$ ranges over a subset 
of $\cc$ of the form $\{k_0,k_0+1,k_0+2,\dots\}$. 
Let $\ell$ be a complex number such that 
  \begin{itemize}
  \item $\ell -k_0\in \zz$ and, 
  \item $0\notin\{k_0+1,k_0+2, \ldots ,\ell -1\}$ if $k_0 \le \ell-2$. 
  \end{itemize}
We define the \emph{truncated Laplace transform} of $K(z)$ to be 
  \begin{align*}
    \Lap^{(\ell)}(K)(x):=\sum_{k} a_k x^{-\gamma-k-1}
\frac{\Gamma(\gamma+k+1)}{\Gamma(\gamma+\ell)}  
\end{align*}
where note that 
\[
\frac{\Gamma(\gamma+k+1)}{\Gamma(\gamma+\ell)}  
=
\begin{cases} 
(\gamma+\ell)(\gamma+ \ell+1) \cdots(\gamma+k) & 
\mbox{ if } k\geq \ell;  
\\
1& \mbox{ if } k= \ell -1;
\\
\frac{1}{(\gamma+k+1)(\gamma+k+2)\cdots(\gamma + \ell-1)}& 
\mbox{ if } k\leq \ell -2,
\end{cases}
\]
and the above condition for $\ell$ ensures that we do not have 
the division by $\gamma$ when $k\leq \ell-2$ and 
that this expression is well-defined. 
\end{defi}

The truncated Laplace transformation satisfies 
the following property: 
\begin{align}
\label{eq:FL_rules}
\begin{split}  
\Lap^{(\ell)}(z^{-1} K) 
& =(-\partial_{x})\Lap^{(\ell)}(K)\\
\Lap^{(\ell)}(-z^{2}\partial_{z}K)
& = \Lap^{(\ell)} (\partial_{z^{-1}} K) = x\Lap^{(\ell)}(K)
\end{split}
\end{align}

\begin{rem} 
Suppose that $K(z)$ is convergent for all $z\in \cc^\times$, 
$\Re(k_0)>-1$ and that we have an estimate $|K(z)|\leq C e^{M/z}$ 
over the interval $z\in (0,1)$ for some $C,M>0$. 
Then we can write 
the truncated Laplace transform as the actual Laplace 
transform: 
\[
\Lap^{(\ell)}(K)(x) = \frac{1}{\Gamma(\gamma+ \ell)} 
\int_0^\infty K(z) e^{-x/z} d(z^{-1}).  
\]
\end{rem}

\begin{prop}
\label{prop:inverse,fundamental,solution,second,metric} 
Let $\ell$ be a complex number such that 
$\ell \equiv \frac{n-1}{2} + \sigma \mod \zz$. 
Assume that we have either $\ell\notin \zz_{>0}$ or 
$\sigma\notin \frac{n-1}{2} + \zz_{\le 0}$.  
Then:
\begin{enumerate}
\item The truncated Laplace transform 
\begin{align*} 
\chK_\alpha^{(\sigma,\ell)}(\tau_2,x)  
& := \Lap^{(\ell)}\left(K_\alpha^{(\sigma-1)}(\tau_2,\cdot) \right) \\
& 
=\sum_{d\in \eff(X)}N_{\alpha,d}(1) 
\frac{e^{\tau_{2}+\tau_{2}(d)}}
{x^{\rho+\rho(d)-|\alpha|+\frac{n+1}{2}+\sigma}}
\frac{\Gamma(\rho+\rho(d)-|\alpha|+\frac{n+1}{2}+\sigma)}
{\Gamma(\rho+\ell)}. 
\end{align*} 
with $\tau_2 \in H^2(X)$ is well-defined. 
Here $\rho = c_1(X)$, $|\alpha| = \frac{1}{2} \deg T_\alpha$ 
and $N_{\alpha,d}(1)$ is given in \eqref{eq:N_alphad}. 
\item 
Under the convergence assumption for the small quantum 
cohomology of $X$ (see \S \ref{subsec:convergence}), 
$\chK^{(\sigma,\ell)}_\alpha(\tau_2,x)$ 
converges on a region of the form 
$\{(\tau_2,x) : \tau_2 \in U_{\rm sm}, |x|>c\}$ 
where $U_{\rm sm}$ is a region of the form \eqref{eq:Usmall} 
and $c\in \rr_{>0}$. 
\item 
These Laplace transforms define a cohomology-valued solution 
to the second structure connection $\chnabla^{(\sigma)}$, that is, 
the multi-valued bundle map
\begin{equation*}
\chK^{(\sigma,\ell)} \colon
\big(\check{F},\chnabla^{(\sigma)}\big)  
\longrightarrow \left(\check{F},d\right), 
\qquad 
T_{\alpha} \longmapsto \chK_{\alpha}^{(\sigma,\ell)}
\end{equation*}  
defined over $\{(\tau,x) \in U_{\rm sm}\times \cc :  
|x|>c\}$  
intertwines $\chnabla^{(\sigma)}$ 
with the trivial connection $d$. 
\end{enumerate} 
\end{prop}

\begin{proof}
The well-definedness of the truncated Laplace transforms 
$\Lap^{(\ell)}(K_\alpha^{(\sigma-1)})$ 
follows easily from Lemma \ref{lem:inverse,funda,sol} 
by checking the conditions in Definition \ref{defi:Laplace}. 
The coefficients $N_{\alpha,d}(1)$ satisfy the 
following estimate \cite[Lemma 4.1]{iritani_convergence}: 
\begin{equation} 
\label{eq:coeff_estimate} 
|N_{\alpha,d}(1)| \le C_1 C_2^{|d| + \rho(d)} \frac{1}{\rho(d)!} 
\end{equation} 
for some constants $C_1,C_2>0$ independent of $\alpha$ and $d$, 
where $|\cdot|$ is a fixed norm on $H^{\ev}(X)$ and $H_2(X)$. 
The convergence of the series $\chK_\alpha^{(\sigma,\ell)}$ 
follows from this. 

Next we show that $\chK^{(\sigma,\ell)}$ gives a solution 
to the second structure connection. Since $K_\alpha^{(\sigma-1)}$, 
$\alpha=0,\dots,s$ are the columns of an inverse fundamental solution 
for $\nabla^{(\sigma-1)}$, they satisfy the same differential 
relations as $T_\alpha$: 
\begin{align*} 
[z\partial_z K_{0}^{(\sigma-1)},\dots, z\partial_z K_s^{(\sigma-1)}]  
& = [K_0^{(\sigma-1)},\dots, K_s^{(\sigma-1)}]  
\left(-\frac{1}{z} (\frE\bullet_\tau) + 
\mu + \frac{1}{2} - \sigma \right) \\ 
[\partial_\beta K_0^{(\sigma-1)},\dots,\partial_\beta K_s^{(\sigma-1)}] 
& =[K_0^{(\sigma-1)},\dots, K_s^{(\sigma-1)}] 
\frac{1}{z} (T_\beta \bullet_\tau)
\end{align*} 
where we regard $(\frE\bullet_\tau)$, $\mu$, $(T_\beta\bullet_\tau)$ 
as matrices written in the basis $[T_0,T_1,\dots,T_s]$. 
Applying the truncated Laplace transformation $\Lap^{(\ell)}$ 
to the above formulae and using \eqref{eq:FL_rules}, 
we find:
\begin{align*} 
[\partial_x x \chK_{0}^{(\sigma,\ell)},\dots, \partial_x x \chK_s^{(\sigma,\ell)}]  
& = [\partial_x \chK_0^{(\sigma,\ell)},\dots, \partial_x\chK_s^{(\sigma,\ell)}]  
(\frE\bullet_\tau)  + [\chK_0^{(\sigma,\ell)},\dots,\chK_s^{(\sigma,\ell)}] 
\left( \mu + \frac{1}{2} - \sigma \right), \\ 
[\partial_\beta \chK_0^{(\sigma-1)},\dots,\partial_\beta \chK_s^{(\sigma-1)}] 
& =- [\partial_x \chK_0^{(\sigma-1)},\dots, \partial_x \chK_s^{(\sigma-1)}] 
(T_\beta \bullet_\tau).
\end{align*} 
The first equation can be rewritten as: 
\[
[\partial_x \chK_0^{(\sigma,\ell)},\dots,\partial_x \chK_s^{(\sigma,\ell)} ] 
= [\chK_0^{(\sigma,\ell)},\dots,\chK_s^{(\sigma,\ell)}]
\left(\mu - \frac{1}{2} - \sigma\right) (x- \frE\bullet_\tau)^{-1}. 
\]
Together with the second equation, this implies that 
$\chK_\alpha^{(\sigma,\ell)}$, $\alpha=0,\dots,s$ define 
a solution to the second structure connection $\chnabla^{(\sigma)}$. 
\end{proof}

\begin{rem} 
Note that the convergence region $U_{\rm sm}$ in the above proposition 
depends on $c$. The real positive number $c$ can be chosen arbitrarily, 
but $U_{\rm sm}$ becomes smaller if we choose a smaller $c$. 
\end{rem} 

\subsection{Small twisted quantum $D$-modules}
\label{subsec:SQDM}
In this section we study the $\eK$-twisted quantum $D$-module 
$\QDM_{\eK}(X)$ and the quantum $D$-module $\QDM(K_X)$ 
of the total space of $K_X$ over the small quantum cohomology 
locus $H^2(X)$ using quantum Lefschetz theorem 
\cite{Givental-Coates-2007-QRR}. 

Since $c_1(X)$ is assumed to be nef, the anticanonical line bundle 
$K_X^{-1}$ is convex. Therefore, by the results of 
\S\ref{subsec:non-eq,Euler,twist} and 
\S \ref{subsec:non-equiv-limit,inverse,euler},  
the quantum $D$-modules $\QDM_{\eK}(X)$ and $\QDM(K_X)$ 
are well-defined. 
We shall see that, under the convergence assumption 
for the small quantum cohomology in \S \ref{subsec:convergence}, 
the quantum connections for these quantum $D$-modules are 
convergent on a region $U_{\rm sm}\subset H^2(X)$ of the form 
\eqref{eq:Usmall}. 
Therefore we have the following \emph{small quantum $D$-modules}: 
\begin{align}
\label{eq:SQDM}
\begin{split} 
\SQDM_{\eK}(X) & := 
(H^{\ev}(X) \otimes \cO_{U_{\rm sm} \times \cc_z}, \nabla^\eu, S_\eu) \\ 
\SQDM(K_X) & :=  
(H^{\ev}(X) \otimes \cO_{U_{\rm sm}\times \cc_z}, \nabla^{\loc}) 
\end{split} 
\end{align}  
where $\nabla^\eu = \nabla^{\eK}$ is the $\eK$-twisted quantum 
connection, $S_\eu=S_{\eK}$ is the $\eK$-twisted 
pairing $S_\eu(u,v) = \int_{X} u(-z) \cup v(z) \cup \rho$, and 
$\nabla^{\loc} = \nabla^{K_X}$ is the quantum 
connection of $K_X$. 
The superscript `eu' means `Euler' and `loc' means `local'. 
We denote the fundamental solutions (in Proposition 
\ref{prop:prop,QDM,S,flat,...}) for these quantum 
$D$-modules by 
\begin{align*} 
L_\eu(\tau,z) & = L_{\eK}(\tau,z)  
\qquad \text{(see Remark \ref{rem:fullyflat_Euler})} \\ 
L_\loc(\tau,z) & = L^{K_X}(\tau,z)  
\qquad\quad \ \ \text{(see Remark \ref{rem:fullyflat_Evee})}  
\end{align*} 
where the Novikov variable is set to be one. 
For a smooth anticanonical hypersurface $Z\subset X$, we can 
similarly consider the small ambient part quantum $D$-module 
of $Z$ (cf.~Definition \ref{def:ambientQDM}): 
\[
\SQDM_{\amb}(Z) := (H^{\ev}_{\amb}(Z) \otimes 
\cO_{U_{\rm sm} \times \cc_z}, \nabla^Z, S_Z). 
\]

\begin{defi}
\label{defi,I,functions}
For $\alpha\in \{0,\ldots, s\}$, we put:  
\begin{align*}
I_{\alpha}^\eu(\tau_{2},z) &:= 
e^{\tau_2/z} \sum_{d\in \eff(X)} 
N_{\alpha,d}(z)e^{\tau_{2}(d)} \prod_{k=1}^{\rho(d)}(\rho+kz) \\
I_{\alpha}^\loc (\tau_{2},z) &:=
  e^{\tau_2/z} 
\sum_{d\in \eff(X)} N_{\alpha,d}(z)e^{\tau_{2}(d)} 
\prod_{k=0}^{\rho(d)-1}(-\rho-kz)
\end{align*}
where recall that $\rho = c_1(X)$ and 
we set $\prod_{k=1}^{\rho(d)} (\rho+kz) = \prod_{k=0}^{\rho(d)-1} 
(-\rho - kz) = 1$ for $d=0$. 
We call $I_\alpha^\eu$ the \emph{$\eK$-twisted $I$-function} of $X$ and 
and $I_\alpha^\loc$ \emph{the $I$-function} of $K_X$. 
\end{defi}

\begin{rem}\label{rem:I,function}
(1) The \emph{large radius limit} is the limit: 
\[
\Re (\tau_{2}(d)) \to 0 \quad \text{for} \quad 
\forall d \in \eff(X) \setminus \{0\}.
\] 
With our choice of co-ordinates, the large radius limit corresponds to 
$e^{t^i}\to 0$ for $1\le i\le r$. 
The $I$-function satisfies the asymptotics 
\begin{align*}
I_\alpha^\eu(\tau_2,z) \sim_{\lrl} e^{\tau_2/z} T_\alpha 
\quad \text{and} \quad 
I_\alpha^\loc(\tau_2,z) \sim_{\lrl} e^{\tau_2/z} T_\alpha 
\end{align*} 
under the large radius limit 
(the subscript `lrl' stands for the large radius limit).  
Therefore they are linearly independent in a neighbourhood 
of the large radius limit. 

(2) Put $\partial_{\rho}:=\sum_{\beta=0}^{s}\rho_{\beta}\partial_{\beta}$, 
where
$\rho=c_{1}(TX)=\sum_{\beta=0}^{s}\rho_{\beta}T_{\beta}$.
We have
\begin{align}
  \label{eq:I,Istar,relation}
  e^{-\sqrt{-1}\pi\rho/z}z\partial_{\rho} 
I_{\alpha}^\loc(h(\tau_2),z) =\rho I_{\alpha}^\eu(\tau_2,z)
\end{align}
where $h$ is the map in \eqref{eq:h} with $c_1(E) = c_1(K_X^{-1}) = \rho$. 
\end{rem} 

\begin{lem} 
\label{lem:convergence_I}
Suppose that the small quantum product of $X$ is convergent 
as in \S \ref{subsec:convergence}. 
There exists a region $U_{\rm sm} \subset H^2(X)$ 
of the form \eqref{eq:Usmall} such that the $I$-functions 
$I_\alpha^\eu(\tau_2,z)$, $I_\alpha^\loc(\tau_2,z)$ 
are convergent and analytic on $U_{\rm sm}\times \cc_z^\times$. 
\end{lem} 
\begin{proof}
This follows easily from the estimate \eqref{eq:coeff_estimate} 
and $N_{\alpha,d}(z)= z^{-(\rho(d)-|\alpha|)}
z^{-\frac{\deg}{2}}N_{\alpha,d}(1)$. 
\end{proof}

For $\alpha=0$, the $I$-functions $I^\eu_0$ and $I^\loc_0$ 
have the following $z^{-1}$-expansions: 
\begin{align*}
I_0^\eu(\tau_{2},z)
&= F(\tau_{2})\bun + G(\tau_{2})z^{-1}+O(z^{-2}),\\
  I_0^{\loc}(\tau_{2},z)
&= \bun + H(\tau_{2})z^{-1}+O(z^{-2}), 
\end{align*}
where $F(\tau_2)$ is a scalar-valued function and 
$G(\tau_2)$ and $H(\tau_2)$ are $H^2(X)$-valued 
functions. We define the mirror maps by 
\begin{equation}
\label{eq:mirrormaps}
\Mir_\eu(\tau_2) := \frac{G(\tau_2)}{F(\tau_2)},  
\qquad 
\Mir_\loc(\tau_2) := H(\tau_2). 
\end{equation} 
Note that $F(\tau_2)$ is invertible in a neighbourhood of the large 
radius limit point and the mirror maps take values in $H^2(X)$. 
The mirror maps have the asymptotic $\Mir(\tau_2)\sim_{\lrl} \tau_2$ 
and thus induce isomorphisms between neighbourhoods 
of the large radius limit point. 
Quantum Lefschetz Theorem of Coates-Givental 
\cite{Givental-Coates-2007-QRR} gives the following proposition: 

\begin{prop}\label{prop:I,fctn}
For any $\alpha\in\{0, \ldots ,s\}$, there exist 
$v_\alpha(\tau_2,z), w_\alpha(\tau_2,z) 
\in H^{\ev}(X)\otimes \cc[z][\![e^{\tau_2}]\!]$ 
such that 
\begin{align*}
I_{\alpha}^{\eu}(\tau_{2},z) 
& = L_{\eu}(\Mir_{\eu}(\tau_{2}),z)^{-1} 
v_\alpha(\tau_{2},z),  
\\
I_{\alpha}^{\loc}(\tau_{2},z) 
& = 
L_{\loc}(\Mir_{\loc}(\tau_{2}),z)^{-1}
w_\alpha(\tau_2,z). 
\end{align*}
Moreover we have the asymptotics 
$v_\alpha \sim_{\lrl} T_\alpha$, $w_\alpha \sim_{\lrl} T_\alpha$ 
under the large radius limit, and 
$v_\alpha$, $w_\alpha$ are homogeneous of degree 
$2|\alpha| = \deg T_\alpha$ with respect to the usual grading on $H^{\ev}(X)$,  
$\deg z =2$ and $\deg e^{\tau_2} =0$. 
\end{prop}
\begin{proof}
We will just prove the equality for the $\eK$-twisted theory. 
The same argument applies to the other case. 
Coates-Givental \cite[Theorem 2, see also p.27 and p.34]{Givental-Coates-2007-QRR} 
introduced the following ``big'' $I$-function: 
\begin{align*}
 \bI(\tau,z)&:=z \bun +  \tau+ 
\sum_{\beta=0}^s \sum_{(d,\ell)\neq (0,0),(0,1)} 
\frac{Q^d}{\ell!} 
\llangle \frac{T_{\beta}}{z-\psi}, \tau,\dots,\tau\rrangle_{0,\ell+1,d} 
T^\beta 
\prod_{k=1}^{\rho(d)}(\rho + \lambda+ kz)
\end{align*}
and showed that $\bI(\tau,-z)$ lies in the Lagrangian cone 
$\cL_{(\eqeuler,K_X^{-1})}$ 
of the $(\eqeuler,K_X^{-1})$-twisted theory. 
This is related to our $I$-functions as 
\begin{equation} 
\label{eq:Ibig_ourI}
I_{\alpha}^{\eu}(\tau_{2},z) = 
\partial_{\alpha}\bI(\tau,z)\Bigr|_{\tau=\tau_{2}, Q=1, \lambda=0}. 
\end{equation} 
Note that $\partial_\alpha \bI(\tau_2,-z)$ is a 
tangent vector to the cone $\cL_{(\eqeuler,K_X^{-1})}$ 
at $\bI(\tau_2,-z)$. Moreover $\partial_0 \bI(\tau_2,-z)$ 
has the following expansion: 
\[
\partial_0 \bI(\tau_2,- z) = \bF(\tau_2) - z^{-1} \bG(\tau_2) + 
O(z^{-2}) 
\]
with $\bF(\tau_2) \in \cc[\![Q,\tau_2]\!]$ and 
$\bG(\tau_2) \in H^{\ev}(X)\otimes \cc[\lambda][\![Q,\tau_2]\!]$. 
Therefore $\partial_0 \bI(\tau_2,-z)/\bF(\tau_2)$ gives the 
unique intersection point: 
\[
(\bun + \cH_-) \cap T_{\bI(\tau_2,-z)} \cL_{(\eqeuler,K_X^{-1})}. 
\]
Set $\ttau = \bG(\tau_2)/\bF(\tau_2)$. 
The discussion in \S \ref{subsubsec:tgt,space,twisted,cones} 
shows that the tangent space at $\bI(\tau_2,-z)$ is 
generated by $\partial_\alpha J_{(\eqeuler,K_X^{-1})}(\ttau,-z) 
= L_{(\eqeuler,K_X^{-1})}(\ttau,-z)^{-1} T_\alpha$ 
over $\cc[z,\lambda][\![Q,\tau_2]\!]$ (see \eqref{eq:J,L}). 
Therefore there exists $\bv_\alpha(\tau_2,z)\in H^{\ev}(X)
\otimes \cc[z,\lambda][\![Q,\tau_2]\!]$ such that 
\[
\partial_\alpha \bI(\tau_2,z) = 
L_{(\eqeuler,K_X^{-1})}(\ttau,z)^{-1} 
\bv_\alpha(\tau_2,z). 
\]
It is easy to check that $\ttau-\tau_2$ and $\bv_\alpha$ in fact 
belong to $H^{\ev}(X)\otimes\cc[z,\lambda][\![Q e^{\tau_2}]\!]$.  
Under the non-equivariant limit $\lambda \to 0$ and 
the specialization $Q=1$, $\ttau$ becomes $\Mir_{\eu}(\tau_2)$. 
Setting $v_\alpha(\tau_2,z) = \bv_\alpha(\tau_2,z)|_{\lambda=0,Q=1}$ 
and using \eqref{eq:Ibig_ourI}, 
we obtain the formula in the proposition. 
The asymptotics of $v_\alpha(\tau_2,z)$ follows from the asmptotics 
of $I_\alpha^\eu(\tau_2,z)$ in Remark \ref{rem:I,function}. 
The homogeneity of $v_\alpha(\tau_2,z)$ follows from the homogeneity 
of $I_\alpha^\eu(\tau_2,z)$ and $L_\eu(\tau_2,z)$. 
\end{proof}

\begin{lem} 
Suppose that the small quantum cohomology of $X$ is 
convergent as in \S \ref{subsec:convergence}. 
The flat connections $\nabla^\eu$, $\nabla^{\loc}$ 
for the small quantum $D$-modules 
$\SQDM_{\eK}(X)$, $\SQDM(K_X)$ are convergent 
over a region $U_{\rm sm}\subset H^2(X)$ of the form \eqref{eq:Usmall}. 
Also the functions $v_\alpha$, $w_\alpha$ in Proposition 
\ref{prop:I,fctn} are convergent over the same region. 
\end{lem} 
\begin{proof} 
We only discuss the convergence of the $\eK$-twisted theory. 
The other case is similar. 
From Lemma \ref{lem:convergence_I}, it follows that 
the mirror map $\Mir_\eu(\tau_2)$ is convergent on 
a region of the form \eqref{eq:Usmall}. 
Recall that $L_\eu(\tau_2,z)$ is homogeneous of degree zero 
and that $L_\eu(\tau_2,z) = \Id + O(z^{-1})$. 
Recall also that $v_\alpha(\tau_2,z)$ is homogeneous of 
degree $2|\alpha|$ from Proposition \ref{prop:I,fctn}. 
Therefore $L_\eu$ is lower-triangular and the matrix 
$[v_0,\dots,v_s]$ is upper-triangular with respect to 
the grading on $H^{\ev}(X)$. 
Therefore the matrix equation 
\[
\begin{pmatrix}
\vert &  \vert &  & \vert \\ 
I_0^\eu & I_1^\eu & \dots &I_s^{\eu} \\ 
\vert & \vert &  & \vert 
\end{pmatrix} 
= L_\eu(\Mir_\eu(\tau_2),z)^{-1} 
\begin{pmatrix}
\vert &  \vert &  & \vert \\ 
v_0^\eu & v_1^\eu & \dots &v_s^{\eu} \\ 
\vert & \vert &  & \vert 
\end{pmatrix} 
\]
in Proposition \ref{prop:I,fctn} 
can be viewed as the LU decomposition. 
Therefore we can solve for $L_\eu^{-1}$ and $v_\alpha$ 
from $[I^\eu_0,\dots,I^\eu_s]$ by simple linear algebra,  
and Lemma \ref{lem:convergence_I} implies that 
both $L_\eu(\Mir_\eu(\tau_2),z)$ and $v_\alpha$ are 
convergent. The conclusion follows. 
\end{proof} 

The above lemma justifies the definition \eqref{eq:SQDM} 
at the beginning of this section \S \ref{subsec:SQDM}. 

\subsection{Second structure connections are twisted quantum connections} 
We show that the small quantum $D$-modules 
$\SQDM_{\eK}(X)$ and $\SQDM(K_X)$  
correspond to the second structure connections $\chnabla^{(\frac{n+1}{2})}$ 
and $\chnabla^{(-\frac{n+1}{2})}$ respectively. 

By the divisor equation, the quantum connections 
$\nabla^\eu$, $\nabla^\loc$ are invariant 
under the shift $\tau \mapsto \tau+ 2\pi\sqrt{-1} v$ 
with $v\in H^2(X,\zz)$. 
Therefore the small quantum 
$D$-modules descend to the quotient space  
\[
U_{\rm sm} /2\pi\sqrt{-1} H^2(X,\zz) 
\cong \{(q^1,\dots,q^r)\in (\cc^\times)^r: |q^i|<\epsilon\}. 
\]
Since the $I$-functions $I_0^\eu$, $I_0^\loc$ satisfy 
the equation $I(\tau_2 + 2\pi\sqrt{-1} v,z) = 
e^{2\pi\sqrt{-1}v/z} I(\tau_2)$, the mirror maps 
$\Mir = \Mir_\eu$ or $\Mir_\loc$ satisfy 
$\Mir(\tau_2 + 2 \pi\sqrt{-1} v) = \Mir(\tau_2) + 2\pi\sqrt{-1} v$; 
therefore the mirror maps descend to isomorphisms 
\[
\Mir_{\eu/\loc} \colon U'_{\rm sm}/2\pi\sqrt{-1}H^2(X,\zz) 
\xrightarrow{\cong} U''_{\rm sm}/2\pi\sqrt{-1}H^2(X,\zz) 
\]
between neighbourhoods $U'_{\rm sm}$, $U''_{\rm sm}$ of 
the form \eqref{eq:Usmall}. 
Define the maps $\pi_\eu, \pi_\loc$ by 
\begin{align}
\label{eq:map_pi}
\begin{split} 
\pi_\eu(\tau_2,x) &=
\Mir_\eu(\tau_2 - \rho \log x), \\ 
\pi_\loc(\tau_2,x) &= 
\Mir_\loc(\tau_2 -\rho \log x + \pi\sqrt{-1}\rho).
\end{split}  
\end{align}
Choosing smaller $U_{\rm sm}'$ if necessary, 
each of $\pi_\eu$ and $\pi_\loc$ defines a map 
\[
U'_{\rm sm} \times \{x \in \cc^\times: |x|> c\}
\longrightarrow U''_{\rm sm}/2\pi\sqrt{-1}H^2(X,\zz). 
\] 
We need to choose sufficiently small large radius limit 
neighbourhoods of the form \eqref{eq:Usmall} 
which may vary in each case: we denote by $U_{\rm sm}'$, 
$U_{\rm sm}''$ for such neighbourhoods. 

\begin{thm}
\label{thm:QS,with,I,Coates}
Suppose that the small quantum cohomology of $X$ is convergent 
as in \S\ref{subsec:convergence}. 
We have the following isomorphisms of vector bundles with connections:
\begin{align*} 
&\psi_\eu \colon \big(\check{F},\chnabla^{(\frac{n+1}{2})}\big)
\Bigr|_{U_{\rm sm}'\times \{|x|> c\}}  
\cong \pi_\eu^* \SQDM_{\eK}(X)\Bigr|_{z=1} \\
& \psi_\loc \colon \big(\check{F},\chnabla^{(-\frac{n+1}{2})}\big)
\Bigr|_{U_{\rm sm}'\times \{|x|> c\}}
\cong \pi_\loc^*\SQDM(K_X)\Bigr|_{z=1}
\end{align*} 
where, as discussed above, we regard $\SQDM_{\eK}(X)$, $\SQDM(K_X)$ 
as flat connections on the quotient space 
$U_{\rm sm}''/2\pi\sqrt{-1} H^2(X,\zz)$. 
These maps are given by the following formulae: 
\begin{align} 
\label{eq:def_psi_tw} 
\psi_\eu(T_\alpha) & = \left(
\left(-\pi_\eu^* \nabla^\eu\right)_{\partial_x} \right )^{n-|\alpha|}
\left(x^{-1} v_\alpha (\tau_2 - \rho \log x,1)\right) \\ 
\label{eq:def_psi_loc} 
\psi_\loc\left( 
\left(-\chnabla_{\partial_x}^{(-\frac{n+1}{2})} \right)^{|\alpha|} 
T_\alpha \right) 
&= w_\alpha(\tau_2 - \rho \log x + \pi\sqrt{-1} \rho,1)  
\end{align} 
\end{thm}

\begin{proof}
To show that $\psi_{\eu/\loc}$ intertwines the connections, 
we compare the solution $\chK^{(\sigma,\ell)}$ 
to the second structure connection 
from Proposition \ref{prop:inverse,fundamental,solution,second,metric} 
with the inverse fundamental solution $L_{\eu/\loc}^{-1}$ 
of the small quantum $D$-modules. 
More precisely, we check the commutativity of 
a diagram of the form: 
\[
\xymatrix{
\big(\check{F},\chnabla^{(\sigma)}
\big)\Big|_{U_{\rm sm} \times \{|x|>c \}} 
\ar[rr]^{\psi} 
\ar[dr]_{\chK^{(\sigma,\ell)}} 
& & 
\pi^* \SQDM\Big|_{z=1} 
\ar[ld]^{\phantom{XYZ}(\text{constant factor})\cdot 
\pi^*L^{-1}|_{z=1}}\\ 
& (\check{F}, d) & 
}
\]
for suitable $\ell$; we shall take $\ell = 1$ for 
$\sigma = \frac{n+1}{2}$ 
and $\ell= 0$ for $\sigma = -\frac{n+1}{2}$. 

We define the $\cO$-module map $\psi_\eu$  
by the formula \eqref{eq:def_psi_tw} and show that it 
intertwines the connections. 
We have 
\begin{align}
\nonumber 
L_\eu(\pi_{\eu}(\tau_{2},x),1)^{-1} 
& \left(\left(-\pi_{\eu}^{*}\nabla^{\eu}
 \right)_{\partial_{x}}^{n-|\alpha|} 
\left( x^{-1}v_{\alpha} (\tau_{2}-\rho\log x,1) \right) 
\right) \\ 
& =(-\partial_{x})^{n-|\alpha|}
L_\eu(\Mir_\eu (\tau_{2}-\rho \log x),1)^{-1} 
\left(x^{-1}v_ \alpha (\tau_{2}-\rho \log x,1) \right)
\nonumber\\
&=(-\partial_{x})^{n-|\alpha|}x^{-1}I_{\alpha}^\eu
(\tau_{2}-\rho\log x ,1)
\nonumber 
\qquad \qquad \text{(by Proposition \ref{prop:I,fctn})}\\
& =\sum_{d}N_{\alpha,d}(1)
\frac{e^{\tau_{2}+\tau_{2}(d)}}{x^{\rho+\rho(d)+n-|\alpha|+1}}
\prod_{k=1}^{\rho(d)+n-|\alpha|}(\rho+k) 
\label{eq:sol_twSQDM}
\\ 
\label{eq:asymptotics_tw}
& \sim_{\lrl}\frac{T_{\alpha}}{x^{\rho+n-|\alpha|+1}}
(\rho+1)(\rho+2) \cdots (\rho+n-|\alpha|) 
\in H^{\geq 2|\alpha|}(X,\cc)
\end{align}
From Proposition \ref{prop:inverse,fundamental,solution,second,metric}, 
we deduce that the expression in Formula \eqref{eq:sol_twSQDM} is
exactly $\chK_{\alpha}^{\left(\frac{n+1}{2},1\right)}(\tau_{2},x)$. 
This implies that the morphism $\psi_\eu$ is a 
morphism of vector bundles with connection. 
The asymptotics \eqref{eq:asymptotics_tw} at the large radius limit 
shows that it is an isomorphism. 

Since $\{w_\alpha(\tau_2)\}$ form a basis in a neighbourhood 
of the large radius limit, we can define an $\cO$-module 
map $\psi_\loc^{-1}$ such that the formula \eqref{eq:def_psi_loc} holds. 
Note that $\chnabla^{(\sigma)}_{\partial_x}$ has no 
singularities on $U_{\rm sm}' \times \{|x|>c\}$ if we take 
the neighbourhood $U_{\rm sm}'$ sufficiently small. 
We have 
\begin{align}
e^{-\pi\sqrt{-1}\rho}L_\loc(\pi_\loc (\tau_{2},x),1)^{-1}
&  w_{\alpha}(\tau_2 - \rho \log x + \sqrt{-1}\pi \rho,1) 
 \nonumber \\
&= e^{-\pi\sqrt{-1}\rho} 
I_{\alpha}^\loc (\tau_{2}-\rho \log x + \sqrt{-1}\pi \rho,1)
\quad  \text{(by Proposition \ref{prop:I,fctn})}
 \nonumber \\
&=
\sum_{d}N_{\alpha,d}(1)\frac{e^{\tau_{2}+\tau_{2}(d)}}
{x^{\rho+\rho(d)}}
\prod_{k=0}^{\rho(d)-1}(\rho+k) 
\label{eq:sol_locSQDM} \\
&\sim_{\lrl} T_{\alpha}
e^{\tau_{2}} x^{-\rho}. 
\label{eq:asymptotics_loc} 
\end{align}
From Proposition \ref{prop:inverse,fundamental,solution,second,metric}, 
we deduce that the expression in Formula \eqref{eq:sol_locSQDM} 
is exactly 
\begin{displaymath}
(-\partial_{x})^{|\alpha|}
\chK_{\alpha}^{\left(-\frac{n+1}{2},0\right)}(\tau_{2},x). 
\end{displaymath}
This implies that the morphism $\psi_\loc^{-1}$ 
is a morphism of vector bundle with connection. 
The asymptotics \eqref{eq:asymptotics_loc} at the large radius limit
shows that it is an isomorphism. 
\end{proof}

\begin{rem}
\label{rem:I,tilde}
By construction in the proof, we have 
\begin{align*} 
\psi_\eu & = 
L_\eu(\pi_{\eu}(\tau_2,x),1) \circ \chK^{(\frac{n+1}{2},1)} \\ 
\psi_\loc & =  
L_\loc(\pi_{\loc}(\tau_2,x),1) e^{\pi\sqrt{-1}\rho} \circ 
\chK^{(-\frac{n+1}{2},0)}. 
\end{align*} 
with $\chK^{(\sigma,\ell)}$ in 
Proposition \ref{prop:inverse,fundamental,solution,second,metric}. 
In particular we have the following formula for $\psi_\loc(T_\alpha)$: 
\begin{equation}
\label{eq:psiloc_Talpha}
\psi_\loc(T_\alpha)= L_\loc(\pi_\loc(\tau_{2},x),1) 
\left[ 
e^{\sqrt{-1}\pi\rho}\sum_{d}N_{\alpha,d}(1)
\frac{e^{\tau_{2}+\tau_{2}(d)}}{x^{\rho+\rho(d)-|\alpha|}}
\frac{\prod_{k=-\infty}^{\rho(d)-|\alpha|-1}(\rho+k)}
{\prod_{k=-\infty}^{-1}(\rho+k)} 
\right]. 
\end{equation}
\end{rem}

\subsection{Quantum Serre duality in terms of the second structure connections}
\label{sec:qSerre_2ndconn}
In this section we see that the quantum Serre duality between 
$\QDM_{\eK}(X)$ and $\QDM(K_X)$ from Theorem 
\ref{thm:euler,quantum,Serre,QDM} can be rephrased 
in terms of the second structure connections. 
As a corollary, we obtain a description of the quantum 
$D$-module of an anticanonical hypersurface $Z\subset X$ 
in terms of the second structure connection. 
When $X$ is Fano, this gives an entirely algebraic 
description of the quantum connection of $Z$.

We begin with the following lemma:
\begin{lem} 
\label{lem:relation_mirrormaps} 
Let $\overline{f}$ be the map in Lemma 
\ref{lem:equ,limit,f,inverse} and $h$ be the map in \eqref{eq:h} 
in the case where $E = K_X^{-1}$. The map $h\circ \overline{f}$ 
relates the two mirror maps \eqref{eq:mirrormaps} as
\[
(h \circ \overline{f}) (\Mir_{\eu}(\tau_2)) 
= \Mir_{\loc}(h(\tau_2)). 
\]
In particular $h\circ \overline{f}|_{H^2(X)}$ is convergent and gives 
an isomorphism between neighbourhoods of the large radius limit point 
of the form \eqref{eq:Usmall}.  
\end{lem}

We will postpone the proof of the lemma until the end of this 
section. Consider the quantum Serre pairing $S^{\QS}$ 
from Theorem \ref{thm:euler,quantum,Serre,QDM} 
in the case where $E= K_X^{-1}$. 
By the above lemma, $h\circ \overline{f}$ preserves 
$H^2(X)$ and therefore $S^{\QS}$ induces a flat pairing 
\[
S^{\QS} \colon \SQDM_{\eK}(X)\bigr|_{z=-1} 
\times (h\circ \overline{f})^*\SQDM(K_X)\bigr|_{z=1} \to \cO_{U_{\rm sm}}. 
\]
Combined with the $\nabla^{\eu}$-flat shift 
$(-1)^{\frac{\deg}{2}} \colon \SQDM_{\eK}(X)|_{z=1} 
\cong \SQDM_{\eK}(X)|_{z=-1}$, we obtain a flat pairing 
\begin{align}
\label{eq:pairing_P} 
\begin{split} 
&P\colon \SQDM_{\eK}(X)\bigr|_{z=1}\times
(h\circ\overline{f})^{*}\SQDM(K_{X})\bigr|_{z=1} 
\to \cO_{U_{\rm sm}} \\
&\qquad \text{defined by} \qquad 
P(u,v) = \int_X \left((-1)^{\frac{\deg}{2}} u\right) \cup v.  
\end{split}
\end{align} 
The second structure connections satisfy a certain ``difference 
equation'' with respect to the parameter $\sigma$. 
It is easy to check that we have the following 
morphism of meromorphic flat connections \cite[Theorem 9.4.b]{Hfm}: 
\begin{align*}
\Delta_{\sigma} \colon (\check{F},\chnabla^{(\sigma+1)}) 
& \longrightarrow (\check{F},\chnabla^{(\sigma)}) \\
T_\alpha &\longmapsto \chnabla_{\partial_x}^{(\sigma)} 
T_\alpha = - \left(\mu -\frac{1}{2}-\sigma\right)
(\frE\circ_{\tau}-x)^{-1} T_{\alpha}. 
\end{align*}
This is an isomorphism over $(U_{\rm sm}\times \cc_x) \setminus \Sigma$ 
if $\mu-\frac{1}{2}-\sigma$ is invertible, i.e.~if 
$\sigma \notin \{-\tfrac{n+1}{2},-\tfrac{n-1}{2},\dots,\tfrac{n-1}{2}\}$. 

Lemma \ref{lem:relation_mirrormaps} shows that 
$(h\circ \overline{f} \circ \pi_{\eu})(\tau_2,x) 
= \pi_\loc(\tau_2,x)$. Thus the pairing 
\eqref{eq:pairing_P} induces the flat pairing 
\begin{equation}
\label{eq:pairing_P_pistar}
P \colon \pi_\eu^* \SQDM_{\eK}(X)\bigr|_{z=1} 
\times \pi_\loc^* \SQDM(K_X)\bigr|_{z=1} 
\longrightarrow \cO_{U_{\rm sm}}.  
\end{equation}
We also have the morphism of flat connections 
\begin{equation}
\label{eq:morphism_rho} 
\rho \colon \pi_\eu^* \SQDM_{\eK}(X)\bigr|_{z=1} 
\longrightarrow 
\pi_\loc^* \SQDM(K_X)\bigr|_{z=1}
\end{equation} 
induced from the 
morphism in Theorem \ref{thm:euler,quantum,Serre,QDM} (2).  

\begin{thm}
\label{thm:pairing,vs,quantum,serre} 
Via the isomorphisms $\psi_\eu$, $\psi_\loc$ in 
Theorem \ref{thm:QS,with,I,Coates},
the pairing  $P$ \eqref{eq:pairing_P_pistar}
coincides with $(-1)^{n+1}\check{g}$, 
i.e.~$P(\psi_{\eu} T_\alpha, \psi_\loc T_\beta) = (-1)^{n+1}
\check{g}(T_\alpha,T_\beta)$ 
and the morphism $\rho$ \eqref{eq:morphism_rho} coincides  
with the composition: 
\[
\Delta := (-1)^{n+1} \Delta_{-\frac{n+1}{2}} \circ \Delta_{-\frac{n-1}{2}} \circ 
\cdots \circ \Delta_{\frac{n-1}{2}} \colon 
\big(\check{F},\chnabla^{(\frac{n+1}{2})} \big) 
\longrightarrow 
\big(\check{F},\chnabla^{(-\frac{n+1}{2})} \big). 
\]
Moreover we have: 
\begin{align*}
\check{g}(\gamma_1,\gamma_2) & = (-1)^{n+1} 
\int_X \left( (-1)^{\frac{\deg}{2}}\chK^{(\frac{n+1}{2},1)} \gamma_1
\right) \cup
\chK^{(-\frac{n+1}{2},0)} \gamma_2, 
\\ 
\chK^{(-\frac{n+1}{2},0)} \circ \Delta &= \rho \circ \chK^{(\frac{n+1}{2},1)} 
\end{align*} 
for $\gamma_1, \gamma_2 \in H^{\ev}(X)$. 
\end{thm}

\begin{proof} 
First we prove that $P$ corresponds to $(-1)^{n+1} \check{g}$. 
Since both pairings are flat, 
it is enough to compare the asymptotics of 
$P(\psi_\eu T_\alpha, \psi_\loc T_\beta)$ 
and $\check{g}(T_\alpha,T_\beta)$ 
at the large radius limit. 
Since the quantum product equals the cup product at the large 
radius limit, we have 
\begin{align}
\label{eq:gcheck_asymptotics}
\begin{split} 
\check{g}(T_\alpha,T_\beta)\Bigr|_{x=1} 
& \sim_{\lrl} -\int_{X}(1-\rho)^{-1}\cup T_\alpha \cup T_\beta \\
& =
\begin{cases} 
-\int_{X}\rho^{n-|\alpha|-|\beta|}\cup
T_{\alpha}\cup T_{\beta} & \text{if $|\alpha|+|\beta|\le n$;}\\ 
0 & \text{otherwise.}
\end{cases} 
\end{split} 
\end{align}
We then compute the asymptotics of 
$P(\psi_\eu(T_\alpha), \psi_\loc(T_\beta))$. 
By Remark \ref{rem:I,tilde}, after some computation, we find:  
\begin{align}
\label{eq:psi_Talpha}
(-1)^{\frac{\deg}{2}}\psi_{\eu}(T_{\alpha}) \Bigr|_{x=1}
=L_{\eu}(\Mir_{\eu}(\tau_{2}),-1) 
\left[ \sum_{d}
\left((-1)^{\frac{\deg}{2}}N_{\alpha,{d}}(1)\right) 
e^{-\tau_{2}+\tau_{2}(d)}
\prod_{k=1}^{\rho(d)+n-|\alpha|}(-\rho+k) 
\right].  
\end{align} 
We have already found a similar formula \eqref{eq:psiloc_Talpha} 
for $\psi_\loc(T_\alpha)$. 
In view of Lemma \ref{lem:relation_mirrormaps}, 
Theorem \ref{thm:euler,quantum,Serre,QDM} (3) 
gives the identity: 
\[
\big(L_\eu (\Mir_\eu( \tau_2 ),-z)\gamma_{1}, 
L_{\loc}(\Mir_\loc (h(\tau_2)),z)\gamma_{2} \big)
=\left(\gamma_{1},e^{-\sqrt{-1}\pi\rho/z}\gamma_{2} \right)
\]
for $\gamma_1,\gamma_2 \in H^{\ev}(X)$, 
where $(u,v) = \int_X u\cup v$ is the Poincar\'{e} pairing. 
Therefore equations \eqref{eq:psi_Talpha} and \eqref{eq:psiloc_Talpha} 
give: 
\[
  P(\psi_{\eu}(T_{\alpha}),\psi_{\loc}(T_{\beta})) \Bigr|_{x=1} 
\sim_{\lrl} (-1)^n 
\int_{X}T_{\alpha}\cup T_{\beta} \cup 
\frac{\prod_{k=1}^{n-|\alpha|}(\rho-k)}{\prod_{k=1}^{|\beta|}(\rho-k)}. 
\] 
To have non-zero asymptotics, 
we must have $|\alpha|+|\beta|\leq n$; in this case, the right-hand side is: 
\begin{displaymath}
(-1)^n \int_{X}T_{\alpha}\cup T_{\beta}
\cup \prod_{k=|\beta|+1}^{n-|\alpha|}(\rho-k)  
= (-1)^n \int_{X}T_{\alpha}\cup
T_{\beta} \cup \rho^{n-|\alpha|-|\beta|}.
\end{displaymath}
Comparing this with \eqref{eq:gcheck_asymptotics}, 
we deduce that $P(\psi_\eu T_\alpha,\psi_\loc T_\beta) 
= (-1)^{n+1} \check{g}(T_\alpha,T_\beta)$. 
Notice that the above computation shows: 
\[
P(\psi_\eu T_\alpha, \psi_\loc T_\beta) = 
\int_X \left( (-1)^{\frac{\deg}{2}}\chK^{(\frac{n+1}{2},1)}_\alpha
\right) \cup
\chK^{(-\frac{n+1}{2},0)}_\beta. 
\]
We deduce the equality of the pairings.

Next we prove that the morphism $\rho$ corresponds to 
$\Delta$, i.e.~$\rho \circ \psi_\eu = \psi_\loc \circ \Delta$. 
Recall from Remark \ref{rem:I,tilde} 
that $\psi_\eu,\psi_\loc$ are given by 
\begin{align*} 
\psi_\eu & = L_\eu(\pi_\eu(\tau_2,x),1) \circ \chK^{(\frac{n+1}{2},1)}, \\ 
\psi_\loc & = L_\loc(\pi_\loc(\tau_2,x),1) e^{\pi\sqrt{-1}\rho} \circ 
\chK^{(-\frac{n+1}{2},0)}. 
\end{align*} 
Therefore it suffices to prove the following formulae: 
\begin{align*} 
\rho\circ L_\eu(\pi_\eu(\tau_2,x),1) & = L_\loc(\pi_\loc(\tau_2,x),1) 
e^{\pi\sqrt{-1}\rho} \circ \rho, \\ 
\rho \circ \chK^{(\frac{n+1}{2},1)} 
& = \chK^{(-\frac{n+1}{2},0)} \circ \Delta.  
\end{align*} 
The first equation follows from Theorem \ref{thm:euler,quantum,Serre,QDM} (3) 
in view of Lemma \ref{lem:relation_mirrormaps}. 
To see the second equation, it suffices to prove: 
\[
\rho \circ \chK^{(\frac{n+1}{2},1)} = \chK^{(\frac{n+1}{2},0)} 
\quad \text{and} \quad \chK^{(\sigma+1,0)} 
= - \chK^{(\sigma,0)} \circ \Delta_\sigma. 
\]
The first formula is immediate from the definition. 
To see the second, we calculate: 
\[
\chK^{(\sigma,0)}(\Delta_\sigma T_\alpha) = 
\chK^{(\sigma,0)} (\chnabla^{(\sigma)}_{\partial_x} T_\alpha) 
= \partial_x (\chK^{(\sigma,0)} T_\alpha) = 
\partial_x \chK^{(\sigma,0)}_\alpha = - \chK^{(\sigma+1,0)}_\alpha. 
\]
The conclusion follows. 
\end{proof}

Combined with Corollary \ref{cor:Z}, the above theorem implies: 
\begin{cor} 
\label{cor:anticanonical_hypersurface}
Let $Z$ be a smooth anticanonical hypersurface of $X$ 
which satisfies one of the conditions in Lemma \ref{lem:cond_Z}.  
Then the small quantum $D$-module 
$(\iota^* \circ \pi_\eu)^*\SQDM_{\amb}(Z)|_{z=1}$ 
of $Z$ 
is isomorphic to the image $\im \Delta$ of the morphism: 
\[
\Delta \colon \big(\check{F},\chnabla^{(\frac{n+1}{2})} \big) 
\Bigr|_{U_{\rm sm}' \times \{|x|>c\}}
\to \big(\check{F},\chnabla^{(-\frac{n+1}{2})}\big) 
\Bigr|_{U_{\rm sm}' \times \{|x|>c\}}. 
\]
where $\iota^* \circ \pi_{\eu}$ is regarded as a map 
\[
\iota^* \circ \pi_{\eu} \colon 
U_{\rm sm}' \times \{|x|>c\} \to 
\iota^*(U_{\rm sm}'')/2\pi\sqrt{-1}H^2(Z,\zz). 
\] 
The isomorphism sends $\Delta(T_\alpha)\in \im(\Delta)$ 
to $\iota^*\psi_\eu(T_\alpha) \in (\iota^*\circ\pi_\eu)^* 
\SQDM_{\amb}(Z)|_{z=1}$.  
\end{cor} 
\begin{rem}
Recall that the conditions in Lemma \ref{lem:cond_Z} are satisfied 
for an anticanonical hypersurface if $X$ is Fano. 
\end{rem} 

\begin{proof}[Proof of Lemma \ref{lem:relation_mirrormaps}] 
We consider the following equivariant $I$-functions
(cf.~Definition \ref{defi,I,functions}): 
\begin{align*} 
I^{\eu,\lambda}_\alpha(\tau_2,z) & = e^{\tau_2/z} 
\sum_{d\in \eff(X)} N_{\alpha,d}(z) e^{\tau_2(d)} 
\prod_{k=1}^{\rho(d)}(\rho+\lambda + kz) \\ 
I^{\loc,\lambda}_\alpha(\tau_2,z) & 
=  e^{\tau_2/z} 
\sum_{d\in \eff(X)} N_{\alpha,d}(z) e^{\tau_2(d)} 
\prod_{k=0}^{\rho(d)-1}(-\rho-\lambda - kz) 
\end{align*} 
and define the equivariant mirror maps $\Mir^\lambda_\eu$, 
$\Mir^\lambda_\loc$ as in \eqref{eq:mirrormaps}. 
By exactly the same argument as Proposition \ref{prop:I,fctn}, 
we have that 
\begin{align} 
\label{eq:equivI_L}
\begin{split}
I^{\eu,\lambda}_\alpha(\tau_2,z) & = 
L_{\eu,\lambda}\left(\Mir^\lambda_\eu(\tau_2),z\right)^{-1}
v_\alpha^{\lambda}(\tau_2,z), \\ 
I^{\loc,\lambda}_\alpha(\tau_2,z) & = 
L_{\loc,\lambda}
\left( \Mir^\lambda_\loc(\tau_2),z \right)^{-1} 
w_\alpha^{\lambda}(\tau_2,z)  
\end{split} 
\end{align}
for some $v_\alpha^{\lambda}(\tau_2,z)$,  
$w_\alpha^{\lambda}(\tau_2,z) \in H^{\ev}(X) \otimes 
\cc[\lambda,z][\![e^{\tau_2}]\!]$. 
Here we set $L_{\eu,\lambda} = L_{(\eqeuler,K_X^{-1})}$ 
and $L_{\loc,\lambda} = L_{(\ctop_{-\lambda}^{-1},K_X)}$. 
Similarly to \eqref{eq:I,Istar,relation}, we have the following 
relationship:  
\begin{equation} 
\label{eq:relation_twoIs} 
e^{-\pi\sqrt{-1} \rho/z} 
(z\partial_\rho + \lambda) I^{\loc,\lambda}_\alpha(h(\tau_2),z) = 
(\rho+\lambda) I^{\eu,\lambda}_\alpha(\tau_2,z). 
\end{equation} 
We compute both sides of this equation. 
By \eqref{eq:equivI_L}, the left-hand side equals 
\begin{align}
\label{eq:lhs_relation_twoIs} 
\begin{split}
e^{-\pi\sqrt{-1} \rho/z} 
(z \partial_\rho + \lambda) 
L_{\loc,\lambda}
&\left(\Mir^\lambda_\loc(h(\tau_2)), z \right)^{-1} 
w_\alpha^{\lambda}(h(\tau_2),z)  \\ 
& = e^{-\pi\sqrt{-1} \rho/z} 
L_{\loc,\lambda}
\left(\Mir^\lambda_\loc(h(\tau_2)),z \right)^{-1} 
\tilde{w}_{\alpha}^\lambda(\tau_2,z) 
\end{split} 
\end{align}
where $\tilde{w}_\alpha^{\lambda}(\tau_2,z)$ is obtained 
from $w_\alpha(h(\tau_2),z)$ by applying 
$z \big((\Mir^\lambda_\loc\circ h)^*
 \nabla^{(\ctop_{-\lambda}^{-1},K_X)} \big)_{\rho} 
+ \lambda$ and 
is an element of $H^{\ev}(X)\otimes \cc[\lambda,z][\![e^{\tau_2}]\!]$. 
On the other hand, by \eqref{eq:equivI_L} again, 
the right-hand side of \eqref{eq:relation_twoIs} is: 
\begin{align} 
\label{eq:rhs_relation_twoIs}
\begin{split} 
(\rho+\lambda) &  L_{\eu,\lambda}
\left(\Mir_\eu^\lambda(\tau_2),z 
\right)^{-1} v_\alpha^\lambda(\tau_2,z) \\ 
& = L_{(\eqeulerstar,K_X)}\left(f(\Mir_\eu^{\lambda}(\tau_2)),z \right)^{-1} 
(\rho+\lambda) v_\alpha^{\lambda}(\tau_2,z)  
\qquad \qquad  \quad
\text{(by Theorem \ref{thm:quantum,Serre,QDM} (3))} \\ 
& = 
e^{-\pi\sqrt{-1} \rho/z} L_{\loc,\lambda}
\left(h(f(\Mir_\eu^{\lambda}(\tau_2))),z \right)^{-1} 
(\rho + \lambda) v_\alpha^{\lambda}(\tau_2,z) 
\quad \text{(by Remark \ref{rem:fullyflat_Evee}).}
\end{split} 
\end{align} 
Comparing \eqref{eq:lhs_relation_twoIs} and \eqref{eq:rhs_relation_twoIs},  
we obtain 
\[
L_{\loc,\lambda}
(\tau_2',z )^{-1} 
\tilde{w}_{\alpha}^\lambda(\tau_2,z) = 
L_{\loc,\lambda}
(\tau_2'',z )^{-1} 
(\rho + \lambda) v_\alpha^{\lambda}(\tau_2,z) 
\]
with $\tau_2' = \Mir^\lambda_\loc(h(\tau_2))$ and 
$\tau_2'' = h(f(\Mir_\eu^{\lambda}(\tau_2)))$. 
Since $(\rho+\lambda) v_\alpha^\lambda(\tau_2,z)$, $\alpha=0,\dots,s$  
form a basis of $H^{\ev}(X)$ 
and $\tilde{w}_\alpha^{\lambda}(\tau_2,z)$ does not 
contain negative powers of $z$, 
we find that 
\[
L_{\loc,\lambda}(\tau_2',z) 
L_ {\loc,\lambda}(\tau_2'',z )^{-1} 
\]
does not contain negative powers in $z$. 
By the asymptotics $L_{\loc,\lambda}(\tau,z)= 
\Id + O(z^{-1})$, we must have 
$L_{\loc,\lambda}(\tau_2',z)^{-1}=L_{\loc,\lambda}(\tau_2'',z)^{-1}$. 
The asymptotics $L_{\loc,\lambda}(\tau,z)^{-1} \bun = \bun 
+ \tau/z + O(z^{-2})$ shows that $\tau_2' = \tau_2''$. 
The conclusion follows by taking the non-equivariant limit. 
\end{proof}

\begin{rem} 
Consider the family of connection 
$\chnabla^{\left(\frac{n+1}{2}+k\right)}$ for $k \in \mathbb{Z}$. 
Via the morphisms $\Delta_\sigma$, we have: 
\begin{itemize}
\item for $k \in \zz_{\ge 0}$, $\chnabla^{\left(\frac{n+1}{2}+k\right)}$ 
is isomorphic to $\chnabla^{\left(\frac{n+1}{2}\right)}$ 
as meromorphic connections; 
\item for $k \in \zz_{\ge 0}$, $\chnabla^{\left(-\frac{n+1}{2}-k\right)}$ 
is isomorphic to $\chnabla^{\left(-\frac{n+1}{2}\right)}$ 
as meromorphic connections. 
\end{itemize}
Theorem \ref{thm:QS,with,I,Coates} above 
gives a geometric interpretation of these two connections. 
It would be interesting to understand the intermediate connections 
$\chnabla^{(k)}$ for $k \in\{-\frac{n-1}{2},\ldots ,\frac{n-1}{2}\}$.
\end{rem} 

\subsection{Hodge filtration for the second structure connection} 
\label{subsec:Hodge_filt}
The small quantum $D$-modules $\SQDM_{\eK}(X)$ 
and $\SQDM(K_X)$ restricted to $z=1$ have a natural 
filtration, called the \emph{A-model Hodge filtration} 
\cite[Lecture 7]{Morrison-mathematical-aspects},  
\cite[\S 8.5.4]{Cox-Katz-Mirror-Symmetry}, 
and these small quantum $D$-modules are variations 
of Hodge structure. 
In this section, we identify the corresponding filtration 
on the second structure connection. 
See \cite{Stienstra-resonant,Konishi-Minabe-localB,
Konishi-Minabe-mixed,Konishi-Minabe-local-quantum} 
for related studies on the Hodge structure for  
local quantum cohomology. 

We follow the notation in Theorem \ref{thm:QS,with,I,Coates} 
and write $U_{\rm sm}'$ and $U_{\rm sm}''$ for large 
radius limit neighbourhoods in $H^2(X)$ on which (respectively) 
the second structure connection 
$(\check{F},\chnabla^{(\sigma)})$ 
and our small quantum $D$-modules are convergent. 

\begin{defi}
We define the subbundle $F^p$ of the trivial 
bundle 
$H^{\ev}(X)\times U_{\rm sm}'' \to U_{\rm sm}''$ by 
\[
F^p := H^{\le 2n-2p}(X) \times U_{\rm sm}''
\]
and call it the \emph{A-model Hodge filtration}. 
Because the small quantum product preserves the degree:  
\begin{align*}
\deg(T_{\alpha}\bullet_{\tau_{2}}^{\eK}T_{\beta}) 
&=\deg(T_{\alpha})+\deg(T_{\beta}) \\ 
\deg(T_{\alpha}\bullet_{\tau_{2}}^{K_{X}}T_{\beta}) 
&=\deg(T_{\alpha})+\deg(T_{\alpha})
 \end{align*}
the filtration satisfies \emph{Griffiths transversality} 
with respect to the small quantum connections: 
\begin{displaymath}
\nabla^{\eu}_{\alpha}(F^{p})\subset F^{p-1} 
\quad \text{and} \quad   
\nabla^{\loc}_{\alpha}(F^{p})\subset F^{p-1}
\end{displaymath}
for $\alpha$ with $|\alpha|=1$. The Hodge filtration also satisfies 
the following orthogonality: 
\begin{displaymath}
    P(F^{p},F^{n-p+1})=0
\end{displaymath}
with respect to the pairing $P$ in \eqref{eq:pairing_P}; in other words,  
the A-model Hodge filtrations on $\SQDM(K_X)|_{z=1}$ and 
$\SQDM_{\eK}(X)|_{z=1}$ are annihilators of each other. 
\end{defi}

Next we introduce a filtration on the second structure 
connection. 
\begin{defi} 
Consider the second structure connection 
$(\check{F},\chnabla^{(-\frac{n+1}{2})})$ 
restricted to the small parameter space 
$U_{\rm sm}' \times \cc_x$. 
Define $\check{F}^p_\loc$ to be the 
$\cO_{U_{\rm sm}'\times \cc_x}(*\Sigma)$-submodule 
of $\cO(\check{F})(*\Sigma)$ generated by 
\[
\left\{ 
\left(\chnabla^{(-\frac{n+1}{2})}_{\partial_{x}}\right)^{k} 
T_{\alpha} : \ |\alpha| \leq k \leq n-p \right\}. 
\] 
Define $\check{F}^p_\eu$ to be the $\check{g}$-orthogonal 
of $\check{F}^{n-p+1}_\loc$, i.e.
\[
\check{F}^p_\eu := \left\{ s\in \cO(\check{F})(*\Sigma) : 
\check{g}(s,\gamma) = 0, \ 
\forall \gamma \in \check{F}^{n-p+1}_\loc \right\}.  
\]
These are decreasing filtrations. 
\end{defi} 

\begin{lem}
The filtrations $\check{F}^p_\loc, \check{F}^p_\eu$ 
satisfy the Griffiths transversality: 
$\chnabla^{(-\frac{n+1}{2})} 
\check{F}^{p}_\loc\subset 
\Omega^1_{U_{\rm sm}'\times \cc_x} \otimes 
\check{F}^{p-1}_\loc$ and 
$\chnabla^{(\frac{n+1}{2})} 
\check{F}^p_\eu \subset 
\Omega^1_{U_{\rm sm}'\times \cc_x} \otimes 
\check{F}^{p-1}_\eu$. 
\end{lem}
\begin{proof} 
It suffices to prove the Griffiths transversality for $\check{F}^p_\loc$. 
We write $\chnabla$ for $\chnabla^{(-\frac{n+1}{2})}$ 
to save notation. 
The inclusion $\chnabla_{\partial_x} 
\check{F}^p_\loc \subset \check{F}^{p-1}_\loc$ is obvious. 
We prove $\chnabla_{\beta} \check{F}^p_\loc 
\subset \check{F}^{p-1}_\loc$ for $\beta$ with $|\beta|=1$. 
Take $\alpha$ and $k\in \zz_{\geq 0}$ satisfying 
$|\alpha| \le k\le n-p$. 
We have 
\begin{align}
\label{eq:Griffiths}  
\begin{split} 
\chnabla_\beta \left(
\chnabla_{\partial_x}\right)^k T_\alpha 
& = \left(
\chnabla_{\partial_x}\right)^k 
\chnabla_\beta T_\alpha 
= \left(
\chnabla_{\partial_x}\right)^k 
\left(\mu+\frac{n}{2} \right) 
\left( (\frE \bullet_\tau) - x\right)^{-1} T_\beta \bullet_{\tau} T_\alpha \\ 
& = - \left(
\chnabla_{\partial_x}\right)^{k+1} 
T_\beta \bullet_\tau T_\alpha. 
\end{split} 
\end{align} 
Since $\rho = c_1(X)$ is nef, the small quantum product
$T_\beta \bullet_\tau T_\alpha$ is a linear combination of 
classes of degree less than or equal to $2|\alpha|+2$. 
Therefore the expression \eqref{eq:Griffiths} lies in $\check{F}^{p-1}$. 
\end{proof} 


\begin{thm} 
There exists a small neighbourhood $U_{\rm sm}'$ of the 
form \eqref{eq:Usmall} such that we have 
\[
\psi_\loc(\check{F}^p_\loc )  = F^p,  \qquad \qquad 
\psi_\eu(\check{F}^p_\eu) = F^p 
\]
over $U_{\rm sm}' \times \{|x|>c\}$, where 
$\psi_\loc, \psi_\eu$ are the isomorphisms in Theorem 
\ref{thm:QS,with,I,Coates} and $F^p$ is the 
A-model Hodge filtration of $\SQDM_{\eK}(X)|_{z=1}$ 
or of $\SQDM(K_X)|_{z=1}$. 
\end{thm} 
\begin{proof} 
Since the A-model Hodge filtration satisfies the orthogonality, 
it suffices to show that $\psi_\loc(\check{F}^p_\loc) = F^p$. 
When $|\alpha|\le k\le n-p$, we have 
\[
\psi_\loc((\chnabla_{\partial_x})^k T_\alpha) 
= (-1)^{|\alpha|} \left( (\pi_\loc^*\nabla)_{\partial_x}
\right)^{k-|\alpha|} 
w_\alpha(\tau_2 - \rho \log x + \pi\sqrt{-1} \rho, 1) 
\]
where $\chnabla = \chnabla^{(-\frac{n+1}{2})}$. 
This belongs to $F^p$ by the Griffiths transversality for the 
A-model Hodge filtration. Considering the case $k=|\alpha|$, 
we can see that these sections span $F^p$. 
\end{proof} 

\begin{rem} 
It follows from the above theorem that the filtrations 
$\check{F}^p_\loc$, $\check{F}^p_\eu$ are subbundles 
over $U_{\rm sm}' \times \{|x|>c\}$ with $U_{\rm sm}'$ 
sufficiently small. It would be interesting to study where 
they are not subbundles, and how we can extend them 
along the singularity $\Sigma$. 
\end{rem} 

Let $Z\subset X$ be a smooth anticanonical hypersurface. 
The small ambient part quantum $D$-module $\SQDM_{\amb}(Z)$ 
also admits the A-model Hodge filtration 
\[
F^p = H^{\leq 2(n-1)-2p}_{\amb}(Z) \times U_{\rm sm}''.  
\]
Combined with Corollaries \ref{cor:Z} and 
\ref{cor:anticanonical_hypersurface}, 
we obtain the following corollary: 
\begin{cor} 
Suppose that an anticanonical hypersurface $Z$ of $X$ satisfies 
one of the conditions in Lemma \ref{lem:cond_Z}. 
Under the isomorphism 
\[
(\iota^* \circ \pi_\eu)^*\SQDM_{\amb}(Z) \cong 
\im\left(\Delta \colon (\check{F}, \chnabla^{(\frac{n+1}{2})}) 
\to (\check{F},\chnabla^{(-\frac{n+1}{2})}) \right) 
\]
in Corollary \ref{cor:anticanonical_hypersurface}, 
the A-model Hodge filtration $F^p$ on $\SQDM_{\amb}(Z)$ 
corresponds to $\Delta(\check{F}^{p+1}_\eu)$,
which is contained in 
$\check{F}^{p}_\loc$. 
\end{cor}

\section{Quintic in $\pp^{4}$}
\label{sec:quintic-pp4}
In this section, we make our result explicit in the case of 
$X=\pp^{4}$ and $E=\cO(5)$. This example was 
also studied by Dubrovin \cite[\S 5.4]{Dubrovin-Almost-Frob-2004}. 
Let $H=c_1(\cO(1)) \in H^2(\pp^4)$ be the hyperplane 
class and let $t$ denote the co-ordinate on $H^2(\pp^4)$ dual to $H$. 
We use the basis 
\[
\{T_0,T_1,T_2,T_3,T_4\}=\{1, H, H^2, H^3,H^4\}
\]
of $H^{\ev}(\pp^4)$. 
The small quantum connection of $\pp^4$ is given by: 
\begin{align*}
  \nabla_{\partial_{t}}^{(\sigma-1)}
=\partial_{t}+z^{-1} (H\bullet_{t}), \qquad 
\nabla_{z\partial_{z}}^{(\sigma-1)} 
=z\partial_{z}-z^{-1} 
5(H \bullet_{t})+ \left(\mu + \frac{1}{2} - \sigma\right) 
\end{align*}
where 
\[
H \bullet_t = 
\begin{pmatrix} 
0 & 0 & 0 & 0 & e^t \\ 
1 & 0 & 0 & 0 & 0 \\ 
0 & 1 & 0 & 0 & 0 \\ 
0 & 0 & 1 & 0 & 0 \\
0 & 0 & 0 & 1 & 0 
\end{pmatrix}
\qquad 
\mu = \begin{pmatrix} 
-2 & 0 & 0 & 0 & 0  \\
0 & -1 & 0 & 0 & 0 \\ 
0 & 0 & 0 & 0 & 0 \\ 
0 & 0 & 0 & 1 & 0 \\ 
0 & 0 & 0 & 0 & 2 
\end{pmatrix} 
\]
\subsection{Fourier-Laplace transformation} 
We illustrate the Fourier-Laplace transformation 
in \S \ref{subsec:second_str_conn} for the small quantum 
connection of $\pp^4$. 
We write $\partial_t$, $z\partial_z$ for the action of the 
small quantum connection $\nabla_{\partial_t}^{(\sigma-1)}$, 
$\nabla_{z\partial_z}^{(\sigma-1)}$ respectively. 
We have 
\begin{align*}
  \partial_{t} T_0 &=z^{-1} T_{1}
&    {z\partial_{z}}T_{0}&=-5z^{-1}T_{1}-(\sigma+\tfrac{3}{2})T_{0} \\
  {\partial_{t}} T_{1}&=z^{-1}T_{2} 
&    {z\partial_{z}} T_{1}&=-5z^{-1}T_{2}-(\sigma+\tfrac{1}{2})T_{1} \\
  {\partial_{t}} T_{2}&=z^{-1}T_{3} 
&    {z\partial_{z}} T_{2}&=-5z^{-1}T_{3}-(\sigma-\tfrac{1}{2})T_{2} \\
  {\partial_{t}} T_{3}&=z^{-1}T_{4} 
&    {z\partial_{z}} T_{3}&=-5z^{-1}T_{4}-(\sigma- \tfrac{3}{2})T_{3} \\
  {\partial_{t}} T_{4}&=e^{t}z^{-1}T_{0}
&    {z\partial_{z}} T_{4}&=-5e^{t}z^{-1}T_{0}-(\sigma-\tfrac{5}{2})T_{4} 
\end{align*}
Under the Fourier-Laplace transformation 
$z\partial_{z}= x\partial_{x}+1$ and 
$z^{-1}= -\partial_{x}$, we have: 
\begin{align} 
\nonumber 
  \partial_{t}T_{0}&=(-\partial_{x})T_{1} 
&    {x\partial_{x}}T_{0}&=5\partial_{x}T_{1}-(\sigma+\tfrac{5}{2})T_{0} \\
\nonumber 
  {\partial_{t}} T_{1}&=(-\partial_{x})T_{2}
&    {x\partial_{x}} T_{1}&=5\partial_{x}T_{2}-(\sigma+\tfrac{3}{2})T_{1} \\
\label{eq:afterFL_P4}
  {\partial_{t}} T_{2}&=(-\partial_{x})T_{3}
   &    {x\partial_{x}} T_{2}&=5\partial_{x}T_{3}-(\sigma+\tfrac{1}{2})T_{2} \\
\nonumber 
  {\partial_{t}} T_{3}&=(-\partial_{x})T_{4}
&    {x\partial_{x}} T_{3}&=5\partial_{x}T_{4}-(\sigma-\tfrac{1}{2})T_{3} \\
\nonumber 
  {\partial_{t}} T_{4}&=e^{t}(-\partial_{x})T_{0}
&    {x\partial_{x}} T_{4}&=5e^{t}\partial_{x}T_{0}-(\sigma-\tfrac{3}{2})T_{4} 
\end{align} 
These formulas define the second structure connection 
$\chnabla^{(\sigma)}$: 
\begin{align*}
\chnabla^{(\sigma)}_{\partial_t} & = 
\partial_t + \frac{1}{5^5 e^t - x^5} 
\left(\mu - \frac{1}{2}  - \sigma \right) 
\begin{pmatrix} 
 5^4 e^t & 5^3 e^t x & 5^2 e^t x^2 & 5 e^t x^3 & e^t x^4\\ 
x^4 & 5^4 e^t & 5^3 e^t x & 5^2 e^t x^2 & 5 e^t x^2 \\ 
5 x^3 & x^4 & 5^4 e^t & 5^3 e^t x & 5^2 e^t x^2 \\ 
5^2 x^2 & 5 x^3 & x^4 & 5^4 e^t & 5^3 e^t x \\ 
 5^3 x & 5^2 x^2 & 5 x^3 & x^4 & 5^4 e^t 
\end{pmatrix} 
\\  
\chnabla^{(\sigma)}_{\partial_x} & = \partial_x -  
\frac{1}{5^5 e^t-x^5}
\left(\mu - \frac{1}{2} - \sigma\right) 
\begin{pmatrix} 
x^4 & 5^4 e^t & 5^3 e^t x & 5^2 e^t x^2 & 5 e^t x^3 \\ 
5 x^3 & x^4 & 5^4 e^t & 5^3 e^t x & 5^2 e^t x^2 \\ 
5^2 x^2 & 5 x^3 & x^4 & 5^4 e^t & 5^3 e^t x \\ 
5^3 x & 5^2 x^2 & 5 x^3 & x^4 & 5^4 e^t \\ 
5^4 & 5^3 x & 5^2 x^2 & 5 x^3 & x^4 
\end{pmatrix}
\end{align*} 
The second structure connection has poles along
the divisor $\Sigma = \{5^5 e^t - x^5 =0\}$. 
We replace all the $\partial_x$-actions in the second column 
of \eqref{eq:afterFL_P4}
with the $\partial_t$-actions using the first column and deduce: 
\begin{align*} 
e^{-t} x\partial_t T_4  & = (5\partial_t + \sigma + \tfrac{5}{2})T_0 \\
x \partial_t T_0 & = (5 \partial_t + \sigma + \tfrac{3}{2})T_1 \\ 
x \partial_t T_1 &= (5\partial_t + \sigma + \tfrac{1}{2}) T_2 \\ 
x \partial_t T_2 & = (5 \partial_t + \sigma - \tfrac{1}{2})T_3 \\ 
x \partial_t T_3 & =(5 \partial_t + \sigma - \tfrac{3}{2})T_4
\end{align*} 
From this we find the following differential equation for $T_0$: 
\begin{align}
\label{eq:diffeq_T0}
\left( (x\partial_{t})^{5}-e^{t}(5\partial_{t}+\sigma+\tfrac{13}{2})
(5\partial_{t}+\sigma+\tfrac{11}{2})
(5\partial_{t}+\sigma+\tfrac{9}{2}) 
(5\partial_{t}+\sigma+\tfrac{7}{2}) 
(5\partial_{t}+\sigma+\tfrac{5}{2})\right)
T_{0}&=0. 
\end{align}
A direct computation on computer (we used Maple) shows: 
\begin{lem} 
\label{lem:2ndstrconn_P4}
Let $\check{F}$ denote the trivial $H^*(\pp^4)$-bundle 
over $\cc^2 = H^2(\pp^4) \times \cc_x$. 
Suppose that $\sigma \notin \{-\frac{3}{2}, -\frac{1}{2},\frac{1}{2}, 
\frac{3}{2}\}$. 
Then the second structure connection 
$(\cO(\check{F})(*\Sigma), \chnabla^{(\sigma)})$ is generated by 
$T_0=1$ as an $\cO(*\Sigma)\langle \partial_t \rangle$-module 
and is defined by the relation \eqref{eq:diffeq_T0}. 
\end{lem} 

\subsection{Euler-twisted and local (small) quantum $D$-modules}
Recall from Theorem \ref{thm:QS,with,I,Coates} 
that the second structure connection corresponds to 
the $(\e,K_{\pp^4}^{-1})$-twisted theory for $\sigma = \frac{5}{2}$ 
and to the local theory for $\sigma = -\frac{5}{2}$. 
For these cases, the differential equation \eqref{eq:diffeq_T0} 
specializes respectively to: 
\begin{align*} 
D_\eu & := (x\partial_{t})^{5}-e^{t}
(5\partial_{t}+9)(5\partial_{t}+8)(5\partial_{t}+7)
(5\partial_{t}+6)(5\partial_{t}+5)  
& & \text{(for $\sigma=\tfrac{5}{2}$),} 
\\ 
D_\loc & := (x\partial_{t})^{5}-e^{t} 
(5\partial_{t}+4) (5\partial_{t}+3)(5\partial_{t}+2)(5\partial_{t}+1)
(5\partial_{t}) 
& & \text{(for $\sigma=-\tfrac{5}{2}$).} 
\end{align*} 
The $I$-functions in Definition \ref{defi,I,functions} are given by 
\begin{align*} 
I^\eu_0(t,z) & = \sum_{d=0}^\infty e^{(d+H/z)t} 
\frac{\prod_{k=1}^{5d} (5 H + kz)}{\prod_{k=1}^d (H+kz)^5}, \\
I^\loc_0(t,z) & = \sum_{d=0}^\infty e^{(d+H/z)t} 
\frac{\prod_{k=0}^{5d-1} (-5H - kz)}{\prod_{k=1}^d (H+kz)^5}.  
\end{align*} 
The mirror maps \eqref{eq:mirrormaps} are given by 
\[
\Mir_\eu(t) = t+  \frac{g_1(e^t)}{g_0(e^t)}, 
\qquad 
\Mir_\loc(t) =  t + g_2(e^t)  
\]
where we set 
\[
g_0(e^t) = \sum_{d=0}^\infty e^{dt}\frac{(5d)!}{(d!)^{ 5}},  \quad 
g_1(e^t) = \sum_{d=1}^\infty e^{dt}\frac{(5d)!}{(d!)^{ 5}}5
\left(\sum_{m=d+1}^{5d}\frac{1}{m}\right), 
\quad 
g_2(e^t) = 5 \sum_{d=1}^\infty 
e^{dt} (-1)^d \frac{(5d-1)!}{(d!)^{5}}. 
\]
We define, as in \eqref{eq:map_pi}, 
\begin{align*} 
\pi_\eu(t,x) &= \Mir_\eu(t- 5 \log x) = 
t - 5 \log x + \frac{g_1(e^t x^{-5})}{g_0(e^t x^{-5})}, \\ 
\pi_\loc(t,x) &= \Mir_\loc(t - 5 \log x + 5 \pi \sqrt{-1}) 
= t - 5 \log x + 5\pi\sqrt{-1} + g_2(-e^t x^{-5}). 
\end{align*} 
These maps converge when $|e^t x^{-5}|<5^{-5}$. 
Theorem \ref{thm:QS,with,I,Coates} and 
Lemma \ref{lem:2ndstrconn_P4} together give 
the following isomorphisms:   
\begin{align*} 
\pi_\eu^* \SQDM_{(\e,K_{\pp^4}^{-1})}(\pp^4)\Big|_{z=1}  & \cong  
\left(\cO(\check{F}),\chnabla^{(\frac{5}{2})}\right) \cong 
\cO\langle \partial_t \rangle\big/  
\cO\langle \partial_t \rangle D_\eu, \\ 
\pi_\loc^* \SQDM(K_{\pp^4}) \Big|_{z=1} & \cong 
\left(\cO(\check{F}),\chnabla^{(-\frac{5}{2})}\right) \cong 
\cO\langle \partial_t \rangle \big/ 
\cO\langle \partial_t \rangle D_\loc
\end{align*} 
over the region $\{(t,x)\in \cc^2 : |e^t x^{-5}|<5^{-5}\}$. 
\subsection{The small quantum $D$-module of a quintic}
Recall from Theorem \ref{thm:euler,quantum,Serre,QDM} and 
\eqref{eq:morphism_rho} 
that we have a natural morphism: 
\[
5H  \colon \pi_\eu^* \SQDM_{(\e,K_{\pp^4}^{-1})}(\pp^4) 
\to \pi_\loc^* \SQDM(K_{\pp^4}) 
\]
By Theorem \ref{thm:pairing,vs,quantum,serre}, this  
corresponds to the map $\Delta$ between the second structure connections.  
Since $\Delta$ maps $T_0$ in $(\check{F},\chnabla^{(\frac{5}{2})})$ 
to $(-\partial_x)^5T_0 = e^{-t} \partial_t^5 T_0$ 
in $(\check{F},\chnabla^{(-\frac{5}{2})})$, 
the above morphism corresponds to the map: 
\[
\delta \colon \cO\langle \partial_t \rangle\big/\cO\langle \partial_t \rangle 
D_\eu \to 
\cO\langle \partial_t \rangle\big/\cO\langle \partial_t \rangle 
D_\loc, \qquad 
[f(t,x,\partial_t)] \mapsto [f(t,x,\partial_t) e^{-t} \partial_t^5].  
\]
This is well-defined since 
$D_\eu e^{-t} \partial_t^5 = \partial_t^5 e^{-t}D_\loc$. 
By Corollary \ref{cor:anticanonical_hypersurface}, we have 
\[  
\pi_\eu^* \SQDM_{\amb}(Z) \cong \im(\delta).  
\] 
for a quintic hypersurface $Z\subset \pp^4$. 
We can therefore view $\SQDM_{\amb}(Z)$ either as a quotient of the 
Euler-twisted quantum $D$-module or as a sub-$D$-module 
of the local quantum $D$-module. 
The former viewpoint yields a presentation: 
\[
\pi_\eu^* \SQDM_{\amb}(Z) \cong \cO\langle \partial_t \rangle\big/
\cO\langle \partial_t \rangle (x (x\partial_t)^4 - 5 e^t (5 \partial_t + 9) 
(5\partial_t + 8) (5 \partial_t + 7) ( 5 \partial_t +6)) 
\]
and the latter yields a (more familiar) presentation: 
\[
\pi_\eu^*\SQDM_{\amb}(Z) \cong 
\cO\langle \partial_t \rangle \big/ 
\cO\langle \partial_t \rangle (x(x\partial_t)^4 - 5 e^t (5 \partial_t +4) 
(5\partial_t + 3) (5\partial_t +2) (5\partial_t +1)). 
\]

\subsection{Solutions} 
\label{subsec:solutions}
For the Euler-twisted theory ($\sigma=\frac{5}{2}$), 
the cohomology-valued function
\begin{align*} 
\varphi(t,x) & = (-\partial_x )^{4}x^{-1}I^{\eu}_0(t-5\log x,1) 
 = \sum_{d=0}^\infty 
\frac{e^{t(H+d)}}{x^{5H+5d+5}}
\frac{\prod_{k=1}^{5d+4}(5H+k)}{\prod_{k=1}^{d} (H+k)^{5}} \\
\intertext{is a solution to the differential equation $D_\eu \varphi =0$; 
for the local theory ($\sigma = -\frac{5}{2}$), 
the cohomology-valued function} 
\varphi(t,x) &=  I^{\loc}_0(t-5\log x +5 \sqrt{-1}\pi, 1) = 
e^{5\pi \sqrt{-1}H}\sum_{d=0}^\infty 
\frac{e^{t(H+d)}}{x^{5H+5d}}\frac{\prod_{k=0}^{5d-1}(5H+k)}
{\prod_{k=1}^{d}(H+k)^{5}}
\end{align*} 
is a solution to the differential equation $D_\loc \varphi=0$. 
These functions are the images of $T_0$ respectively 
under the maps $\chK^{(\frac{5}{2},1)}$ and 
$e^{5\pi\sqrt{-1}H}\chK^{(-\frac{5}{2},0)}$ in Proposition 
\ref{prop:inverse,fundamental,solution,second,metric}.  
In terms of the quantum $D$-modules, these solutions 
correspond respectively to $L_\eu(\pi_\eu(t,x),1)^{-1}$ 
and $L_\loc(\pi_\loc(t,x),1)^{-1}$.

\subsection{Mirror maps and $f$}
Recall from Lemma \ref{lem:relation_mirrormaps} that the two 
mirror maps are related as follows:  
\begin{align*}
\Mir_\loc(t+5\pi\sqrt{-1}) 
& = \overline{f}(\Mir_\eu(t)) + 5 \pi \sqrt{-1} \\ 
\pi_\loc(t,x) & = \overline{f}(\pi_\eu(t,x)) + 5 \pi \sqrt{-1} 
\end{align*} 
where $\overline{f}$ is the map appearing in Lemma \ref{lem:equ,limit,f,inverse}: 
\[
\overline{f}(t)=t+\sum_{d=1}^\infty 
e^{dt} \langle H^{3},\widetilde{1} \rangle_{0,2,d}^{\eK}. 
\] 
Consider the exponentiated mirror maps and $\exp(\overline{f})$: 
\[
\frM_\eu(e^t) := \exp(\Mir_\eu(t)), 
\quad  
\frM_\loc(e^t) := \exp(\Mir_\loc(t)), \quad 
\overline{F}(e^t) := \exp(\overline{f}(t)). 
\]
These maps are related by $\frM_\loc(-q) = - \overline{F}(\frM_\eu(q))$. 
Surprisingly, they have Taylor expansions in $q=e^t$ 
with \emph{integral} coefficients \cite{Lian-Yau, Zhou:integrality_local}: 
\begin{align*} 
\frM_\eu(q) &= q + 770 q^2 + 1014275 q^3 + 1703916750 q^4 + 
3286569025625 q^5 + \cdots \\
\frM_\loc(q) & = q - 120 q^2 + 63900 q^3 - 63148000 q^4 + 
85136103750 q^5 + \cdots  \\ 
\overline{F}(q) & = q - 650 q^2 + 50625 q^3 - 
5377000 q^4 -49529975000 q^5 +\cdots.  
\end{align*}
We can also deduce the Gromov-Witten invariants $N_d := 
\langle H^{3},\widetilde{1} \rangle_{0,2,d}^{\eK}$ 
as in Table \ref{tab:GW_inv}. 
\begin{table}[htbp]
\begin{tabular}{|r|r|}
\hline 
$d$ & $N_d$\, \\ 
\hline 
1 & $-650$ \\ 
2 & $-160625$ \\ 
3 & $-337216250/3$ \\ 
4 & $- 217998840625/2$  \\ 
5 & $-125251505498880$ \\ 
6 & $-479299410776921825/3$ \\ 
7 & $-1531227197616745455000/7$ \\ 
8 & $-1260949629604284268280625/4$ \\
\hline
\end{tabular}  
\caption{Gromov-Witten Invariants 
$N_d = \langle H^{3},\widetilde{1} \rangle_{0,2,d}^{\eK}$}
\label{tab:GW_inv} 
\end{table}

\subsection{Hodge filtration} 
Recall from \S\ref{subsec:Hodge_filt} that we have Hodge 
filtrations on the second structure connections 
$(\check{F},\nabla^{(\frac{5}{2})})$ and 
$(\check{F},\nabla^{(-\frac{5}{2})})$ 
denoted respectively by $\check{F}^p_\eu$ and 
$\check{F}^p_\loc$. They are given by 
\begin{align*} 
\check{F}^0_\loc & = \check{F},  
& \check{F}_\eu^0 & = \check{F} 
\\ 
\check{F}^1_\loc & = \langle T_0, \partial_x T_0, \partial_x^2 T_0, 
\partial_x^3 T_0 \rangle, 
& \check{F}_\eu^1 & = (\check{F}_\loc^4)^\perp 
\\ 
\check{F}^2_\loc &= \langle T_0, \partial_x T_0, \partial_x^2 T_0 \rangle, 
& \check{F}_\eu^2 & = (\check{F}_\loc^3)^\perp 
\\ 
\check{F}^3_\loc & = \langle T_0, \partial_x T_0 \rangle, 
& \check{F}_\eu^3 & = (\check{F}_\loc^2)^\perp 
\\
\check{F}^4_\loc &= \langle T_0\rangle 
& \check{F}_\eu^4 & = (\check{F}_\loc^1)^\perp 
\end{align*} 
where $\perp$ means the orthogonal with respect to 
the second metric  $\check{g}(\gamma_1,\gamma_2) = 
\int_{\pp^4} \gamma_1 \cup (5H\bullet_t - x)^{-1}\gamma_2$ 
and $\partial_x$ means $\nabla^{(-\frac{5}{2})}_{\partial_x}$ 
in the first column. Using Maple, we find that 
\[
\check{F}^4_\eu  = \langle \tT_0 \rangle \quad 
\check{F}^3_\eu  = \langle \tT_0, \partial_x \tT_0\rangle, 
\quad 
\check{F}^2_\eu = \langle \tT_0, \partial_x \tT_0, \partial_x^2 \tT_0 
\rangle, \quad 
\check{F}^1_\eu = \langle \tT_0, \partial_x \tT_0, \partial_x^2 \tT_0, 
\partial_x^3 \tT_0 \rangle  
\]
where  $\partial_x =\nabla^{(\frac{5}{2})}_{\partial_x}$ and 
\begin{align*} 
\tT_0 & :=  T_0 - \tfrac{125}{3} x^{-1} T_1 + 
\tfrac{2125}{3}x^{-2} T_2 - 5625 x^{-3} T_3 + 15000 x^{-4}T_4, \\ 
\partial_x \tT_0 & = -5 x^{-1}T_0 + \tfrac{565}{3} x^{-2}T_1 
- \tfrac{8975}{3} x^{-3} T_2 + 
22875 x^{-4} T_3 - 60000 x^{-5} T_4, \\
\partial_x^2 \tT_0 & =  30 {x}^{-2} T_0  
-1030 {x}^{-3}T_1 + 15500 {x}^{-4}T_2 -115500 {x}^{-5} T_3 
+ 300000 {x}^{-6} T_4, \\ 
\partial_x^3 \tT_0 & = -210 {x}^{-3} T_0 +  
6610{x}^{-4}T_1 -95300 {x}^{-5}T_2 + 
697500 {x}^{-6}T_3 -1800000 {x}^{-7}T_4, \\
\partial_x^4 \tT_0 & = 
1680 {x}^{-4}T_0 -48680 {x}^{-5}T_1 + 
679000{x}^{-6}T_2 - 4905000 {x}^{-7}T_3 + 
12600000 {x}^{-8} T_4. 
\end{align*} 
One can check that $\tT_0$ corresponds to a multiple of the twisted $I$-function 
$I^\eu_0$ under the solution in \S \ref{subsec:solutions}: 
we have $\chK^{(\frac{5}{2},1)}(\tT_0) = 24 x^{-5} I_0^\eu(t-5\log x, 1)$. 
 
\bibliographystyle{alpha} 
\bibliography{biblio}

\end{document}